\documentclass[a4paper, 12pt, final]{amsart}
\usepackage{nag}
\usepackage{setspace}
\usepackage{multicol}
\usepackage{multirow}
\usepackage[indentafter]{titlesec}
\usepackage{upgreek}
\usepackage[pdftex]{graphicx}
\usepackage{amsmath, amsfonts, amssymb, amsthm, float, bm}
\usepackage{scalerel}
\usepackage{mathtools}
\usepackage{mathrsfs}
\usepackage{boldline}
\usepackage{enumitem}
\usepackage[toc,page]{appendix}
\usepackage[a4paper, bottom=2cm]{geometry}
\usepackage{array}
\usepackage{xtab}
\usepackage{lmodern}
\usepackage{tcolorbox}
\usepackage{xcolor}
\usepackage{url}
\usepackage{tikz-cd}
\usepackage{bm}
\usepackage{tabularx}
\usepackage{ltablex} 
\usepackage{etoolbox}
\usepackage{tipa}
\usepackage[backend=biber, style=alphabetic]{biblatex}
\usepackage[colorlinks=true,linkcolor=black,citecolor=blue!80!black]{hyperref}
\usepackage[capitalize]{cleveref}
\usepackage{relsize}
\usepackage{blindtext}
\usepackage{rotating}
\usepackage{fancyvrb}
\usepackage[justification=centering]{caption}

\renewbibmacro{in:}{}
\DeclareFieldFormat[article]{title}{#1}

\titleformat{name=\section}{}{\thetitle.}{0.8em}{\centering\scshape}
\titleformat{name=\subsection}[runin]{}{\thetitle.}{0.5em}{\bfseries}[.]
\titleformat{name=\subsubsection}[runin]{}{\thetitle.}{0.5em}{\itshape}[.]
\titleformat{name=\paragraph,numberless}[runin]{}{}{0em}{}[.]
\titlespacing{\paragraph}{0em}{0em}{0.5em}
\titleformat{name=\subparagraph,numberless}[runin]{}{}{0em}{}[.]
\titlespacing{\subparagraph}{0em}{0em}{0.5em}

\theoremstyle{plain}
\newtheorem{theorem}{Theorem}[section]
\newtheorem{proposition}[theorem]{Proposition}
\newtheorem{lemma}[theorem]{Lemma}

\newtheorem*{theorem*}{Theorem}

\newtheorem*{ThmA}{Theorem A}
\newtheorem*{ThmB}{Theorem B}
\newtheorem*{ThmC}{Theorem C}
\newtheorem*{ThmD}{Theorem D}
\newtheorem*{ThmE}{Theorem E}

\theoremstyle{definition}
\newtheorem{definition}[theorem]{Definition}
\newtheorem*{notation}{Notation}

\theoremstyle{remark}
\newtheorem*{remark}{Remark}

\AtBeginEnvironment{theorem}{\setlist[enumerate,1]{label=(\roman*),font=\upshape}}
\AtBeginEnvironment{mtheorem}{\setlist[enumerate,1]{label=(\roman*),font=\upshape}}
\AtBeginEnvironment{theorem*}{\setlist[enumerate,1]{label=(\roman*),font=\upshape}}
\AtBeginEnvironment{ftheorem}{\setlist[enumerate,1]{label=(\roman*),font=\upshape}}
\AtBeginEnvironment{example}{\setlist[enumerate,1]{label=(\roman*),font=\upshape}}
\AtBeginEnvironment{definition}{\setlist[enumerate,1]{label=(\roman*),font=\upshape}}
\AtBeginEnvironment{lemma}{\setlist[enumerate,1]{label=(\roman*),font=\upshape}}
\AtBeginEnvironment{proposition}{\setlist[enumerate,1]{label=(\roman*),font=\upshape}}
\AtBeginEnvironment{sublemma}{\setlist[enumerate,1]{label=(\roman*),font=\upshape}}
\AtBeginEnvironment{corollary}{\setlist[enumerate,1]{label=(\roman*),font=\upshape}}
\AtBeginEnvironment{theorem}{\setlist[enumerate,2]{label=(\alph*),font=\upshape}}
\AtBeginEnvironment{lemma}{\setlist[enumerate,2]{label=(\alph*),font=\upshape}}
\AtBeginEnvironment{proposition}{\setlist[enumerate,2]{label=(\alph*),font=\upshape}}
\AtBeginEnvironment{sublemma}{\setlist[enumerate,2]{label=(\alph*),font=\upshape}}

\renewcommand{\Gamma}{\varGamma}
\renewcommand{\epsilon}{\varepsilon}
\renewcommand{\bar}{\overline}
\renewcommand{\hat}{\widehat}
\renewcommand{\leq}{\leqslant}
\renewcommand{\geq}{\geqslant}

\newcommand{\normaleq}{\trianglelefteq}

\newcommand{\divides}{\bigm|}
\newcommand{\fs}{\mathcal{F}}

\newcommand{\N}{\mathbb{N}}

\newcommand{\SL}{\mathrm{SL}} 
\newcommand{\SU}{\mathrm{SU}} 
\newcommand{\GF}{\mathrm{GF}} 
\newcommand{\syl}{\mathrm{Syl}}
\newcommand{\GL}{\mathrm{GL}}
\newcommand{\Sp}{\mathrm{Sp}}

\newcommand{\PGL}{\mathrm{PGL}}
\newcommand{\PSL}{\mathrm{PSL}}

\newcommand{\PSU}{\mathrm{PSU}}

\newcommand{\Sz}{\mathrm{Sz}}

\newcommand{\Sym}{\mathrm{Sym}} 
\newcommand{\Alt}{\mathrm{Alt}} 
\newcommand{\Dih}{\mathrm{Dih}}

\newcommand{\Aut}{\mathrm{Aut}}
\newcommand{\Out}{\mathrm{Out}}
\newcommand{\Inn}{\mathrm{Inn}}

\newcommand{\Hom}{\mathrm{Hom}}

\newcommand{\sbt}{\,\begin{picture}(-1,1)(-1,-4)\circle*{2}\end{picture}\ }
\def \wt {\widetilde}
\begin{document}

\title[Exotic Fusion Systems Related to Sporadic Simple Groups]{Exotic Fusion Systems Related to Sporadic Simple Groups (version with code)}
\author{Martin van Beek}
\thanks{Part of this work contributed to the author's PhD thesis at the University of Birmingham under the supervision of Prof. Chris Parker. The author gratefully acknowledges the support received from the EPSRC (EP/N509590/1) during this period.}

\begin{abstract}
We describe several exotic fusion systems related to the sporadic simple groups at odd primes. More generally, we classify saturated fusion systems supported on Sylow $3$-subgroups of the Conway group $\mathrm{Co}_1$ and the Thompson group $\mathrm{F}_3$, and a Sylow $5$-subgroup of the Monster $\mathrm{M}$, as well as a particular maximal subgroup of the latter two $p$-groups. This work is supported by computations in MAGMA.
\end{abstract}

\maketitle

\section{Introduction}
A fusion system over a finite $p$-group $S$ is a category satisfying certain conditions modeled on properties of finite groups and the internal actions associated to their Sylow $p$-subgroups. The typical example of a fusion system arises just like this: as the \emph{$p$-fusion category} of a finite group. In this case, certain additional conditions are satisfied which may be abstracted as additional axioms, defining the class of \emph{saturated} fusion systems. However, not all saturated fusion systems can be realized as the $p$-fusion category of some finite group, giving rise to \emph{exotic} fusion systems.  

Over the course of this work, we completely classify all saturated fusion systems supported on Sylow $3$-subgroups of the Conway group $\mathrm{Co}_1$ and the Thompson group $\mathrm{F}_3$, and a Sylow $5$-subgroup of the Monster $\mathrm{M}$. In addition, we also classify saturated fusion systems supported on a particular maximal subgroup of a Sylow $3$-subgroup of $\mathrm{F}_3$, and of a Sylow $5$-subgroup of $\mathrm{M}$. Of particular interest in this determination is the occurrence of several exotic fusion systems supported on these $p$-groups. In total we uncover sixteen new exotic systems up to isomorphism, seven of which are simple, giving a rich source of reasonably complicated examples. 

We have not yet considered the implications of these new exotic fusion systems to any of the areas in which fusion systems have application (see \cite{OliAsch} for a survey), and have studied them purely for their interesting structural properties, and for their appearance in other ongoing classification programs concerning fusion systems. Since exotic fusion systems themselves are still poorly understood, at this moment a considerable amount of attention is just focused on determining new families of examples with the ultimate goal of discerning exotic fusion systems from those occurring as $p$-fusion categories of finite groups, without having to rely on heavy machinery from finite group theory.

Our first main result is the following, and is proved via \cref{Co1Essens}, \cref{Sp63} and \cref{Co1}:

\begin{ThmA}\hypertarget{ThmA}{}
Let $\fs$ be a saturated fusion system on a $3$-group $S$ with $S$ isomorphic to a Sylow $3$-subgroup of $\mathrm{Co}_1$. If $O_3(\fs)=\{1\}$ then $\fs$ is isomorphic to the $3$-fusion category of $\mathrm{Co}_1$, $\Sp_6(3)$ or $\Aut(\Sp_6(3))$.
\end{ThmA}

We point out that the $3$-fusion system of $\mathrm{Co}_1$ has been identified by work of Oliver \cite[Theorem A]{BobTodd} but from a different starting point than what is considered in this paper. We remark that the proofs of \cite[Theorem A]{BobTodd} and \hyperlink{ThmA}{Theorem A} do not depend on each other, however they both reduce to a situation where one has strong information about the local actions in the fusion system. At this point, either paper could use the other's result but yet again, different (and complementary) approaches are taken to prove the uniqueness of the fusion system of $\mathrm{Co}_1$.

We now move on to the construction of some exotic fusion systems. We use the same methodology to prove \hypertarget{ThmB}{Theorem B} and \hypertarget{ThmC}{Theorem C}, although the arguments vary slightly depending on the structure of the underlying $p$-group $S$. The author first encountered the systems in \hypertarget{ThmB}{Theorem B} while classifying certain fusion systems which contain only two essential subgroups \cite{MainThm}. These systems arise as a fusion theoretic generalization of weak BN-pairs of rank $2$, a collection of amalgams classified by work of Delgado and Stellmacher \cite{Greenbook}. In \hypertarget{ThmB}{Theorem B} one of the exotic systems we uncover arises as a fusion system ``completion" of an amalgam of $\mathrm{F}_3$-type, as defined in \cite{Greenbook}. In the case of the group $\mathrm{F}_3$, the corresponding amalgam generates the entire group. This is in contrast to the fusion system case, where the $3$-fusion category of $\mathrm{F}_3$ requires another set of $3$-local actions, corresponding to the maximal subgroups of $\mathrm{F}_3$ of shape $3^5:2.\Sym(6)$, to be properly generated.

\begin{ThmB}\hypertarget{ThmB}{}
Let $\fs$ be a saturated fusion system on a $3$-group $S$ with $S$ isomorphic to a Sylow $3$-subgroup of $\mathrm{F}_3$. If $O_3(\fs)=\{1\}$ then either $\fs$ is isomorphic to the $3$-fusion category of $\mathrm{F}_3$; or $\fs$ is isomorphic to one of two exotic examples. In all cases, $\fs$ is simple.
\end{ThmB}

\begin{ThmC}\hypertarget{ThmC}{}
Let $\fs$ be a saturated fusion system on a $5$-group $S$ with $S$ isomorphic to a Sylow $5$-subgroup of $\mathrm{M}$. If $O_5(\fs)=\{1\}$ then either $\fs$ is isomorphic to the $5$-fusion category of $\mathrm{M}$; or $\fs$ is isomorphic to one of two exotic examples. In all cases, $\fs$ is simple.
\end{ThmC}

\hyperlink{ThmB}{Theorem B} is proved as \cref{ThmB1} while \hyperlink{ThmC}{Theorem C} is proved as \cref{5nontriv}, \cref{ThmC1} and \cref{ThmC2}. We note that the process by which we construct some of the systems in \hyperlink{ThmB}{Theorem B} and \hyperlink{ThmC}{Theorem C} can also be applied to the $2$-fusion category of $\mathrm{J}_3$. In this application, one obtains three proper saturated subsystems, all of which contain no non-trivial normal $2$-subgroups. However, unlike the odd prime cases, the subsystems recovered are realizable by finite groups. Indeed, these subsystems are isomorphic to the $2$-fusion categories of $\mathrm{J}_2$, $\PSL_3(4):2$ and $\PGL_3(4):2$, as demonstrated in \cite[Theorem 4.8]{OliMod}.

Interestingly, we record that some of the exotic fusion systems described in \hyperlink{ThmB}{Theorem B} and \hyperlink{ThmC}{Theorem C} contain a unique proper non-trivial strongly closed subgroup which does not support a normal fusion subsystem. In both cases, this strongly closed subgroup is the centralizer of the second center of the Sylow $p$-subgroup, and is also essential in the fusion system. This is another instance where fusion systems seem to depart from the conventions of finite simple groups. As witnessed in \cite[Corollary 1.4]{Flores}, if $G$ is a finite simple group with a non-trivial strongly closed subgroup $A$ then $N_G(A)$ controls strong $G$-fusion in $S\in\syl_p(G)$ and so $\fs_S(G)$ is not simple in this instance.

Where we have a proper non-trivial strongly closed subgroup $T$, we are able to descend to exotic subsystems supported on $T$, and we speculate that this may be an illustration of a more generic method to construct exotic subsystems of exotic fusion systems. The examples we obtain in the theorems below arise in this fashion.

\begin{ThmD}\hypertarget{ThmD}{}
Let $S$ be isomorphic to a Sylow $3$-subgroup of $\mathrm{F}_3$. Then, up to isomorphism, there are two saturated fusion system supported on $C_S(Z_2(S))$ in which $C_S(Z_2(S))$ is not normal. Both of these systems are exotic and only one is simple.
\end{ThmD}

\begin{ThmE}\hypertarget{ThmE}{}
Let $S$ be isomorphic to a Sylow $5$-subgroup of $\mathrm{M}$. Then, up to isomorphism, there are nine saturated fusion system supported on $C_S(Z_2(S))$ in which $C_S(Z_2(S))$ is not normal. All of these systems are exotic and two are simple.
\end{ThmE}

The exotic fusion systems described in \hyperlink{ThmD}{Theorem D} and \hyperlink{ThmE}{Theorem E} are reminiscent of the exotic fusion systems supported on $p$-groups of maximal class, as determined in \cite{ParkerMax}. There, in almost all cases where $\fs$ is an exotic fusion systems, there is a class of essential subgroups which are \emph{pearls}: essential subgroups isomorphic to $p^2$ or $p^{1+2}_+$. It is clear that for a fusion system $\fs$ with a pearl $P$, $O^{p'}(\Out_{\fs}(P))\cong \SL_2(p)$ and so these occurrences are strongly connected to certain pushing up configurations in local group theory. In our case, the analogous set of essential subgroups $P$ are of the form $p^4\times C_P(O^{p'}(\Out_{\fs}(P)))$ where $O^{p'}(\Out_{\fs}(P))\cong \SL_2(p^2)$, and in one of our cases $C_P(O^{p'}(\Out_{\fs}(P)))$ is non-trivial.  We speculate that both the systems containing pearls and our examples are part of a much larger class of exotic fusion systems which arise as the odd prime counterparts to ``obstructions to pushing up" in the sense of Aschbacher \cite{asch1}. A clear understanding of this would go some way to explaining the dearth of exotic fusion systems at the prime $2$.

With this work, we move closer to classifying all saturated fusion systems supported on Sylow $p$-subgroups of the sporadic simple groups, for $p$ an odd prime, complementing several other results in the literature. Indeed, all that remains is the study of saturated fusion systems on Sylow $3$-subgroups of the Fischer groups, the Baby Monster and the Monster. For the reader's convenience, we tabulate the known results with regards to fusion systems on Sylow $p$-subgroups of sporadic simple groups in \cref{tbl : SporTable}.

We remark that, perhaps aside from the Sylow $3$-subgroup of $\mathrm{Fi}_{22}$, the remaining cases are large and complex enough that it is laborious and computationally expensive to verify any results using the fusion systems package in MAGMA \cite{MAGMA}, \cite{Comp1}. Throughout this work, we lean on a small portion of these algorithms for the determination of the essentials subgroups of the saturated fusion systems under investigation (as in \cref{EssenDeter}), although the techniques used in \cite{parkersem} and \cite{G2pPaper} could be employed here instead. We record that several of the main theorems have been verified using the full potential of this MAGMA package. However, we believe it is important to provide handwritten arguments in order to exemplify some of the more interesting structural properties of the fusion systems described within, while simultaneously elucidating some of the computations performed implicitly by the MAGMA package. For the sake of brevity, the MAGMA code we use is not included here and has instead been relegated to an alternate version of this paper \cref{code}.

Our notation and terminology for finite groups is a jumble of conventions from \cite{asch2}, \cite{gor} and \cite{kurz}, and we hope that our usage will be clear from context. With regards to notation concerning the sporadic simple groups, we will generally follow the Atlas \cite{atlas} with the exception of Thompson's sporadic simple group, which we refer to as $\mathrm{F}_3$ instead of the usual $\mathrm{Th}$, except in \cref{tbl : SporTable}. We make this choice to emphasize the connection with ``amalgams of type $\mathrm{F}_3$" as defined in \cite{Greenbook} and \cite{F3}. For fusion systems, we almost entirely follow the conventions in \cite{ako}.

\begin{table}[H]
    \caption{Fusion systems on non-abelian Sylow $p$-subgroups of sporadic groups for $p$ odd.}
        \begin{tabular}{|c|c|c|c|}\hline
            Simple Group & $|S|$ & Reference & \#Exotic Systems \\
             & & & Supported \\\hline
            ${}^2\mathrm{F}_{4}(2)'$, $\mathrm{J}_{2}$, $\mathrm{J}_{4}$, $\mathrm{M}_{12}$, $\mathrm{M}_{24}$, $\mathrm{Ru}$, $\mathrm{He}$ & $3^3$ & \cite{RV1+2} & 0 \\\hline
            $\mathrm{J}_{3}$ & $3^5$ & \cite{Comp1} & 0\\\hline
            $\mathrm{Co}_{1}$ & $3^9$ & \cref{Co1Sec} & 0 \\\hline
            $\mathrm{Co}_{2}$, $\mathrm{McL}$ & $3^6$ &\cite{Baccanelli} & 0 \\\hline
            $\mathrm{Co}_{3}$ & $3^7$ & \cite{Comp1} & 0 \\\hline
            $\mathrm{Fi}_{22}$ & $3^9$ & - & Open \\\hline
            $\mathrm{Fi}_{23}, \mathrm{B}$ & $3^{13}$ & - & Open \\\hline
            $\mathrm{Fi}_{24}'$ & $3^{16}$ & - & Open \\\hline
            $\mathrm{Suz}$, $\mathrm{Ly}$ & $3^7$ & \cite{Comp1} & 0 \\\hline
            $\mathrm{HN}$ & $3^6$ &\cite{Comp1} & 0 \\\hline
            $\mathrm{Th}$ & $3^{10}$ & \cref{F3Sec} & $2$ \\\hline
            $\mathrm{M}$ & $3^{20}$ & - & Open \\\hline
            $\mathrm{Co}_{1}$ & $5^4$ & \cite{Sp4} & 24 \\\hline
            $\mathrm{Co}_{2}$, $\mathrm{Co}_{3}$, $\mathrm{Th}$, $\mathrm{HS}$, $\mathrm{McL}$, $\mathrm{Ru}$ & $5^3$ &\cite{RV1+2} & 0 \\\hline
            $\mathrm{HN}, \mathrm{Ly}, \mathrm{B}$ & $5^6$ &\cite{parkersem} & 0 \\\hline
            $\mathrm{M}$ & $5^9$ & \cref{MonSec} & 4 \\\hline
            $\mathrm{Fi}_{24}'$, $\mathrm{He}$, $\mathrm{O'N}$ & $7^3$ &\cite{RV1+2} & 3 \\\hline
            $\mathrm{M}$ & $7^6$ &\cite{parkersem} & 27 \\\hline
            $\mathrm{J}_4$ & $11^3$ & \cite{RV1+2} & 0 \\\hline
            $\mathrm{M}$ & $13^3$ &\cite{RV1+2} & 0\\\hline
        \end{tabular}
    \label{tbl : SporTable}
\end{table}

\section{Preliminaries: Groups}\label{GrpSec}

We start with some elementary observations regarding the Thompson subgroup of a finite $p$-group and the related notion of failure to factorize modules. For a more in depth account of this phenomena, see \cite[Section 9.2]{kurz}.

\begin{definition}\label[definition]{thomp}
Let $S$ be a finite $p$-group. Set $\mathcal{A}(S)$ to be the set of all elementary abelian subgroup of $S$ of maximal rank. Then the \emph{Thompson subgroup} of $S$ is defined as $J(S):=\langle A \mid A\in\mathcal{A}(S)\rangle$.
\end{definition}

\begin{proposition}\label[proposition]{BasicJS}
Let $S$ be a non-trivial finite $p$-group. Then the following hold:
\begin{enumerate}
\item $J(S)$ is a non-trivial characteristic subgroup of $S$;
\item $\Omega_1(C_S(J(S)))=\Omega_1(Z(J(S)))=\bigcap_{A\in\mathcal{A}(S)} A$; and
\item if $J(S)\le T\le S$, then $J(S)=J(T)$.
\end{enumerate}
\end{proposition}
\begin{proof}
See \cite[{{9.2.8}}]{kurz} for parts (i) and (iii). Additionally, by part (d) of that result, we see that $\Omega_1(C_S(J(S)))\le \Omega_1(Z(J(S)))$. Since $Z(J(S))\le C_S(J(S))$, it is clear that $\Omega_1(C_S(J(S)))=\Omega_1(Z(J(S)))$. 

Let $a\in \bigcap_{A\in\mathcal{A}(S)} A$. Then $a$ has order $p$ and $[a, A]=\{1\}$ for all $A\in\mathcal{A}(S)$. By definition, $[a, J(S)]=\{1\}$ so that $a\in \Omega_1(C_S(J(S)))$ and $\bigcap_{A\in\mathcal{A}(S)} A\le \Omega_1(C_S(J(S)))$. Now, for $x\in C_S(J(S))$ of order $p$, we have that $x\le C_S(J(S))\le C_S(A)$ for all $A\in\mathcal{A}(S)$. Hence, $x\in \Omega_1(C_S(A))$ for all $A\in\mathcal{A}(S)$. But now, $\langle x\rangle A$ is elementary abelian of order at least as large as $A$ and by the definition of $\mathcal{A}(S)$, we have that $x\in A$. Therefore, $x\in \bigcap_{A\in\mathcal{A}(S)} A$ and $\Omega_1(C_S(J(S)))=\bigcap_{A\in\mathcal{A}(S)} A$, completing the proof of (ii).
\end{proof}

\begin{definition}
Let $G$ be a finite group, $V$ a $\GF(p)G$-module and $A\le G$. If 
\begin{enumerate}
\item $A/C_A(V)$ is an elementary abelian $p$-group;
\item $[V,A]\ne\{1\}$; and 
\item $|V/C_V(A)|\leq |A/C_A(V)|$
\end{enumerate}
then $V$ is a \emph{failure to factorize module} (abbrev. FF-module) for $G$ and $A$ is an \emph{offender} on $V$. 
\end{definition}

We will also make liberal use of several coprime action results, often without explicit reference. 

\begin{proposition}[Coprime Action]
Suppose that a group $G$ acts on a group $A$ coprimely, and $B$ is a $G$-invariant subgroup of $A$. Then the following hold:
\begin{enumerate}
\item $C_{A/B}(G)=C_A(G)B/B$;
\item $[A, G]=[A,G,G]$;
\item $A=[A,G]C_A(G)$ and if $A$ is abelian then $A=[A,G]\times C_A(G)$; and
\item if $G$ acts trivially on $A/\Phi(A)$, then $G$ acts trivially on $A$.
\end{enumerate}
\end{proposition}
\begin{proof}
See, for instance, \cite[Chapter 8]{kurz}.
\end{proof}

In conclusion (iv) in the statement above, one can say a little more. The following is a classical result of Burnside, but the version we use follows from \cite[(I.5.1.4)]{gor}.

\begin{lemma}[Burnside]\label[lemma]{burnside}
Let $S$ be a finite $p$-group. Then $C_{\Aut(S)}(S/\Phi(S))$ is a normal $p$-subgroup of $\Aut(S)$.
\end{lemma}

\begin{lemma}\label[lemma]{GrpChain}
Let $E$ be a finite $p$-group and $Q\le A\le \Aut(E)$. Suppose there exists a normal chain $\{1\} =E_0 \normaleq E_1  \normaleq E_2 \normaleq \dots \normaleq E_m = E$ of subgroups such that for each $\alpha \in A$, $E_i\alpha = E_i$ for all $0 \le i \le m$. If for all $1\le i\le m$, $Q$ centralizes $E_i/E_{i-1}$, then $Q\le O_p(A)$.
\end{lemma}
\begin{proof}
See \cite[{(I.5.3.3)}]{gor}.
\end{proof}

\begin{lemma}[A$\times$B-Lemma]
Let $AB$ be a finite group which acts on a $p$-group $V$. Suppose that $B$ is a $p$-group, $A=O^p(A)$ and $[A,B]=\{1\}=[A, C_V(B)]$. Then $[A,V]=\{1\}$.
\end{lemma}
\begin{proof}
See \cite[(24.2)]{asch2}.
\end{proof}

For the following, we refer to \cite[pg 100]{Greenbook} and \cite[pg 150 (a)-(f)]{F3} for the notion of an amalgam of type $\mathrm{F}_3$, noting that such amalgams are unique up to \emph{parabolic isomorphism}. We provide the following result, which appears to have evaded the literature up until this point.

\begin{proposition}\label[proposition]{F3Unique}
Let $\mathcal{A}$ be an amalgam of type $\mathrm{F}_3$. Then $\mathcal{A}$ is unique up to isomorphism.
\end{proposition}

For this, we apply the computer implementation of Goldschmidt's lemma \cite[(2.7)]{goldschmidt} found in Cano's PhD Thesis \cite[pg 34]{cano} (mirrored in \cref{code}) in MAGMA. This takes as input four groups: $P_1$, $B_1$, $P_2$, $B_2$. It then outputs a $4$-tuple, of which the first entry is the one we are interested in. We appeal to the online version of the Atlas of Finite Group Representations \cite{OnlineAtlas} for a matrix representation of the group $\mathrm{F}_3$, namely its $248$-dimensional representation over $\GF(2)$. We then use \cite{OnlineAtlas} to obtain the matrices which generate two distinct maximal subgroups of $\mathrm{F}_3$ which contain a Sylow $3$-subgroup. These groups represent $P_1$ and $P_2$ in our case. 

By the main result of \cite{F3}, the parabolic subgroups defining an amalgam of type $\mathrm{F}_3$ are unique up to isomorphism and so, the groups $P_1$ and $P_2$ have the isomorphism type of the parabolic groups in any $\mathrm{F}_3$-type amalgam. Hence, we are justified in our choice of subgroups to take. Then the groups $B_i$ are defined as $N_{P_i}(S_i)$, where $S_i$ is any Sylow $3$-subgroup of $P_i$, for $i\in\{1,2\}$. The function then outputs $1$ as its first entry, and so the amalgam is unique.

Our final results in this section with regards to groups and modules concerns the identification of some local actions within the groups $\mathrm{Co}_1$, $\Sp_6(3)$ and $\mathrm{M}$.

\begin{lemma}\label[lemma]{36Iden}
Suppose that $G$ is a finite group with $O_3(G)=\{1\}$ and $V$ is a faithful $\mathrm{GF}(3)G$-module of dimension $6$. Assume that for $S\in\syl_3(G)$, $S\cong 3^{1+2}_+$, $G=O^{3'}(G)$ and there is an elementary abelian subgroup $A\le S$ of order $9$ with $|V/C_V(A)|=|V/C_V(a)|=3^3$ for all $a\in A^{\#}$. Then $G\cong \PSL_3(3)$ or $2.\mathrm{M}_{12}$.
\end{lemma}
\begin{proof}
Let $G$ be a minimal counterexample with respect to $|G|$. By \cite[8.3.4(a)]{kurz}, $O_{3'}(G)=\langle C_{O_{3'}(G)}(a)\mid a\in A^\#\rangle$ and since $C_V(a)=C_V(A)$ for all $a\in A^\#$, we have that $O_{3'}(G)$ normalizes $C_V(A)$. Set $T:=\langle A^{AO_{3'}(G)}\rangle$ so that $C_V(A)=C_V(T)\le C_V(O_{3'}(T))$. By coprime action again, $V=[V, O_{3'}(T)]\times C_V(O_{3'}(T))$. But now, 
\[ C_{[V, O_{3'}(T)]}(A)\le [V, O_{3'}(T)]\cap C_V(A)=[V, O_{3'}(T)]\cap C_V(T)=\{1\}\]
and as $A$ is a $3$-group, we must have that $[V, O_{3'}(T)]=\{1\}$. Since $G$ acts faithfully on $V$, we infer that $O_{3'}(T)=\{1\}$. Then, as $A\le T \le AO_{3'}(G)$ and $T\cap O_{3'}(G)\le O_{3'}(T)=\{1\}$, we conclude that $A=T$ is normalized by $O_{3'}(G)$. In particular, $[A, O_{3'}(G)]=\{1\}$.

Since $O_3(G)=\{1\}$ and $F^*(G)$ is self-centralizing in $G$, we have shown that $G$ contains a component, $K$ say, whose order is divisible by $3$. Then $E:=\langle K^G\rangle$ is normalized by $S$ and so we deduce that it contains $Z(S)$. Note that since $m_3(S)=2$ and $O_3(E)\le O_3(G)=\{1\}$, $E$ contains at most two components of $G$ whose orders are divisible by $3$. Indeed, since $S$ is a $3$-group, we see that $S$ normalizes these components. If $E$ contains exactly two components of $G$ whose orders are divisible by $3$, $K_1$ and $K_2$ say, then $K_i\cap S\normaleq S$ for $i\in\{1,2\}$ so that $Z(S)\le K_1\cap K_2\le Z(E)$. Hence, $Z(S)\le O_3(Z(E))\le O_3(G)$, a contradiction. 

Thus, $E=K$ is quasisimple. Now, $K=\langle Z(S)^K\rangle=\langle Z(S)^G\rangle$ and so $K$ is a component of $H:=\langle Z(S)^G\rangle S$ so that $H=O^{3'}(H)$ is almost-quasisimple. Note that $O_3(H)$ is trivial for otherwise $Z(S)\le O_3(H)\cap K\le O_3(K)\le O_3(G)$ and since $O_3(G)=\{1\}$, this is a contradiction. Hence, by minimality, either $H\cong \PSL_3(3)$ or $2.\mathrm{M}_{12}$; or $G=H$. In the former case, we deduce that $H=K\normaleq G$ and since $S\le H$ and $G=O^{3'}(G)$, we have that $G=H$.

Hence, a minimal counterexample of this lemma is almost quasisimple. Now, $|A|^2=3^4>3^3=|V/C_V(A)|$ so that $V$ is a \emph{$2$F-module} for $G$ with \emph{offender} $A$ in the language of \cite{2F3}. By \cite[Table 1]{2F3}, $G$ is isomorphic to either a group of Lie type in characteristic $3$ or $2.\mathrm{M}_{12}$. The groups of Lie type in characteristic $3$ with Sylow $3$-subgroup isomorphic to $S$ are well known (see \cite[(3.3)]{GLS3}), and so we have that $G\cong 2.\mathrm{M}_{12}$, $\PSL_3(3)$ or $\SU_3(3)$. Now, $\SU_3(3)$ has only one non-trivial module of dimension $6$ over $\GF(3)$, the \emph{natural module}. But for this module, we have that $|C_V(B)|=3^2$ for any subgroup $B$ of the Sylow $3$-subgroup which has order $9$.
\end{proof}

In the following proposition, MAGMA is used to verify certain calculations. The actual code itself may be found in \cref{code}.

\begin{lemma}\label[lemma]{J2Iden}
Suppose that $Q\cong 5^{1+6}_+$, $G\le \Out(Q)$ and write $V=Q/Z(Q)$ and $S\in\syl_5(G)$. Suppose the following hold:
\begin{enumerate}
    \item $S$ is elementary abelian of order $25$;
    \item $G=\langle S^G\rangle$;
    \item $O_5(G)=\{1\}$; and
    \item $|C_V(S)|=5$ and $|C_V(s)|=25$ for all $s\in S^{\#}$.
\end{enumerate}
Then $G\cong 2.\mathrm{J}_2$.
\end{lemma}
\begin{proof}
Since $Q$ is extraspecial, $O_5(G)=\{1\}$ and $G=O^{5'}(G)$, applying \cite{Winter} we have that $G$ is isomorphic to a subgroup of $\Sp_6(5)$ and $Q/Z(Q)$ may be identified with the natural module for $\Sp_6(5)$ in this action. We appeal to \cite[Table 8.28, Table 8.29]{LowMax} for the list of maximal subgroups of $\Sp_6(5)$. These are 
\[2.\mathrm{J}_2,\, \Sp_2(5)\circ \mathrm{GO}_3(5),\, \mathrm{GU}_3(5).2, \Sp_2(5^3).3,\, \Sp_2(5)^3:\Sym(3),\, \Sp_2(5)\times \Sp_4(5),\]
\[5^6:\GL_3(5),\, 5^{3+4}:\GL_2(5)\times \Sp_2(5)\,\,\text{and}\,\,5^{1+4}_+:C_4\times \Sp_4(5).\] 
Aiming for a contradiction, assume throughout that $G\not\cong 2.\mathrm{J}_2$.

We compute that the maximal subgroups in which a Sylow $5$-subgroup fixes a subspace of dimension $1$ are $2.\mathrm{J}_2$, $\Sp_2(5)\circ \mathrm{GO}_3(5)$, $5^6:\GL_3(5)$, $5^{3+4}:\GL_2(5)\times \Sp_2(5)$ and $5^{1+4}_+:C_4\times \Sp_4(5)$. We refer to these subgroups as $M_1,\dots, M_5$ respectively. In $M_2$, one can compute that there is a $5$-element which fixes a subspace of dimension $3$ and as a Sylow $5$-subgroup of $M_1$ has order $25$, $G$ cannot be isomorphic to a subgroup of $M_1$. If $G$ is isomorphic to a subgroup of $M_3$, then as $O_5(G)=\{1\}$, $G$ projects as a subgroup of $\GL_3(5)$. But every subgroup of $\GL_3(5)$ which has a Sylow $5$-subgroup of order $25$ has a normal $5$-subgroup, a contradiction.

Similarly, if $G$ is isomorphic to a subgroup of $M_4$, then $G$ is isomorphic to a subgroup of $\GL_2(5)\times \Sp_2(5)$. Indeed, since $G:=\langle S^G\rangle$, $|S|=25$ and $O_5(G)=\{1\}$, it follows that $G\cong \SL_2(5)\times \Sp_2(5)$. Hence, $GO_5(M_4)=O^{5'}(M_4)$. Let $L\le G$ be such that $L\normaleq G$ and $L\cong \SL_2(5)\cong \Sp_2(5)$. Then $L$ contains a Sylow $2$-subgroup $T$ of $LO_5(M_4)\normaleq G$. By a calculation, we have that $C_{GO_5(M_4)}(T)\cong 2\times \SL_2(5)$. Since $C_G(T)\cong 2\times \SL_2(5)$, we have that $C_{GO_5(M_4)}(T)=C_G(T)$. However, for $R\in\syl_5(C_{GO_5(M_4)}(T))$, we have that $|C_V(R)|=5^5$, a clear contradiction.

If $G$ is isomorphic to a subgroup of $M_5$ then $G$ is isomorphic to a subgroup of $C_4\times \Sp_4(5)$. Since $G=\langle S^G\rangle$, we see that $G$ is isomorphic to a subgroup of $\Sp_4(5)$. Using MAGMA, since $O_5(G)=\{1\}$, $|S|=25$ and $G=\langle S^G\rangle$, we calculate that $G\cong \SL_2(25)$ or $\SL_2(5)\times \SL_2(5)$. Moreover, the center of $\Sp_4(5)$ is equal to the center of a Sylow $2$-subgroup of $\Sp_4(5)$ and it follows from computations that $G$ centralizes the center of a Sylow $2$-subgroup of $L5:=O^{5'}(M_5)$, which we denote by $T$. Then $G=G'\le C_{L5}(T)'\cong \Sp_4(5)$ and so $G$ is contained in a specified complement to $O_5(M_5)$ in $L_5$. But then we calculate for all such candidates for $G$ that $|C_V(S)|=5^4$, a contradiction. 

Hence, $G$ is isomorphic to a proper subgroup of $M_1\cong 2\mathrm{J}_2$. But, appealing to \cite{atlas} for a list of maximal subgroups of $\mathrm{J}_2$, the only maximal subgroups of $2.\mathrm{J}_2$ which have a Sylow $5$-subgroup of order $25$ also have a normal $5$-subgroup, a contradiction. 
\end{proof}

\section{Preliminaries: Fusion Systems}\label{FusSec}

We now let $S$ be a finite $p$-group and $\fs$ a saturated fusion system on $S$, referring to \cite{ako} and \cite{craven} for standard terminology and results regarding fusion systems. We use the remainder of this section to reaffirm some important concepts regarding fusion systems which pertain to this work, and mention some vital results from other sources in the literature.

We begin with the notion of isomorphism for fusion systems.

\begin{definition}
Let $\fs$ be a saturated fusion system on a $p$-group $S$ and let $\alpha: S\to T$ be a group isomorphism. Define $\fs^\alpha$ to be the fusion system on $T$ with \[\Hom_{\fs^\alpha}(P, Q)=\{\alpha^{-1}\gamma \alpha \mid \gamma\in\Hom_{\fs}(P\alpha^{-1}, Q\alpha^{-1}\}\] for $P,Q\le T$.

We then say that a fusion system $\mathcal{E}$ over a $p$-group $T$ is \emph{isomorphic} to $\fs$, written $\mathcal{E}\cong \fs$, if there a group isomorphism $\alpha: S\to T$ with $\mathcal{E}=\fs^\alpha$.
\end{definition}

\begin{remark}
We note that our definition of isomorphism coincides with morphisms defined in \cite[Definition II.2.2]{ako} which are surjective and have trivial kernel.
\end{remark}

Importantly, for $G$ a finite group, $S\in\syl_p(G)$ and $K$ a normal $p'$-subgroup of $G$, writing $\bar{G}:=G/K$, we have that $\fs_S(G)\cong \fs_{\bar{S}}(\bar{G})$. This is often viewed as one of the main attractions for working with fusion systems in place of finite groups. 

We recall that $\fs$ is \emph{realizable} if there is a finite group $G$ and $S\in\syl_p(G)$ such that $\fs=\fs_S(G)$, and $\fs$ is \emph{exotic} otherwise. By the above observation, if we aim to show that $\fs$ is realized by a finite group $G$, then we may as well assume that $O_{p'}(G)=\{1\}$.

\begin{notation}
Let $\fs$ be a fusion system and let $\fs_1, \fs_2$ be fusion subsystems of $\fs$. That is, $\fs_i$ is a subcategory of $\fs$ which is itself a fusion system. Write $\langle \fs_1, \fs_2\rangle_S$ for the smallest subsystem of $\fs$ supported on $S$ which contains both $\fs_1$ and $\fs_2$. 

At various points, we may also write $\langle \mathcal{M}_1, \mathcal{M}_2,\dots\rangle_S$ where $\mathcal{M}_i$ is some set of morphisms contained in $\fs$ and by this we mean the smallest subsystem of $\fs$ supported on $S$ which contains $\mathcal{M}_i$ for all $i$. We also mix the two conventions e.g. $\langle \fs_1, \mathcal{M}_1, \mathcal{M}_2\rangle_S$ is the smallest subsystem of $\fs$ supported on $S$ containing $\fs_1$, $\mathcal{M}_1$ and $\mathcal{M}_2$.

We emphasize that saturation is not imposed here. So even if $\fs$, $\fs_1$ and $\fs_2$ are saturated, then $\langle \fs_1, \fs_2\rangle_S$ need not be saturated. 
\end{notation}

We denote the set of $\fs$-centric subgroup of $\fs$ by $\fs^c$ and the fully $\fs$-normalized, $\fs$-centric-radical subgroups of $\fs$ by $\fs^{frc}$, referring to \cite[Definition I.2.4, Definition I.3.1]{ako} for the appropriate definitions. We present the following result as a lemma, but in truth it may be considered as part of the definition of saturation of a fusion system.

\begin{lemma}\label[lemma]{extaxiom}
Let $\fs$ be a saturated fusion system on a $p$-group $S$. For a fully $\fs$-normalized subgroup $P$ of $S$ and $R$ a subgroup of $N_S(P)$ strictly containing $P$, the morphisms in $N_{\Aut_{\fs}(P)}(\Aut_R(P))$ lift to $\fs$-automorphisms of $R$.
\end{lemma}
\begin{proof}
Since $P$ is fully $\fs$-normalized and $\fs$ is saturated, $P$ is \emph{receptive}, as defined in \cite[Definition I.2.2]{ako}. Hence, for $\alpha\in N_{\Aut_{\fs}(P)}(\Aut_R(P))$ and $N_\alpha:=\{g\in N_S(P) |{}^\alpha c_g\in \Aut_S(P)\}$, there is $P<R\le N_\alpha\le N_S(P)$ such that there is $\hat{\alpha}\in\Hom_{\fs}(N_\alpha, S)$ with $R\hat{\alpha}=R$ and $\hat{\alpha}|_P=\alpha$. Indeed, $\hat{\alpha}$ restricts to $\bar{\alpha}\in \Aut_{\fs}(R)$, as desired.
\end{proof}

Often, the morphisms we choose to lift in \cref{extaxiom} can chosen to lift all the way to certain \emph{essential subgroups} of $\fs$.

\begin{definition}
Let $\fs$ be a saturated fusion system on $S$ and let $E<S$. Then $E$ is essential in $\fs$ if $E$ is fully $\fs$-normalized, $\fs$-centric and has the property that $\Out_{\fs}(E)$ contains a strongly $p$-embedded subgroup. 

We denote by $\mathcal{E}(\fs)$ the essential subgroups of $\fs$.
\end{definition}

\begin{lemma}\label[lemma]{essfrc}
We have that $\mathcal{E}(\fs)\subseteq \fs^{frc}$.
\end{lemma}
\begin{proof}
See \cite[Proposition I.3.3(a)]{ako}.
\end{proof}

In later sections, our treatment of saturated fusion systems will focus specifically on the actions associated to essential subgroups, and the morphisms lifted to them. The reasoning behind this is that a saturated fusion system is completely determined by this information. This observation is contained in the following theorem.

\begin{theorem}[Alperin -- Goldschmidt Fusion Theorem]
Let $\fs$ be a saturated fusion system over a $p$-group $S$ and let $\mathcal{E}^0(\fs)$ be a set of representatives of the $\fs$-conjugacy classes of $\mathcal{E}(\fs)$. Then \[\fs=\langle \Aut_{\fs}(Q), \Aut_{\fs}(S) \mid Q\in\mathcal{E}^0(\fs)\rangle_S.\]
\end{theorem}
\begin{proof}
See \cite[Theorem I.3.5]{ako} and \cite[Proposition 7.25]{craven}.
\end{proof}

Importantly, we recognize that if $P$ is a group which is not properly contained in any essential subgroup, then applying \cref{extaxiom}, for any subgroup $P<R\le N _S(P)$ the morphisms in $N_{\Aut_{\fs}(P)}(\Aut_R(P))$ lift to $\fs$-automorphisms of $S$. Moreover, if $P$ is properly contained in a unique essential subgroup $E$, and $E$ is invariant under $\Aut_{\fs}(S)$ (e.g. if $E$ is characteristic in $S$), then the morphisms in $N_{\Aut_{\fs}(P)}(\Aut_R(P))$ can be chosen to lift to $\fs$-automorphisms of $E$. We will refer to this frequently throughout as the \emph{extension axiom}.

Throughout the later portions of this work, we will often employ computational methods to determine a list of potential essential subgroups of a fusion system $\fs$ supported on a given $p$-group $S$ via the fusion systems package in MAGMA \cite{Comp1}, \cite{Webpage}. 

Roughly speaking, the algorithm first determines a list a subgroups of $S$ which are self-centralizing in $S$, a prerequisite to being essential. Since the groups with a strongly $p$-embedded subgroup are ``known", the isomorphism type of $N_S(E)/E$ for a potential essential subgroup $E$ should have a prescribed form too. Then further checks are carried out which verify that certain internal conditions in $E$ hold which necessarily hold if $E$ is essential in some saturated fusion system supported on $S$. These checks and more are described in \cite{Comp1}.

The following result is a useful tool for identifying automizers of essential subgroups.

\begin{theorem}\label[theorem]{SEFF}
Suppose that $E$ is an essential subgroup of a saturated fusion system $\fs$ over a $p$-group $S$, and assume that there are $\Aut_{\fs}(E)$-invariant subgroups $U\le V\le E$ such that $E=C_S(V/U)$ and $V/U$ is an FF-module for $G:=\Out_{\fs}(E)$. Then, writing $L:=O^{p'}(G)$ and $W:=V/U$, we have that $L/C_L(W)\cong \SL_2(p^n)$, $C_L(W)$ is a $p'$-group and $W/C_W(L)$ is a natural $\SL_2(p^n)$-module for some $n\in\N$.
\end{theorem}
\begin{proof}
Since $E=C_S(W)$, we infer that $\Inn(E)=C_{\Aut_S(E)}(W)$ so that $C_G(W)$ is a $p'$-group. In particular, $G/C_G(W)$ has a strongly $p$-embedded subgroup and so too does $L/C_L(W)\cong LC_G(W)/C_G(W)=O^{p'}(G/C_G(W))$ by \cite[Remark 3.5]{henkesl2}. Then $W$ is an FF-module for $L/C_L(W)$ and we apply \cite[Theorem 5.6]{henkesl2} to obtain the result.
\end{proof}

The next two results of this section are pivotal in creating exotic fusion systems from $p$-fusion categories while maintaining saturation. The first of these techniques we refer to as ``pruning."

\begin{lemma}\label[lemma]{Pruning}
Suppose that $\fs$ is a saturated fusion system on $S$ and $P$ is an $\fs$-essential subgroup of $S$. Let $\mathcal{C}$ be a set of $\fs$-class representatives of $\fs$-essential subgroups with $P\in \mathcal{C}$. Assume that if $Q<P$ then $Q$ is not $S$-centric. Let $H_\fs(P)$ be the subgroup of $\Aut_{\fs}(P)$ which is generated by $\fs$-automorphisms of $P$ which extend to $\fs$-isomorphisms between strictly larger subgroups of $S$. If $H_{\fs}(P)\le K \le \Aut_{\fs}(P)$, then $\mathcal{G} = \langle \Aut_{\fs}(S), K, \Aut_{\fs}(E) \mid E\in\mathcal{C}\setminus \{P\}\rangle_S$ is saturated. Furthermore, $P$ is $\mathcal{G}$-essential if and only if $H_{\fs}(P)< K$.
\end{lemma}
\begin{proof}
See \cite[Lemma 6.4]{Comp1}.
\end{proof}

\begin{proposition}\label[proposition]{JasonAdd}
Let $\fs_0$ be a saturated fusion system on a finite $p$-group $S$. Let $V\le S$ be a fully $\fs_0$-normalized subgroup, set $H=\Out_{\fs_0}(V)$ and let $\wt \Delta \le \Out(V)$ be such that $H$ is a strongly $p$-embedded subgroup of $\wt \Delta$. For $\Delta$ the full preimage of $\wt \Delta$ in $\Aut(V)$, write $\fs = \langle \fs_0, \Delta\rangle_S$. Assume further that
\begin{enumerate}
    \item $V$ is $\fs_0$-centric and minimal under inclusion amongst all $\fs$-centric subgroups; and
    \item no proper subgroup of $V$ is $\fs_0$-essential.
\end{enumerate}
Then $\fs$ is saturated.
\end{proposition}
\begin{proof}
See \cite[Proposition 5.1]{BLOGeo} or \cite[Theorem C]{JasonTrees}. 
\end{proof}

We recall the notion of normalizer fusion systems from \cite[Section I.6]{ako}, noting that for $P$ a fully $\fs$-normalized subgroup, $N_{\fs}(P)$ is a saturated fusion subsystem of $\fs$. We say $P$ is \emph{normal} in $\fs$ if $\fs=N_{\fs}(P)$ and we denote by $O_p(\fs)$ the unique largest normal subgroup of $\fs$. The following proposition connects normal subgroups of $\fs$, strongly closed subgroups of $\fs$ in the sense of \cite[Definition I.4.1]{ako}, and the essential subgroups of $\fs$.

\begin{proposition}\label[proposition]{normalinF}
Let $\fs$ be a saturated fusion system over a $p$-group $S$. Then the following are equivalent for a subgroup $Q\le S$:
\begin{enumerate}
    \item $Q\normaleq \fs$;
    \item $Q$ is strongly closed in $\fs$ and contained in every centric radical subgroup of $\fs$; and
    \item $Q$ is contained in each essential subgroup, $Q$ is $\Aut_{\fs}(E)$-invariant for any essential subgroup $E$ of $\fs$ and $Q$ is $\Aut_{\fs}(S)$-invariant.
\end{enumerate}
Moreover, if $Q$ is an abelian subgroup of $S$, then $Q\normaleq \fs$ if and only if $Q$ is strongly closed in $\fs$.
\end{proposition}
\begin{proof}
See \cite[Proposition I.4.5]{ako} and \cite[Corollary I.4.7]{ako}.
\end{proof}

Fundamental to our analysis of fusion systems is the application of a plethora of known results from finite group theory. Particularly, given a fully normalized subgroup $Q$, we wish to understand the actions induced by $N_{\fs}(Q)$ and to do this, we wish to work in a finite group which models the behaviour of this normalizer subsystem.

\begin{theorem}[Model Theorem]\label[theorem]{model}
Let $\fs$ be a saturated fusion system over a $p$-group $S$. Assume that there is $Q\le S$ which is $\fs$-centric and normal in $\fs$. Then the following hold:
\begin{enumerate} 
\item There is a \emph{model} for $\fs$. That is, there is a finite group $G$ with $S\in\syl_p(G)$, $F^*(G)=O_p(G)$ and $\fs=\fs_S(G)$.
\item If $G_1$ and $G_2$ are two models for $\fs$, then there is an isomorphism $\phi: G_1\to G_2$ such that $\phi|_S = \mathrm{Id}_S$.
\item For any finite group $G$ with $S\in\syl_p(G)$, $F^*(G)=Q$ and $\Aut_G(Q)=\Aut_{\fs}(Q)$, there is $\beta\in\Aut(S)$ such that $\beta|_Q = \mathrm{Id}_Q$ and $\fs_S(G) =\fs^\beta$. Thus, there is a model for $\fs$ which is isomorphic to $G$.
\end{enumerate}
\end{theorem}
\begin{proof}
See \cite[Theorem I.4.9]{ako}.
\end{proof}

As with finite groups, we desire a more global sense of normality in fusion systems, not just restricted to $p$-subgroups. That is, we are interested in subsystems of a fusion system $\fs$ which are \emph{normal}. We use the notion of normality provided in \cite[Definition I.6.1]{ako}, noting that this condition is stronger than some of other definitions in the literature. 

By \cite[Theorem A]{CravenSub}, a proper, non-trivial normal subsystem of $\fs$ with respect to one of the accepted definitions of normality gives rise to a proper, non-trivial normal subsystem of $\fs$ with respect to the other accepted definitions. Thus, we can unambiguously declare $\fs$ to be \emph{simple} if it has no proper, non-trivial normal subsystems and so, for our purposes, the distinction between the definitions of normality is unimportant.

Of particular importance in our case is the normal subsystem $O^{p'}(\fs)$ of $\fs$, and more generally, the saturated subsystems of $\fs$ of \emph{index prime to $p$}, as in \cite[Definition I.7.3]{ako}. The following result characterizes some of the most important properties of these subsystems.

\begin{lemma}\label[lemma]{p'lemma}
Fix a saturated fusion system $\fs$ over a $p$-group $S$, and set $\mathcal{E}_0:=\langle O^{p'}(\Aut_{\fs}(P)) \mid P\le S\rangle_S$, as a (not necessarily saturated) fusion system on $S$. Define
\[\Aut_{\fs}^0(S):=\langle \alpha\in\Aut_{\fs}(S)\mid \alpha|_P\in\Hom_{\mathcal{E}_0}(P, S)\text{, some}\,\,P\in\fs^c\rangle\]
and let $\mathcal{E}$ be a saturated fusion system on $S$ of index prime to $p$ in $\fs$. Then
\begin{enumerate}
\item $\Aut_{\fs}^0(S)\le \Aut_{\mathcal{E}}(S)\le \Aut_{\fs}(S)$ and each group $L$ with $\Aut_{\fs}^0(S)\le L\le \Aut_{\fs}(S)$ gives rise to a unique saturated fusion subsystem of index prime to $p$ in $\fs$;
\item $\mathcal{E}\normaleq \fs$ if and only if $\Aut_{\mathcal{E}}(S)\normaleq \Aut_{\fs}(S)$; and
\item there is a unique minimal saturated subsystem $O^{p'}(\fs)\normaleq \fs$ of index prime to $p$, and $\Aut_{O^{p'}(\fs)}(S)=\Aut_\fs^0(S)$.
\end{enumerate}
In particular, $\Aut_{\fs}^0(S)=\Aut_{\fs}(S)$ implies that $\fs=O^{p'}(\fs)$, and $O^{p'}(O^{p'}(\fs))=O^{p'}(\fs)$.
\end{lemma}
\begin{proof}
See \cite[Theorem I.7.7]{ako}.
\end{proof}

If $\mathcal{E}$ is a saturated subsystem of index prime to $p$ in $\fs$ with $[\Aut_{\fs}(S) : \Aut_{\mathcal{E}}(S)]=r$, then we say that $\mathcal{E}$ has index $r$ in $\fs$.

We close this section with a result concerning strongly closed subgroups of fusion systems, and how they might be used to verify the exoticity of certain saturated fusion systems. In the analogous definition for finite groups, a result of Foote, building on work of Goldschmidt, promises that when $p=2$, the only simple groups $G$ which contain a non-trivial strongly closed subgroup $T<S\in\syl_2(G)$ are $\PSU_3(2^n)$ and $\Sz(2^n)$. Work of Flores and Foote \cite{Flores} complements the result in the odd prime case, using the classification of finite simple groups. From their results, we deduce the following consequence ready for use in fusion systems.

\begin{theorem}\label[theorem]{SCExotic}
Suppose that $\fs$ is a saturated fusion system over a $p$-group $S$ and $A$ is a proper non-trivial strongly closed subgroup chosen minimally with respect to adhering to these conditions. Assume that no normal subsystem of $\fs$ is supported on $A$. Then $\fs$ is exotic.
\end{theorem}
\begin{proof}
Assume that $\fs$ and $A$ satisfy the hypotheses of the lemma, and suppose that there is a finite group $G$ with $\fs=\fs_S(G)$. We may as well choose $G$ such that $O_{p'}(G)=\{1\}$. Then $A$ is a proper non-trivial strongly closed subgroup of $G$. Following \cite{Flores}, let $\mathcal{O}_A(G)$ be the largest normal subgroup $N$ of $G$ such that $A\cap N\in\syl_p(N)$. Then $\mathcal{O}_A(G)\cap A$ is a strongly closed subgroup of $G$. By the minimality of $A$, and using that $O_{p'}(G)=\{1\}$, we deduce that either $A\in\syl_p(\mathcal{O}_A(G))$ or $\mathcal{O}_A(G)=\{1\}$. In the former case, we have that $\fs_A(\mathcal{O}_A(G))\normaleq \fs_S(G)=\fs$, a contradiction. Hence, $\mathcal{O}_A(G)=\{1\}$. Applying \cite[Theorem 1.1]{Flores} when $p=2$ and \cite[Theorem 1.3]{Flores} when $p$ is odd, we conclude that $A$ is elementary abelian. But then by \cref{normalinF} we have that $A\normaleq \fs$ so that $\fs_A(A)$ is a normal subsystem of $\fs$ supported on $A$, another contradiction. Hence, no such $G$ exists and $\fs$ is exotic.
\end{proof}

This result provides an alternate check on exoticity distinct from the techniques currently used in the literature, albeit still relying on the classification of finite simple groups. 

\section{Fusion Systems on a Sylow $3$-subgroup of $\mathrm{Co}_1$}\label{Co1Sec}

In this section, we classify all saturated fusion systems supported on a $3$-group $S$ which is isomorphic to a Sylow $3$-subgroup of the sporadic simple group $\mathrm{Co}_1$, validating \hyperlink{ThmA}{Theorem A}. Utilizing the Atlas \cite{atlas}, we extract the following $3$-local maximal subgroups from $G:=\mathrm{Co}_1$: 

\begin{center}
$M_1\cong 3^6:2.\mathrm{M}_{12}$\\\vspace{0.5em}
$M_2\cong 3^{1+4}_+:\Sp_4(3).2$\\\vspace{0.5em}
$M_3\cong 3^{3+4}:2.(\Sym(4)\times \Sym(4))$
\end{center}

\noindent and remark that for a given $S\in\syl_3(G)$, $M_i$ can be chosen such that $S\in\syl_3(M_i)$. We record that $|S|=3^9$ and $J(S)=O_3(M_1)$ (where $J(S)$ is as defined in \cref{thomp}). We denote $\mathbf{J}:=O_3(M_1)$, $\mathbf{Q}:=O_3(M_2)$ and $\mathbf{R}:=O_3(M_3)$.

In addition, $S$ is isomorphic to a Sylow $3$-subgroup of $\Sp_6(3)$ and in this isomorphism we recognize the subgroups $E_1, E_2, E_3\le S$ whose images correspond to the unipotent radicals of the minimal parabolic subgroups of $\Sp_6(3)$. Indeed, $E_1, E_2, E_3$ are also essential subgroups of $\fs_S(\mathrm{Co}_1)$ such that

\begin{center}
$N_G(E_1)=M_1\cap M_2\cong 3^{1+4}_+.3^3:2.\GL_2(3)$\\\vspace{0.5em}
$N_G(E_2)=M_1\cap M_3\cong 3^6.3^2:2.\GL_2(3)$\\\vspace{0.5em}
$N_G(E_3)=M_2\cap M_3\cong 3^{3+4}.3:2.\GL_2(3)$.
\end{center}

In an abuse of notation, we suppress the isomorphism between $S$ and a Sylow $3$-subgroup of $\Sp_6(3)$ and let $E_1, E_2, E_3$ be subgroups of $\mathrm{Co}_1$ or of $\Sp_6(3)$ where appropriate. 

We also note the following characterizations of $E_1, E_2$ and $E_3$ from their embeddings in $S$.

\begin{itemize}
    \item $E_1$ is the unique subgroup of $S$ of order $3^8$ such that $\Phi(E_1)=\Phi(Y)$, where $Y$ is the preimage in $S$ of $Z(S/\mathbf{J})$ and has order $3^7$.
    \item $E_2=C_S(Z_2(S))$.
    \item $E_3$ is the unique subgroup $X$ of $S$ of order $3^8$ which is not equal to $E_1$ but satisfies $\mho^1(X)=Z(S)$.
\end{itemize}

In particular, $E_1$, $E_2$ and $E_3$ are characteristic subgroups of $S$, and so too is $\mathbf{R}=E_2\cap E_3$. In what follows, we take several liberties with the determination of various characteristic subgroups of the $E_i$, but all of these properties are easily verified by computer (e.g. using MAGMA and taking $S$ to be a Sylow $3$-subgroup of $\Sp_6(3)$).

\begin{proposition}
Let $\fs=\fs_S(\Sp_6(3))$. Then $\mathcal{E}(\fs)=\{E_1, E_2, E_3\}$.
\end{proposition}
\begin{proof}
This is a consequence of the Borel--Tits theorem \cite[Corollary 3.1.6]{GLS3}.
\end{proof}

We record one final subgroup of $G$. Let $X\normaleq M_1$ with $M_1/X\cong \mathrm{M}_{12}$ and consider the maximal subgroup $H\cong \Alt(4)\times \Sym(3)$ of $M_1/X$. Define $E_4$ to be the largest normal $3$-subgroup of the preimage of $H$ in $M_1$ so that \[N_G(E_4)=N_{M_1}(E_4)\cong 3^6.3:(\SL_2(3)\times 2).\] Then $E_4$ is an essential subgroup of $\fs_S(\mathrm{Co}_1)$, $E_4$ is not contained in any other essential subgroup of $\fs_S(\mathrm{Co}_1)$ and $[N_G(S): N_{N_G(S)}(E_4)]=6$.

\begin{proposition}
Let $\fs=\fs_S(\mathrm{Co}_1)$. Then $\mathcal{E}(\fs)=\{E_1, E_2, E_3, E_4^\fs\}$.
\end{proposition}
\begin{proof}
See \cite{Co1Sub}.
\end{proof}

We now move onto to the classification of all saturated fusion systems on $S$. Throughout we suppose that $\fs$ is a saturated fusion system on a $3$-group $S$ such that $S$ is isomorphic to a Sylow $3$-subgroup of $\mathrm{Co}_1$. 

We utilize the fusion systems package in MAGMA \cite{Comp1} \cite{Webpage} to verify the following proposition. The code and outputs are included in \cref{code}.

\begin{proposition}\label[proposition]{EssenDeter}
$\mathcal{E}(\fs)\subseteq \{E_1, E_2, E_3, E_4^\fs\}$.
\end{proposition}

For the duration of this section, we will frequently use that $\mathbf{J}=J(S)=J(E_i)$ is a characteristic subgroup of $E_i$ for $i\in\{1,2,4\}$. This follows from \cref{BasicJS} (iii).

\begin{lemma}\label[lemma]{PSLId}
Suppose that $O^{3'}(\Out_{\fs}(\mathbf{J}))\cong \PSL_3(3)$. Then $E_4^\fs\cap \mathcal{E}(\fs)=\emptyset$.
\end{lemma}
\begin{proof}
We note first that $E_4$ is contained in no other essential subgroup of $\fs$ and so by the Alperin--Goldschmidt theorem, $\{E_4^\fs\}=\{E_4^{\Aut_{\fs}(S)}\}$. Since $\mathbf{J}$ is invariant under $\Aut_{\fs}(S)$ we may as well assume, aiming for a contradiction, that $E_4\in\mathcal{E}(\fs)$. 

Since $\mathbf{J}=J(E_4)$, we have that $N_{\fs}(E_4)\le N_{\fs}(\mathbf{J})$ and so $N_{\fs}(E_4)=N_{N_{\fs}(\mathbf{J})}(E_4)$. In particular, if $E_4\in\mathcal{E}(\fs)$ then $E_4\in\mathcal{E}(N_{\fs}(\mathbf{J}))$. By the uniqueness of models provided by \cref{model}, for $H$ a model of $N_{\fs}(\mathbf{J})$, we have that $N_{\fs}(E_4)=\fs_{N_S(E_4)}(N_H(E_4))$ so that $\Out_{\fs}(E_4)=N_H(E_4)/E_4$. Since $O^{3'}(H/\mathbf{J})\cong \PSL_3(3)$, we have that $N_H(E_4)\le N_H(S)$ and so $N_H(E_4)/E_4$ does not have a strongly $3$-embedded subgroup, a contradiction.
\end{proof}

\begin{lemma}\label[lemma]{M12Iden}
If $E_4\in\mathcal{E}(\fs)$, then $\{E_1, E_2\}\subseteq\mathcal{E}(\fs)$. Moreover, $E_4\in\mathcal{E}(\fs)$ if and only if $O^{3'}(\Out_{\fs}(\mathbf{J}))\cong 2.\mathrm{M}_{12}$.
\end{lemma}
\begin{proof}
Suppose that $\fs$ is a saturated fusion system on $S$ with $E_4\in\mathcal{E}(\fs)$. Then, as $\mathbf{J}=J(E_4)$, $N_{\fs}(E_4)\le N_{\fs}(\mathbf{J})$ and so $E_4$ is also essential in $N_{\fs}(\mathbf{J})$. Since $E_4\not\normaleq S$ and $|E_4/\mathbf{J}|=3$, \cref{normalinF} implies that $\mathbf{J}=O_3(N_{\fs}(\mathbf{J}))$. By \cref{model}, there is a finite group $H$ with $S\in\syl_3(H)$, $N_{\fs}(\mathbf{J})=\fs_S(H)$ and $F^*(H)=\mathbf{J}$. Then, $O^{3'}(H)/\mathbf{J}$ is determined by \cref{36Iden}. Using that $E_4\in\mathcal{E}(\fs)$ and applying \cref{PSLId}, we conclude that $O^{3'}(H)/\mathbf{J}\cong 2.\mathrm{M}_{12}$.

Suppose now that $O^{3'}(\Out_{\fs}(\mathbf{J}))\cong 2.\mathrm{M}_{12}$ and again set $H$ to be a model for $N_{\fs}(\mathbf{J})$ so that $O^{3'}(H)/\mathbf{J}\cong 2.\mathrm{M}_{12}$. We examine the maximal subgroups of $2.\mathrm{M}_{12}$ as can be found in the Atlas \cite{atlas}, and identify them with their preimage in $O^{3'}(H)$. Then there are three classes of maximal $3$-local subgroups $H_1$, $H_2$, $H_3$, and we may arrange in each case that $S\cap H_i\in\syl_3(H_i)$. These groups have the same shape as $N_G(E_1), N_G(E_2)$ and $N_G(E_4)$ respectively. Indeed, in each case, $|N_S(O_3(H_i))/O_3(H_i)|=3$ and so we deduce that $H_i/O_3(H_i)$ contains a strongly $3$-embedded subgroup for $i\in\{1,2,3\}$. Since $\mathbf{J}=J(O_3(H_i))$, we have that $O^{3'}(\Out_{\fs}(O_3(H_i)))=O^{3'}(\Out_{H}(O_3(H_i)))=O^{3'}(\Out_{H_i}(O_3(H_i)))$ contains a strongly $3$-embedded subgroup for all $i\in\{1,2,3\}$. Moreover, each $O_3(H_i)$ is fully $\fs$-normalized and as $\mathbf{J}$ is $\fs$-centric, so too is $O_3(H_i)$ for $i\in\{1,2,3\}$. Applying \cref{EssenDeter}, we have that $O_3(H_1), O_3(H_2), O_3(H_3)$ are equal to $E_1, E_2$ and $E_4\alpha$ for some $\alpha\in\Aut_{\fs}(S)$. Hence, $E_1, E_2, E_4\in\mathcal{E}(\fs)$, as required.
\end{proof}

\begin{lemma}\label[lemma]{SingleEssenAct1}
Suppose that $E_1\in\mathcal{E}(\fs)$. Then $O^{3'}(\Out_{\fs}(E_1))\cong \SL_2(3)$, both $E_1/\mathbf{J}$ and $\Phi(E_1)/Z(S)$ are natural $\SL_2(3)$-modules for $O^{3'}(\Out_{\fs}(E_1))$, and $\mathbf{J}/\Phi(E_1)$ is a natural $\Omega_3(3)$-module for $O^{3'}(\Out_{\fs}(E_1))/Z(O^{3'}(\Out_{\fs}(E_1)))\cong \PSL_2(3)$.
\end{lemma}
\begin{proof}
Assume that $E_1\in\mathcal{E}(\fs)$. We calculate that $Z(S)=Z(E_1)$ is of order $3$, and $\Phi(E_1)=\mathbf{J}\cap \mathbf{Q}$ is elementary abelian of order $3^3$ with $C_S(\Phi(E_1))=\mathbf{J}$. Let $K:=C_{\Aut_{\fs}(E_1)}(\Phi(E_1))$ so that $\Aut_{\mathbf{J}}(E_1)\in\syl_3(K)$ and $K$ normalizes $\Inn(E_1)$. In particular, $[K, \Inn(E_1)]\le K\cap \Inn(E_1)=\Aut_{\mathbf{J}}(E_1)$ and $K$ centralizes the quotient $E_1/\mathbf{J}$. Now, as $\mathbf{J}$ is elementary abelian, $K/C_K(\mathbf{J})$ is a $3'$-group and centralizes $Z(E_1)=C_{\mathbf{J}}(\Inn(E_1))\le \Phi(E_1)$. Applying the A$\times$B-lemma, with $K|_{\mathbf{J}}$, $\Inn(E_1)|_{\mathbf{J}}$ and $\mathbf{J}$ in the roles of $A, B$ and $V$ we deduce that $K$ centralizes $\mathbf{J}$, and so $K$ centralizes the chain $\{1\}\normaleq \mathbf{J}\normaleq E_1$. By \cref{GrpChain}, $K$ is a $3$-group. Thus, $K=\Aut_{\mathbf{J}}(E_1)$ and so we infer that $\Aut_{\fs}(E_1)/K$ acts faithfully on $\Phi(E_1)$. Since $\Aut_S(E_1)$ centralizes $Z(E_1)=Z(S)$ and $\Inn(E_1)=C_{\Aut_S(E_1)}(\Phi(E_1)/Z(S))$, we conclude that $\Phi(E_1)/Z(S)$ is a natural module for $O^{3'}(\Out_{\fs}(E_1))\cong \SL_2(3)$.

 We note that for $r\in O^{3'}(\Out_{\fs}(E_1))$, if $r$ centralizes $E_1/\mathbf{J}$, then as $[E_1, \Phi(E_1)]=Z(S)$, we have by the three subgroups lemma that $[r, \Phi(E_1), E_1]=\{1\}$ so that $r$ centralizes $\Phi(E_1)/Z(S)$, a contradiction. Hence, $E_1/\mathbf{J}$ is also a natural module for $O^{3'}(\Out_{\fs}(E_1))\cong \SL_2(3)$.  Set $V:=\mathbf{J}/\Phi(E_1)$ of order $3^3$. Then for $T=Z(O^{3'}(\Out_{\fs}(E_1)))$ we have by coprime action that $V=[V,T]\times C_V(T)$. However, $\Out_S(E_1)$ acts indecomposably on $V$ and we conclude that $V=[V, T]$ or $V=C_V(T)$ is an irreducible $3$-dimensional $\SL_2(3)$-module. Thus, $V=C_V(T)$ is a natural $\Omega_3(3)$-module for $O^{3'}(\Out_{\fs}(E_1))/Z(O^{3'}(\Out_{\fs}(E_1)))\cong \PSL_2(3)$.
\end{proof}

\begin{lemma}\label[lemma]{SingleEssenAct2}
Suppose that $E_2\in\mathcal{E}(\fs)$. Then $O^{3'}(\Out_{\fs}(E_2))\cong \SL_2(3)$, $Z(E_2)$ is of order $3^3$, $|\mathbf{J}/\Phi(E_2)|=3$ and both $E_2/\mathbf{J}$ and $\Phi(E_2)/Z(E_2)$ are natural $\SL_2(3)$-modules for $O^{3'}(\Out_{\fs}(E_2))$.
\end{lemma}
\begin{proof}
Assume that $E_2\in \mathcal{E}(\fs)$. One can calculate that that $\Phi(E_2)=[E_2, \mathbf{J}]$ is of order $3^5$ and is contained in $\mathbf{J}$, and $|Z(E_2)|=3^3$ and $|Z_2(S)|=3^2$. By \cref{BasicJS} (iii), $\mathbf{J}=J(E_2)$ and $\mathbf{J}/\Phi(E_2)$ is of order $3$ and centralized by $S$. Hence, $O^{3'}(\Out_{\fs}(E_2))$ acts trivially on $\mathbf{J}/E_2$ and so must act faithfully on $E_2/\mathbf{J}$ of order $3^2$ by \cref{burnside} and coprime action. It follows that $E_2/\mathbf{J}$ is a natural module for $O^{3'}(\Out_{\fs}(E_2))\cong\SL_2(3)$. Letting $r\in O^{3'}(\Out_{\fs}(E_2))$, if $r$ centralized $\Phi(E_2)/Z(E_2)$ then by coprime action, $[r, \mathbf{J}, E_2]=\{1\}$. Moreover, since $[E_2, \mathbf{J}, r]\le Z(E_2)$ we conclude by the three subgroups lemma that $[E_2, r, \mathbf{J}]\le Z(E_2)$. But $\mathbf{J}$ is abelian so that $[E_2, r, \mathbf{J}]=[[E_2, r]\mathbf{J}, \mathbf{J}]=[E_2, \mathbf{J}]=\Phi(E_2)$, a contradiction. Hence, $\Phi(E_2)/Z(E_2)$ is a $2$-dimensional faithful module for $O^{3'}(\Out_{\fs}(E_2))$ and so is a natural $\SL_2(3)$-module.
\end{proof}

\begin{lemma}\label[lemma]{SingleEssenAct3}
Suppose that $E_3\in\mathcal{E}(\fs)$. Then $O^{3'}(\Out_{\fs}(E_3))\cong \SL_2(3)$, $\mathbf{R}$ is normalized by $\Aut_{\fs}(E_3)$, and $\mathbf{R}/\Phi(E_3)$ and $\Phi(E_3)/Z(\mathbf{R})$ are natural $\SL_2(3)$-modules for $O^{3'}(\Out_{\fs}(E_3))$.
\end{lemma}
\begin{proof}
Assume that $E_3\in \mathcal{E}(\fs)$. One may calculate that $Z_2(S)=Z_2(E_3)$ and so $C_{E_3}(Z_2(E_3))=E_2\cap E_3=\mathbf{R}\normaleq \Out_{\fs}(E_3)$. Since $S$ centralizes $E_3/\mathbf{R}$, we must have that $O^{3'}(\Out_{\fs}(E_3))$ centralizes $E_3/\mathbf{R}$. Moreover, one can calculate that $\Phi(E_3)$ is of order $3^5$ and so by \cref{burnside} and coprime action, $O^{3'}(\Out_{\fs}(E_3))$ acts faithfully on $\mathbf{R}/\Phi(E_3)$ which has order $3^2$. We conclude that $O^{3'}(\Out_{\fs}(E_3))\cong\SL_2(3)$ and $\mathbf{R}/\Phi(E_3)$ is a natural module. 

Let $r\in O^{3'}(\Out_{\fs}(E_3))$ of $3'$-order. We note that $Z(S)<Z_2(E_3)=Z_2(S)<Z(\mathbf{R})=Z(E_2)$ and $Z(E_2)$ has order $3^3$. It follows that $O^3(O^{3'}(\Out_{\fs}(E_3)))$ acts trivially on $Z(\mathbf{R})$. Assume that $r$ centralizes $\Phi(E_3)/Z(\mathbf{R})$. Then by coprime action $r$ centralizes $\Phi(E_3)$. One can calculate that $C_{E_3}(\Phi(E_3)))=Z_2(E_3)\le \Phi(E_3)$ so that $[r, E_3]=\{1\}$ by coprime action and the three subgroups lemma, a contradiction. Hence, $\Phi(E_3)/Z(\mathbf{R})$ is a $2$-dimensional faithful module for $O^{3'}(\Out_{\fs}(E_3))$ and so is a natural $\SL_2(3)$-module.
\end{proof}

\begin{proposition}\label[proposition]{SingleEssen}
Assume that $\mathcal{E}(\fs)\subseteq \{E_i\}$ for some $i\in\{1,2,3\}$. Then one of the following occurs:
\begin{enumerate}
    \item $\fs=N_{\fs}(S)$; or
    \item $\fs=N_{\fs}(E_i)$ where $O^{3'}(\Out_{\fs}(E_i))\cong \SL_2(3)$ for some $i\in\{1,2,3\}$.
\end{enumerate}
\end{proposition}
\begin{proof}
If $\mathcal{E}(\fs)=\emptyset$, then outcome (i) is satisfied by the Alperin--Goldschmidt theorem. Thus, we may assume that $E_i$ is the unique essential subgroup of $\fs$. Indeed, we must have that $E_i$ is invariant under $\Aut_{\fs}(S)$ and so $E_i\normaleq \fs$. Then \cref{SingleEssenAct1}, \cref{SingleEssenAct2} and \cref{SingleEssenAct3} completes the proof in case (ii).
\end{proof}

\begin{lemma}\label[lemma]{ExtraSpecial}
Assume that $E_1\in\mathcal{E}(\fs)$. Then there is a unique $\Aut(S)$-conjugate of $\mathbf{Q}$ which is $\Aut_{\fs}(E_1)$-invariant and $\Aut_{\fs}(S)$-invariant.
\end{lemma}
\begin{proof}
Assume that $E_1\in\mathcal{E}(\fs)$. By \cref{SingleEssenAct1}, $O^{3'}(\Out_{\fs}(E_1))\cong \SL_2(3)$ normalizes $\mathbf{J}$ and $\Phi(E_1)$, and $V:=\mathbf{J}/\Phi(E_1)$ is a irreducible $3$-dimensional module for $O^{3'}(\Out_{\fs}(E_1))/T\cong \PSL_2(3)$ where $T=C_{O^{3'}(\Out_{\fs}(E_1))}(V)$. For $U:=E_1/\Phi(E_1)$ we have that $U/V$ has order $9$ and as $T$ acts non-trivially on $U$ by coprime action, we deduce that $C_U(T)=V$ and $U=[U, T]\times V$ where $[U,T]$ is a natural $\SL_2(3)$-module.

For $X$ the preimage in $E_1$ of $[U,T]$, we have that $X\normaleq S$ and $X\cap \mathbf{J}=\Phi(E_1)$. Moreover, since $X/\Phi(E_1)$ is an irreducible module for $\Out_{\fs}(E_1)$, we deduce that $|\Omega_1(X)|\ne 3^4$. With this information, we calculate that there are $3$ subgroups of $E_1$ satisfying these properties including $X$. Furthermore, since $E_1, \Phi(E_1)$ and $\mathbf{J}$ are all characteristic subgroups of $S$, we have that $X\alpha$ also satisfies these properties for all $\alpha\in\Aut(S)$, and we calculate that under the action of $\Aut(S)$, all $3$ subgroups of $E_1$ are conjugate (see \cref{code} for the explicit code for these calculations). Finally, since $\mathbf{Q}$ satisfies these properties, we conclude that there is $\alpha\in\Aut(S)$ such that $X=\mathbf{Q}\alpha$. By the module decomposition of $U$ above, $X$ is the unique such $\Aut(S)$-conjugate of $\mathbf{Q}$ which is $\Aut_{\fs}(E_1)$-invariant.
\end{proof}

By definition, $\fs^{\alpha^{-1}}$ is a saturated fusion system on $S$ which is isomorphic to $\fs$, for $\alpha\in\Aut(S)$. Furthermore, it follows from the above lemma that there is $\alpha\in\Aut(S)$ such that $\mathbf{Q}$ is the unique subgroup of $S$ in its $\Aut(S)$-conjugacy class which is both $\Aut_{\fs^{\alpha^{-1}}}(E_1)$-invariant and $\Aut_{\fs^{\alpha^{-1}}}(S)$-invariant. Since we are only interested in investigating the possibilities of $\fs$ up to isomorphism, we may as well assume for the remainder of this section that $\mathbf{Q}$ is $\Aut_{\fs}(E_1)$-invariant whenever $E_1\in\mathcal{E}(\fs)$. Indeed, $\mathbf{Q}$ is the preimage in $E_1$ of $[E_1/\Phi(E_1), Z(O^{3'}(\Aut_{\fs}(E_1)))]$.

\begin{proposition}\label[proposition]{DoubleEssen1}
Suppose that $\{E_1, E_2\}\subseteq \mathcal{E}(\fs)$. Then either
\begin{enumerate}
    \item $E_4^\fs\cap \mathcal{E}(\fs)=\emptyset$, $\mathcal{E}(N_{\fs}(\mathbf{J}))=\{E_1, E_2\}$ and $O^{3'}(\Out_{\fs}(\mathbf{J}))\cong \PSL_3(3)$; or
    \item $E_4\in\mathcal{E}(\fs)$, $\mathcal{E}(N_{\fs}(\mathbf{J}))=\{E_1, E_2, E_4^\fs\}$ and $O^{3'}(\Out_{\fs}(\mathbf{J}))\cong 2.\mathrm{M}_{12}$.
\end{enumerate}
Moreover, in each case, if $E_3\not\in\mathcal{E}(\fs)$ then $\fs=N_{\fs}(\mathbf{J})$.
\end{proposition}
\begin{proof}
By \cref{EssenDeter}, $\mathcal{E}(N_{\fs}(\mathbf{J}))\subseteq \{E_1, E_2, E_3, E_4^\fs\}$. We note that as $\mathbf{J}=J(E_1)=J(E_2)=J(E_4)$, $\Out_{\fs}(E_i)=\Out_{N_{\fs}(\mathbf{J})}(E_i)$ for $i\in\{1,2,4\}$. Furthermore, since $E_i$ is self-centralizing in $S$ and fully normalized in $\fs$, we see that $E_i\in\mathcal{E}(N_{\fs}(\mathbf{J}))$ if and only if $E_i\in\mathcal{E}(\fs)$ for $i\in\{1,2,4\}$. Since $\mathbf{J}\not\le E_3$, we necessarily have that $E_3\not\in\mathcal{E}(N_{\fs}(\mathbf{J}))$ by \cref{normalinF}.

Suppose that $\{E_1, E_2\}\subseteq \mathcal{E}(\fs)$. Let $X$ be the largest subgroup normalized by $\Aut_{\fs}(E_1)$ and $\Aut_{\fs}(E_2)$. Since $\mathbf{J}=J(E_1)=J(E_2)$, we have that $\mathbf{J}\le X\le E_1\cap E_2$. Furthermore, by \cref{SingleEssenAct1}, $E_1/\mathbf{J}$ is irreducible under $\Aut_{\fs}(E_1)$ and we deduce that $X=\mathbf{J}$ and $\mathbf{J}=O_3(N_{\fs}(\mathbf{J}))$. Indeed, $\Out_{\fs}(\mathrm{J})$ satisfies the hypothesis of \cref{36Iden} and we deduce that $O^{3'}(\Out_{\fs}(\mathbf{J}))\cong \PSL_3(3)$ or $2.\mathrm{M}_{12}$. In the former case, we have by \cref{PSLId} that $E_4^\fs\cap \mathcal{E}(\fs)=\emptyset$, and so (i) holds. In the latter case, we have by \cref{M12Iden} that $E_4\in\mathcal{E}(\fs)$, and so (ii) holds. Finally, since $\mathbf{J}=J(S)$ and $\mathbf{J}$ is invariant under $\Aut_{\fs}(S)$, \cref{normalinF} and \cref{EssenDeter} imply that if $E_3\not\in\mathcal{E}(\fs)$ then $\fs=N_{\fs}(\mathbf{J})$.
\end{proof}

\begin{proposition}\label[proposition]{DoubleEssen2}
Suppose that $\{E_1, E_3\}\subseteq \mathcal{E}(\fs)$. Then $O^{3'}(\Out_{\fs}(\mathbf{Q}))\cong \Sp_4(3)$, $\mathcal{E}(N_{\fs}(\mathbf{Q}))=\{E_1, E_3\}$ and if $E_2\not\in\mathcal{E}(\fs)$ then $\fs=N_{\fs}(\mathbf{Q})$.
\end{proposition}
\begin{proof}
Suppose that $\{E_1, E_3\}\subseteq \mathcal{E}(\fs)$ and let $X$ be the largest subgroup of $S$ normalized by both $\Aut_{\fs}(E_1)$ and $\Aut_{\fs}(E_3)$. Then $X\le E_1\cap E_3$ so that $\mathbf{J}\not\le X$. Since $\Aut_{\fs}(E_1)$ acts irreducibly on $\mathbf{J}/\Phi(E_1)$, by the choice of $\mathbf{Q}$ following \cref{ExtraSpecial} we have that $X\le \mathbf{Q}$. We note that $Z(S)=Z(E_1)=Z(E_3)$ so that $Z(S)\le X$.

Assume first that $X=Z(S)$, let $G_i$ be a model for $N_{\fs}(E_i)$, where $i\in\{1,3\}$, and $G_{13}$ be a model for $N_{\fs}(S)$. Since $E_1$ and $E_3$ are $\Aut_{\fs}(S)$-invariant, we can arrange that there are injective maps $\alpha_i:G_{13}\to G_i$ for $i\in\{1,3\}$. Furthermore, since $Z(S)\normaleq G_1, G_3$, we may form injective maps $\alpha_i^*:G_{13}/Z(S)\to G_i/Z(S)$ so that the tuple $(G_1/Z(S), G_3/Z(S), G_{13}/Z(S), \alpha_1^*, \alpha_3^*)$ satisfies the hypothesis of \cite[Theorem A]{Greenbook}. Since $|S/Z(S)|=3^8$ and $|Z(S/Z(S))|=3$, comparing with the outcomes provided by \cite[Theorem A]{Greenbook}, we have a contradiction.

Thus, $Z(S)<X$ and we deduce that $Z(S)<X\cap Z_2(E_1)\le \Phi(E_1)$. By \cref{SingleEssenAct1}, $\Aut_{\fs}(E_1)$ is irreducible on $\Phi(E_1)/Z(S)$ and so we have that $\Phi(E_1)\le X$. If $X=\Phi(E_1)$ then $|X|=3^3$ and $X\cap Z_2(S)>Z(S)$. Hence, $|X\Phi(E_3)/\Phi(E_3)|\leq 3$ and as $\Aut_{\fs}(E_3)$ acts irreducibly on $\mathbf{R}/\Phi(E_3)$ by \cref{SingleEssenAct3} we deduce that $X\le \Phi(E_3)$. Similarly, $|XZ_2(S)/Z_2(S)|\leq 3$ and as $\Aut_{\fs}(E_3)$ acts irreducibly on $\Phi(E_3)/Z_2(S)$ by \cref{SingleEssenAct3} we deduce that $X=Z_2(S)$, a contradiction since $X/Z(S)$ is a natural $\SL_2(3)$-module for $O^{3'}(\Out_{\fs}(E_1))$. Hence, $\Phi(E_1)<X$. Finally, since $X\le \mathbf{Q}$ and $\Aut_{\fs}(E_1)$ acts irreducibly on $\mathbf{Q}/\Phi(E_1)$ by \cref{SingleEssenAct1}, we have that $X=\mathbf{Q}$.

We have that $O^{3'}(\Out_{\fs}(\mathbf{Q}))$ acts faithfully on $\mathbf{Q}$. By \cite{Winter}, we deduce that $O^{3'}(\Out_{\fs}(\mathbf{Q}))$ is isomorphic to a subgroup of $O^{3'}(\Out(\mathbf{Q}))\cong\Sp_4(3)$. Hence, $\Out_S(\mathbf{Q})\in\syl_3(\Out(\mathbf{Q}))$ and $O^{3'}(\Out_{\fs}(\mathbf{Q}))$ is an overgroup of $\Out_S(\mathbf{Q})$ with no non-trivial normal $3$-subgroups. By \cite[Table 8.12]{LowMax}, any maximal subgroup of $\Sp_4(3)$ which contains a Sylow $3$-subgroup is a parabolic subgroup so has a normal $3$-subgroups. Hence, $O^{3'}(\Out_{\fs}(\mathbf{Q}))$ is contained in no maximal subgroups so that $O^{3'}(\Out_{\fs}(\mathbf{Q}))\cong \Sp_4(3)$.

We note that the maximal abelian subgroups of $\mathbf{Q}$ have order $3^3$ and so $\mathbf{Q}\mathbf{J}=E_1$. In particular, $E_2\not\ge \mathbf{Q}\not\le E_4$ and neither $E_2$ nor $E_4$ are essential in $N_{\fs}\mathbf{(Q)}$ by \cref{normalinF}. Since $E_1, E_3$ are $\fs$-centric, normal in $S$ and satisfy $\Out_{\fs}(E_i)=\Out_{N_{\fs}(\mathbf{Q})}(E_i)$, we deduce that $E_1, E_3\in\mathcal{E}(\fs)$ if and only if $E_1, E_3\in\mathcal{E}(N_{\fs}(\mathbf{Q}))$. By \cref{M12Iden}, if $E_4^\fs\cap \mathcal{E}(\fs)\ne \emptyset$, then $E_2\in\mathcal{E}(\fs)$ and so by \cref{EssenDeter}, if $E_2\not\in\mathcal{E}(\fs)$ then $\mathcal{E}(\fs)=\{E_1, E_3\}$. In particular, since we have arranged that $\mathbf{Q}$ is $\Aut_{\fs}(S)$-invariant by \cref{ExtraSpecial}, applying \cref{normalinF} we see that $\fs=N_{\fs}(\mathbf{Q})$, completing the proof.
\end{proof}

\begin{proposition}\label[proposition]{DoubleEssen3}
Suppose that $\{E_2, E_3\}\subseteq \mathcal{E}(\fs)$. Then $O^{3'}(\Out_{\fs}(\mathbf{R}))\cong \Omega_4^+(3)\cong \SL_2(3) \sbt_{C_2} \SL_2(3)$, $\mathcal{E}(N_{\fs}(\mathbf{R}))=\{E_2, E_3\}$ and if $E_1\not\in\mathcal{E}(\fs)$ then $\fs=N_{\fs}(\mathbf{R})$.
\end{proposition}
\begin{proof}
Suppose that $\{E_2, E_3\}\subseteq \mathcal{E}(\fs)$. By \cref{SingleEssenAct3}, we have that $\mathbf{R}=E_2\cap E_3$ is characteristic in $E_3$. Recall from \cref{SingleEssenAct2} that for $V:=E_2/\Phi(E_2)$ and $L:=O^{3'}(\Out_{\fs}(E_2))\cong \SL_2(3)$, $V=[V,L]\times C_V(L)$ where $[V,L]$ has order $3^2$ and $C_V(L)=\mathbf{J}/\Phi(E_2)$. 

We claim that $\mathbf{R}$ is the preimage of $[V,L]$ in $E_2$ and so is normalized by $\Aut_{\fs}(E_2)$. First, observe that $[E_2, E_3]\Phi(E_2)/\Phi(E_2)$ has order $3$ and is contained in $([V,L]\cap \mathbf{R})/\Phi(E_2)$. Since $E_3$ is $\Aut_{\fs}(S)$-invariant, we deduce that either $\mathbf{R}$ is the preimage of $[V,L]$, or $L$ centralizes $\mathbf{R}/[E_2, E_3]\Phi(E_2)$. In the latter case, we deduce that $C_V(L)\le \mathbf{R}/\Phi(E_2)$ so that $\mathbf{J}\le \mathbf{R}$, a contradiction. Hence, $\mathbf{R}$ is the preimage in $E_2$ of $[V,L]$ and so is normalized by $\Aut_{\fs}(E_2)$.

Since $\Phi(\mathbf{R})=Z(\mathbf{R})$, $|\mathbf{R}/\Phi(\mathbf{R})|=3^4$ and applying \cref{burnside}, we deduce that $O^{3'}(\Out_{\fs}(\mathbf{R}))$ is isomorphic to a subgroup $\SL_4(3)$. Set $\bar{\mathbf{R}}=\mathbf{R}/\Phi(\mathbf{R})$. We note that $|C_{\bar{\mathbf{R}}}(\Out_S(\mathbf{R}))|=3$ and that $\bar{\mathbf{R}}=\langle C_{\bar{\mathbf{R}}}(\Out_S(\mathbf{R}))^{\Out_{\fs}(\mathbf{R})}\rangle$ by the actions of $N_{\Out_{\fs}(\mathbf{R})}(\Out_{E_i}(\mathbf{R}))\cong \Aut_{\fs}(E_i)/\Aut_{\mathbf{R}}(E_i)$ for $i\in\{2,3\}$. In particular, $\Out_{\fs}(\mathbf{R})$ stabilizes no subspaces of $\bar{R}$ and $\bar{R}$ is indecomposable under $\Out_{\fs}(\mathbf{R})$. Moreover, $|N_{O^{3'}(\Out_{\fs}(\mathbf{R}))}(\Out_{E_i}(\mathbf{R}))|$ is divisible by $8$ and we deduce that $O^{3'}(\Out_{\fs}(\mathbf{R}))\not\cong \mathrm{(P)SL}_2(9)$. Comparing with \cite[Table 8.8]{LowMax}, we deduce that $O^{3'}(\Out_{\fs}(\mathbf{R}))$ is isomorphic to a subgroup of $\mathrm{SO}_4^+(3)$ or $\Sp_4(3)$. In the latter case, we check against the tables of maximal subgroups of $\Sp_4(3)$ \cite[Table 8.12]{LowMax} and find no suitable candidates which contain $O^{3'}(\Out_{\fs}(\mathbf{R}))$. In the former case, since $|\mathrm{SO}_4^+(3)|_3=3^2$ and comparing orders we deduce that $O^{3'}(\Out_{\fs}(\mathbf{R}))\cong O^{3'}(\mathrm{SO}_4^+(3))=\Omega_4^+(3)$, as desired.

Since $\mathbf{R}\normaleq S$ is of order $3^7$, contained in $E_2$ and does not contain $\mathbf{J}$ (for otherwise $\mathbf{J}\le E_3$), we see that $E_1\not\ge \mathbf{R}\not\le E_4$ and neither $E_1$ nor $E_4$ are essential in $N_{\fs}\mathbf{(R)}$ by \cref{normalinF}. Since $E_2, E_3$ are $\fs$-centric, normal in $S$ and satisfy $\Out_{\fs}(E_i)=\Out_{N_{\fs}(\mathbf{R})}(E_i)$, we deduce that $E_2, E_3\in\mathcal{E}(\fs)$ if and only if $E_2, E_3\in\mathcal{E}(N_{\fs}(\mathbf{R}))$. By \cref{M12Iden}, if $E_4^\fs\cap \mathcal{E}(\fs)\ne \emptyset$, then $E_1\in\mathcal{E}(\fs)$ and so by \cref{EssenDeter}, if $E_1\not\in\mathcal{E}(\fs)$ then $\mathcal{E}(\fs)=\{E_2, E_3\}$. Since $E_2$ and $E_3$ are characteristic subgroups of $S$, so too is $\mathbf{R}=\mathbf{R}$. Hence, $\mathbf{R}$ is $\Aut_{\fs}(S)$-invariant and so if $E_1\not\in\mathcal{E}(\fs)$, then applying \cref{normalinF}, we have that $\fs=N_{\fs}(\mathbf{R})$, completing the proof.
\end{proof}

Hence, as consequence of \cref{EssenDeter}, \cref{M12Iden} and  \cref{SingleEssen}-\cref{DoubleEssen3}, we have proved the following result.

\begin{proposition}\label[proposition]{Co1Essens}
Suppose that $\fs$ is a saturated fusion system on a $3$-group $S$ such that $S$ is isomorphic to a Sylow $3$-subgroup of $\mathrm{Co}_1$. If $O_3(\fs)=\{1\}$ then $\mathcal{E}(\fs)=\{E_1, E_2, E_3\}$ or $\mathcal{E}(\fs)=\{E_1, E_2, E_3, E_4^\fs\}$.
\end{proposition}

We now complete the classification of all saturated fusion systems supported on $S$. As evidenced in \cref{DoubleEssen2} and \cref{DoubleEssen3}, the structure of $O^{3'}(\Out_{\fs}(\mathbf{Q}))$ and $O^{3'}(\Out_{\fs}(\mathbf{R}))$ is fairly rigid and the flexibility we exploit is in the possible choices of actions for $\Out_{\fs}(\mathbf{J})$.

The identification of the fusion systems of $\Sp_6(3)$ and $\Aut(\Sp_6(3))$ is proved using a result of Onofrei \cite{Ono} which identifies a \emph{parabolic system} in $\fs$. Further restrictions then identify $\Sp_6(3)$ from an associated chamber system. We remark that in the case of parabolic systems in groups, the definition is meant to abstractly capture a set of minimal parabolics containing a ``Borel", in analogy with groups of Lie type in defining characteristic. We cannot hope to capture the rich theory of parabolic systems in groups (and fusion systems) here, but we refer to \cite{MeixSurvey} for a survey of this area in the group theory case, and refer to \cite{Ono} for the fusion system parallel.

\begin{theorem}\label[theorem]{Sp63}
Suppose that $\fs$ is a saturated fusion system on a $3$-group $S$ such that $S$ is isomorphic to a Sylow $3$-subgroup of $\mathrm{Co}_1$. If $\mathcal{E}(\fs)=\{E_1, E_2, E_3\}$ then $\fs=\fs_S(H)$ such that $H\cong \Sp_6(3)$ or $\Aut(\Sp_6(3))$.
\end{theorem}
\begin{proof}
Let $\fs_{ij}:=\langle N_{\fs}(E_i), N_{\fs}(E_j)\rangle_S$ for $i,j\in\{1,2,3\}$, noting that $N_{\fs}(S)\le N_{\fs}(E_i)$ for all $i\in\{1,2,3\}$. Then $\mathbf{J}\normaleq \fs_{12}$ and as $E_4\not\in\mathcal{E}(\fs)$, \cref{DoubleEssen1} along with the Alperin--Goldschmidt theorem imply that $\fs_{12}=N_{\fs}(\mathbf{J})$. Applying \cref{DoubleEssen2} we have that $\fs_{13}=N_{\fs}(\mathbf{Q})$, and \cref{DoubleEssen3} yields that $\fs_{23}=N_{\fs}(\mathbf{R})$. 

Let $\alpha\in \Hom_{N_{\fs}(E_i)\cap N_{\fs}(E_j)}(P, Q)$ for $P,Q\le S$, $i\ne j$ and $i,j\in \{1,2,3\}$. Since $E_i\normaleq N_{\fs}(E_i)\cap N_{\fs}(E_j)$, there is $\hat{\alpha}\in\Hom_{N_{\fs}(E_i)\cap N_{\fs}(E_j)}(PE_i, QE_j)$ with $\hat{\alpha}|_P=\alpha$. But $E_j\normaleq N_{\fs}(E_i)\cap N_{\fs}(E_j)$ and so there is $\wt \alpha\in \Hom_{N_{\fs}(E_i)\cap N_{\fs}(E_j)}(PE_iE_j, QE_iE_j)$ with $\wt\alpha|_{PE_i}=\hat{\alpha}$. Since $E_iE_j=S$, we have shown that for all $\alpha\in\Hom_{N_{\fs}(E_i)\cap N_{\fs}(E_j)}(P, Q)$, there is $\wt \alpha\in \Aut_{N_{\fs}(E_i)\cap N_{\fs}(E_j)}(S)$ with $\wt \alpha|_P=\alpha$. Hence, $S\normaleq N_{\fs}(E_i)\cap N_{\fs}(E_j)$ so that $N_{\fs}(S)=N_{\fs}(E_i)\cap N_{\fs}(E_j)$ whenever $i\ne j$. Hence, $\{\fs_i; i\in\{1,2,3\}\}$ is a family of parabolic subsystems in the sense of \cite[Definition 5.1]{Ono}.

In fact, following \cite[Definition 7.4]{Ono}, $\fs$ has a family of parabolic subsystems of type $\mathfrak{M}$, where $\mathfrak{M}$ is the diagram associated to $\fs$ described in that definition. By \cite[Theorem A]{meixner}, $\mathfrak{M}$ is a spherical diagram and \cite[Proposition 7.5 (ii)]{Ono} implies that $\fs$ is the fusion system of a finite simple group $G$ of Lie type in characteristic $p$ extended by diagonal and field automorphisms. Then $N_{\fs}(\mathbf{Q})=\fs_S(N_G(\mathbf{Q}))$ and as $O^{3'}(\Out_{\fs}(\mathbf{Q}))\cong \Sp_4(3)$ acts irreducibly on $\mathbf{Q}/Z(S)$, we conclude that $N_G(\mathbf{Q})=O^3(N_G(\mathbf{Q}))$ so that $G=O^3(G)$. Comparing with the structure of the Sylow $3$-subgroups of the finite simple groups of Lie type (as can be found in \cite[Section 3.3]{GLS3}), we deduce that $\fs=\fs_S(G)$ where $\Inn(\Sp_6(3))\le G\le \Aut(\Sp_6(3)))$. 
\end{proof}

\begin{theorem}\label[theorem]{Co1}
Suppose that $\fs$ is a saturated fusion system on a $3$-group $S$ such that $S$ is isomorphic to a Sylow $3$-subgroup of $\mathrm{Co}_1$. If $\mathcal{E}(\fs)=\{E_1, E_2, E_3, E_4^\fs\}$ then $\fs\cong\fs_S(\mathrm{Co}_1)$.
\end{theorem}
\begin{proof}
We observe first that $\mathcal{G}:=\fs_S(\mathrm{Co}_1)$ satisfies the hypothesis of the proposition and that $\mathcal{G}=\langle \Aut_{\mathcal{G}}(E_1), \Aut_{\mathcal{G}}(E_2), \Aut_{\mathcal{G}}(E_3), \Aut_{\mathcal{G}}(E_4), \Aut_{\mathcal{G}}(S)\rangle$ by the Alperin--Goldschmidt theorem. Since we are only interested in determining $\fs$ up to isomorphism, we may arrange by \cref{ExtraSpecial} that $\mathbf{Q}$ is $\Aut_{\fs}(S)$-invariant.

Since $\mathbf{J}$ is characteristic in $E_1, E_2 $ and $E_4$, $\mathbf{Q}\normaleq N_{\fs}(E_3)$ and $\mathbf{Q}\normaleq N_{\mathcal{G}}(E_3)$, we see that $\mathcal{G}=\langle N_{\mathcal{G}}(\mathbf{J}), N_{\mathcal{G}}(\mathbf{Q})\rangle_S$ and $\fs=\langle N_{\fs}(\mathbf{J}), N_{\fs}(\mathbf{Q})\rangle_S$. Hence, upon showing that $N_{\mathcal{G}}(\mathbf{J})=N_{\fs}(\mathbf{J})$ and $N_{\mathcal{G}}(\mathbf{Q})=N_{\fs}(\mathbf{Q})$, we will have shown that $\fs=\mathcal{G}$ and the proof will be complete.

Applying \cref{DoubleEssen1} and \cref{DoubleEssen2}, since $E_4\in\mathcal{E}(\fs)$, we have that $O^{3'}(\Aut_{\fs}(\mathbf{J}))\cong 2.\mathrm{M}_{12}$ and $O^{3'}(\Out_{\fs}(\mathbf{Q}))\cong \Sp_4(3)$. We may lift the $3'$-order morphisms in $N_{O^{3'}(\Aut_{\fs}(\mathbf{J}))}(\Aut_S(\mathbf{J}))$ to morphisms in $\Aut_{\fs}(S)$ by the extension axiom, which then restrict faithfully to morphisms of $\Aut_{\fs}(\mathbf{Q})$ by \cref{ExtraSpecial}. Similarly, any morphism in $N_{\Aut_{\fs}(\mathbf{Q})}(\Aut_S(\mathbf{Q}))$ lift to morphisms in $\Aut_{\fs}(S)$ by the extension axiom and restrict faithfully to morphisms in $N_{\Aut_{\fs}(\mathbf{J})}(\Aut_S(\mathbf{J}))$. Comparing the orders of the normalizer of a Sylow $3$-subgroup of $2.\mathrm{M}_{12}$ with the normalizer of a Sylow $3$-subgroup of $\Out(\mathbf{Q})\cong \Sp_4(3).2$, and applying the Frattini argument, we deduce that $\Aut_{\fs}(\mathbf{Q})=\Aut(\mathbf{Q})=\Aut_{\mathcal{G}}(\mathbf{Q})\cong 3^4:(\Sp_4(3):2)$ and $\Aut_{\fs}(\mathbf{J})\cong 2.\mathrm{M}_{12}$. Then by \cref{model}, we conclude that there is $\beta\in\Aut(S)$ with $N_{\fs^\beta}(\mathbf{Q})=N_{\mathcal{G}}(\mathbf{Q})$ and as we are only interested in determining $\fs$ up to isomorphism, we arrange that $\fs=\fs^\beta$ and $N_{\mathcal{G}}(\mathbf{Q})=N_{\fs}(\mathbf{Q})$.

Now, $N_{\mathcal{G}}(E_1)=N_{N_{\mathcal{G}}(\mathbf{Q})}(E_1)=N_{N_{\fs}(\mathbf{Q})}(E_1)=N_{\fs}(E_1)$. Then $N_{\mathcal{G}}(\mathbf{J})\ge N_{\mathcal{G}}(E_1)\le N_{\fs}(\mathbf{J})$ and by \cite[Proposition 2.11]{BobTodd}, it suffices to show that $\Aut_{\mathcal{G}}(\mathbf{J})=\Aut_{\fs}(\mathbf{J})$ and that the homomorphism $H^1(\Out_{\mathcal{G}}(\mathbf{J}); \mathbf{J}) \to H^1(\Out_{N_{\mathcal{G}}(E_1)}(\mathbf{J}); \mathbf{J})$ induced by restriction is surjective. For the latter condition, we calculate in MAGMA (see \cref{code}) that $H^1(\Out_{N_{\mathcal{G}}(E_1)}(\mathbf{J}); \mathbf{J})=\{1\}$ and so the homomorphism is surjective.

Let $K:=\Aut_{N_{\mathcal{G}}(E_1)}(\mathbf{J})$, $X:=\Aut_{\mathcal{G}}(\mathbf{J})$ and $Y:=\Aut_{\fs}(\mathbf{J})$ so that $K\le X\cap Y\le \Aut(\mathbf{J})\cong \GL_6(3)$. We aim to show that $X=Y$. Since there is only one conjugacy class of groups isomorphic to $2.\mathrm{M}_{12}$ in $\GL_6(3)$, we may assume that there is $g\in \Aut(\mathbf{J})$ with $Y=X^g$ and $K\le X\cap Y$. Hence, $K, K^g\le Y\cong 2.\mathrm{M}_{12}$. Now, $K$ is the unique overgroup of $T\in\syl_3(X)$ of its isomorphism type whose largest normal $3$-subgroup centralizes only an element of order $3$ in $\mathbf{J}$. Then, $K^g$ is the unique overgroup of $T^g\in\syl_3(X^g)$ with the same properties. Since $K\le X^g=Y$, $K$ is an overgroup of $P\in\syl_3(X^g)$ with $O_3(K)$ centralizing only an element of order $3$ in $\mathbf{J}$. Thus, for $m\in X^g$ with $P^m=T^g$, $K^m$ and $K^g$ are isomorphic overgroups of $T^g\in\syl_3(X)$ and by uniqueness, we deduce that $K^m=K^g$. But now, $K=K^{gm^{-1}}$ so that $gm^{-1}\in N_{\GL_6(3)}(K)$ and $X^g=X^{gm^{-1}}$. However, one can calculate that $N_{\GL_6(3)}(K)=N_{X}(K)$ so that $Y=X^g=X$.
\end{proof}

\begin{remark}
Suppose that $\fs=\fs_S(\mathrm{Co}_1)$ and set $\fs_0:=\langle N_{\fs}(E_1), N_{\fs}(E_2), N_{\fs}(E_3)\rangle_S$. The Alperin--Goldschmidt theorem yields that $\Aut_{\fs}(E_4)\not\subset \fs_0$ so that $\fs_0<\fs$. By \cref{DoubleEssen2}, we have that $N_{\fs}(\mathbf{Q})\le \fs_0$. Then as $O_3(\fs_0)\le O_3(N_{\fs}(\mathbf{Q}))$, we conclude that if $O_p(\fs_0)\ne\{1\}$ that $Z(S)=\Phi(\mathbf{Q})\normaleq \fs_0$. But $Z(S)\not\normaleq N_{\fs}(E_2)=N_{\fs_S(\mathrm{Co}_1)}(E_2)$ and so $O_3(\fs_0)=\{1\}$. Hence, by \cref{Sp63} and \cref{Co1}, if $\fs_0$ is saturated then $\fs_0\cong \fs_S(G)$ where $G\in\{\Sp_6(3), \Aut(\Sp_6(3))\}$. But then $\PSL_3(3)\cong O^{3'}(\Out_{\fs_0}(\mathbf{J}))\le O^{3'}(\Out_{\fs}(\mathbf{J}))\cong 2.\mathrm{M}_{12}$. But $13$ divides $|\PSL_3(3)|$ and does not divide $|2.\mathrm{M}_{12}|$ and so we conclude that $\fs_0$ is not saturated.
\end{remark}

The above remark is of particular interest in the mission of classifying fusion systems which contain parabolic systems. In the case of the group $G:=\mathrm{Co}_1$, the groups $N_G(E_i)$ for $i\in\{1,2,3\}$ all contain the ``Borel" $N_G(S)$ and together generate $G$ and so successfully form something akin to a parabolic system. Utilized above, work by Onofrei \cite{Ono} parallels the group phenomena in fusion systems and provides conditions in which a parabolic system within a fusion system $\fs$ gives rise to a parabolic system in the group sense. The resulting completion of the group parabolic system realizes the fusion system and if certain additional conditions are satisfied, the fusion system is saturated. 

Comparing with \cite[Definition 5.1]{Ono}, if $\fs_0$ does not have a family of parabolic subsystems then the only possible condition we fail to satisfy for $\fs_0$ is condition (F4). Indeed, the subsystem $\langle N_{\fs}(E_1), N_{\fs}(E_2)\rangle_S$ is \emph{not} a saturated fusion system. Part of the reason this problem arises is that the $3$-fusion category of $2.\mathrm{M}_{12}$ is isomorphic to the $3$-fusion category of $\PSL_3(3)$ and, consequently, the image of $E_4$ is not essential in the quotient $N_{\fs}(J(S))/J(S)$. 

However, we still retain that \[\langle N_{\fs}(E_1), N_{\fs}(E_2)\rangle_S \le \fs_S(\langle N_H(E_1), N_H(E_2)\rangle)=N_{\fs}(J(S))\] where $N_{\fs}(J(S))$ is a saturated constrained fusion system with model $H$. Thus, we can still embed the models for $N_{\fs}(E_1)$, $N_{\fs}(E_2)$ uniquely in $H$ and obtain a parabolic system of groups. Perhaps it is possible in all the situations we care about to create an embedding $\langle N_{\fs}(E_1), N_{\fs}(E_2)\rangle_S \le \fs_S(\langle G_1, G_2\rangle) \le N_{\fs}(U)$ where $N_{\fs}(U)$ is constrained and $G_1, G_2$ are the models of $N_{\fs}(E_1), N_{\fs}(E_2)$. In such a circumstance, we should always be able to work in a group setting and can then force restrictions on the structures of $N_{\fs}(E_i)$ for $i\in\{1,2\}$.

Finally, we remark that the above example of $\mathrm{Co}_1$ at the prime $3$ is similar in spirit to the example of $\mathrm{M}_{24}$ at the prime $2$ given in \cite[pg 53]{ParkerMax}.

\section{Fusion Systems related to a Sylow $3$-subgroup of $\mathrm{F}_3$}\label{F3Sec}

We now investigate fusion systems supported on a $3$-group $S$ which is isomorphic to a Sylow $3$-subgroup of the Thompson sporadic simple group $\mathrm{F}_3$. For the exoticity checks in this section, we will use some terminology and results regarding the known finite simple groups. As a reference, we use \cite{GLS3}. Again, for structural results concerning $S$ and its internal actions, we appeal to the Atlas \cite{atlas}. We begin by noting the following $3$-local maximal subgroups of $\mathrm{F}_3$:

\begin{center}
$M_1\cong 3^{2+3+2+2}:\GL_2(3)$\\\vspace{0.5em}
$M_2\cong 3^{1+2+1+2+1+2}:\GL_2(3)$\\\vspace{0.5em}
$M_3\cong 3^5:\SL_2(9).2$
\end{center}

remarking that $|S|=3^{10}$ and that for a given $S\in\syl_3(\mathrm{F}_3)$, each $M_i$ may be chosen so that $S\cap M_i\in\syl_3(M_i)$. We make this choice for each $M_i$. 

Set $E_i=O_3(M_i)$ so that $E_1=C_S(Z_2(S))$ and $E_2=C_S(Z_3(S)/Z(S))$ are characteristic subgroups of $S$, and so are $\Aut_{\fs}(S)$-invariant in any fusion system $\fs$ on $S$. We obtain generators for $M_1$ and $M_2$ (and hence for $S$, $E_1$ and $E_2$) as in \cref{F3Unique}. For ease of notation, we fix $\mathcal{G}:=\fs_S(\mathrm{F}_3)$ for the remainder of this section.

\begin{proposition}\label[proposition]{ThRadical}
We have that $\mathcal{G}^{frc}=\{E_1, E_2, E_3^S, S\}$. In particular, $\mathcal{E}(\mathcal{G})=\{E_1, E_2, E_3^S\}$.
\end{proposition}
\begin{proof}
This follows from \cite{ThSub}.
\end{proof}

We appeal to MAGMA (see \cref{code}) for the following result.

\begin{proposition}\label[proposition]{F3Essen}
Suppose that $\fs$ is saturated fusion system on $S$. Then $\mathcal{E}(\fs)\subseteq \{E_1, E_2, E_3^S\}$.
\end{proposition}

We will need the following observation in the proofs of the coming results.

\begin{lemma}\label[lemma]{E3inE1}
Let $\fs$ be a saturated fusion system on $S$. Then every $\fs$-conjugate of $E_3$ is contained in $E_1$ and not contained in $E_2$. Moreover, if $E_3\in\mathcal{E}(\fs)$, then $E_1\in\mathcal{E}(\fs)$ and $O_3(\fs)=\{1\}$.
\end{lemma}
\begin{proof}
Since $\{E_3^\fs\}=\{E_3^S\}$ and both $E_1$ and $E_2$ are normal in $S$, it suffices to show that $E_3\le E_1$ and $E_3\not\le E_2$. To this end, we note that $[Z_2(S), E_3]=\{1\}$. One can see this in $\mathcal{G}$ for otherwise, since $E_3$ is elementary abelian, we would have that $Z_2(S)\not\le E_3$ and $[Z_2(S), E_3]\le Z(S)$, a contradiction since $\Out_{\mathcal{G}}(E_3)\cong\SL_2(9).2$ has no non-trivial modules exhibiting this behaviour. If $E_3\le E_2$, then since $E_2\normaleq S$, we would have that $E_1=\langle E_3^S\rangle\le E_2$, a clear contradiction. Thus, $E_3\not\le E_2$.

Now, $\Phi(E_1)$ is elementary abelian of order $3^5$ and is not contained in $E_3$. Furthermore, $[E_3,\Phi(E_1)]\le [E_1, \Phi(E_1)]=Z_2(S)\le E_3$ so that $\Phi(E_1)\le N_S(E_3)$. Comparing with $\mathcal{G}$, we get that $N_S(E_3)=E_3\Phi(E_1)$, $E_3\cap \Phi(E_1)$ is of order $3^3$ and $\Phi(E_1)$ induces an FF-action on $E_3$. 

Assume that $E_3$ is essential in $\fs$. Then for $L:=O^{3'}(\Aut_{\fs}(E_3))$, applying \cref{SEFF}, we have that $L\cong \SL_2(9)$ and $E_3=[E_3, L]\times C_{E_3}(L)$, where $[E_3, L]$ is a natural $\SL_2(9)$-module. It follows that $[\Phi(E_1), E_3]=Z_2(S)$ has order $9$ and that $C_{E_3}(L)\cap Z_2(S)=\{1\}$. Let $K$ be a Sylow $2$-subgroup of $N_L(\Aut_S(E_3))$ so that $K$ is cyclic of order $8$ and acts irreducibly on $Z_2(S)$. Then the morphisms in $K$ lift to a larger subgroup of $S$ by the extension axiom and if $E_1$ is not essential then, using \cref{F3Essen} and the Alperin--Goldschmidt theorem, the morphisms in $K$ must lift to automorphisms of $S$. But then, upon restriction, the morphisms in $K$ would normalize $Z(S)$, contradicting the irreducibility of $Z_2(S)$ under the action of $K$. Hence, $E_1\in\mathcal{E}(\fs)$. Since $O_3(\fs)\normaleq S$ and, by \cref{normalinF}, $O_3(\fs)$ is an $\Aut_{\fs}(E_3)$-invariant subgroup of $E_3$, we conclude that $O_3(\fs)=\{1\}$.
\end{proof}

Throughout the remainder of this section, we set \[\mathcal{H}=\langle \Aut_{\mathcal{G}}(E_1), \Aut_{\mathcal{G}}(E_2), \Aut_{\mathcal{G}}(S)\rangle_S\] and \[\mathcal{D}=\langle \Aut_{\mathcal{G}}(E_1), \Aut_{\mathcal{G}}(E_3), \Aut_{\mathcal{G}}(S)\rangle_S.\]

\begin{proposition}\label[proposition]{HExotic1}
$\mathcal{H}$ is a saturated fusion system with $\mathcal{H}^{frc}=\{E_1, E_2, S\}$.
\end{proposition}
\begin{proof}
Set $\mathcal{H}^*:=\langle \mathcal{H}, H_{\mathcal{G}}(E_3)\rangle_S$, the saturated fusion system determined by applying \cref{Pruning} to $\mathcal{G}$ with $P=E_3$ and $K=H_\fs(E_3)$. Then $\mathcal{H}^*$ is saturated and $E_3\not\in\mathcal{E}(\mathcal{H}^*)$ by \cref{Pruning}. Hence, by \cref{F3Essen}, $ \mathcal{E}(\mathcal{H}^*)=\{E_1, E_2\}$ and the Alperin--Goldschmidt theorem implies that $\mathcal{H}^*=\langle \Aut_{\mathcal{G}}(E_1), \Aut_{\mathcal{G}}(E_2), \Aut_{\mathcal{G}}(S)\rangle_S=\mathcal{H}$ is saturated. 

Let $R$ be a fully $\mathcal{H}$-normalized, radical, centric subgroup of $S$ not equal to $E_1, E_2$ or $S$. Then $R$ must be contained in an $\mathcal{H}$-essential subgroup for otherwise, by the extension axiom and the Alperin--Goldschmidt theorem, we infer that $\Out_S(R)\normaleq \Out_{\mathcal{H}}(R)$ and $R$ is not $\mathcal{H}$-radical. If $R$ is contained in a $\mathcal{G}$-conjugate of $E_3$, $A$ say, then since $R$ is $\mathcal{H}$-centric, $R=A$. Then $\Out_S(R)\le O^{3'}(\Out_{\mathcal{H}}(R))\le O^{3'}(\Out_{\mathcal{G}}(R))\cong \SL_2(9)$. Since $R$ is not $\mathcal{H}$-essential, it follows that $O^{3'}(\Out_{\mathcal{H}}(R))$ is contained in the unique maximal subgroup of $O^{3'}(\Out_{\mathcal{G}}(R))$ which contains $\Out_S(R)$ and so $\Out_S(R)\normaleq O^{3'}(\Out_{\mathcal{H}}(R))$. Then the Frattini argument implies that $\Out_S(R)\normaleq \Out_{\mathcal{H}}(R)$, a contradiction. Thus, $R$ is not contained in an $S$-conjugate of $E_3$. Hence, by the Alperin--Goldschmidt theorem and using \cref{F3Essen}, since $\mathcal{H}=\langle \Aut_{\mathcal{G}}(E_1), \Aut_{\mathcal{G}}(E_2), \Aut_{\mathcal{G}}(S)\rangle_S$ and $R$ is fully $\mathcal{H}$-normalized, $R$ is fully $\mathcal{G}$-normalized and so is $\mathcal{G}$-centric. Finally, since $O_3(\Out_{\mathcal{G}}(R))\le O_3(\Out_{\mathcal{H}}(R))=\{1\}$, we conclude that $R$ is $\mathcal{G}$-centric-radical and comparing with \cref{ThRadical}, we have a contradiction.
\end{proof}

\begin{proposition}\label[proposition]{HExotic2}
$\mathcal{H}$ is simple.
\end{proposition}
\begin{proof}
Assume that $\mathcal{N}\normaleq \mathcal{H}$ and $\mathcal{N}$ is supported on $T$. Then $T$ is a strongly closed subgroup of $\mathcal{H}$. In particular, $T\normaleq S$ and $Z(S)\le T$. Taking repeated normal closures of $Z(S)$ under the actions of $\Aut_{\mathcal{G}}(E_1)$ and $\Aut_{\mathcal{G}}(E_2)$, we apply the description of $\mathrm{F}_3$ from \cite[pg 100]{Greenbook} to ascertain that $\Phi(E_1)\le T\not\le E_1$. The $E_1=\langle [T, E_1]^{\Aut_{\mathcal{G}}(E_1)}\rangle\le T$ and so $S=T$. Since $\Aut_{\mathcal{H}}(S)$ is generated by lifted morphisms from $O^{3'}(\Aut_{\mathcal{H}}(E_1))$ and $O^{3'}(\Aut_{\mathcal{H}}(E_2))$, in the language of \cref{p'lemma} we have that $\Aut_{\mathcal{H}}^0(S)=\Aut_{\mathcal{H}}(S)$. Then \cite[Theorem II.9.8(d)]{ako} implies that $\mathcal{H}$ is simple.
\end{proof}

\begin{proposition}\label[proposition]{HExotic}
$\mathcal{H}$ is exotic.
\end{proposition}
\begin{proof}
Aiming for a contradiction, suppose that $\mathcal{H}=\fs_S(G)$ for some finite group $G$ with $S\in\syl_3(G)$. We may as well assume that $O_3(G)=O_{3'}(G)=\{1\}$ so that $F^*(G)=E(G)$ is a direct product of non-abelian simple groups, all divisible by $3$. Since $\fs_{S\cap F^*(G)}(F^*(G))\normaleq \mathcal{H}$, we have that $G=F^*(G)$. Furthermore, since $|\Omega_1(Z(S))|=3$, we deduce that $G$ is simple. In particular, we reduce to searching for simple groups with a Sylow $3$-subgroup of order $3^{10}$ and $3$-rank $5$. Since $E_3$ is not normal in $S$, $S$ does not have a unique elementary abelian subgroup of maximal rank.

If $G\cong \Alt(n)$ for some $n$ then $m_3(\Alt(n))=\lfloor\frac{n}{3}\rfloor$ by \cite[Proposition 5.2.10]{GLS3} and so $n<18$. But a Sylow $3$-subgroup of $\Alt(18)$ has order $3^8$ and so $G\not\cong\Alt(n)$ for any $n$. If $G$ is isomorphic to a group of Lie type in characteristic $3$, then comparing with \cite[Table 3.3.1]{GLS3}, we see that the groups with a Sylow $3$-subgroup which has $3$-rank $5$ are $\PSL_2(3^5)$, $\Omega_7(3)$, ${}^3\mathrm{D}_4(3)$ and $\PSU_5(3)$, and only $\PSU_5(3)$ has a Sylow $3$-subgroup of order $3^{10}$ of these examples. Since the unipotent radicals of parabolic subgroups of $\PSU_5(3)$ are essential subgroups and since neither has index $3$ in a Sylow $3$-subgroup, we have shown that $G$ is not a group of Lie type of characteristic $3$.

Assume now that $G$ is a group of Lie type in characteristic $r\ne 3$. By \cite[Theorem 4.10.3]{GLS3}, $S$ has a unique elementary abelian subgroup of $3$-rank $5$ unless $G\cong\mathrm{G}_2(r^a), {}^2\mathrm{F}_4(r^a), {}^3\mathrm{D}_4(r^a), \PSU_3(r^a)$ or $\PSL_3(r^a)$.  Moreover, by \cite[Theorem 4.10.2]{GLS3}, there is a normal abelian subgroup $S_T$ of $S$ such that $S/S_T$ is isomorphic to a subgroup of the Weyl group of $G$. But $|S_T|\leq 3^5$ so that $|S/S_T|\geq 3^5$. All of the candidate groups above have Weyl group with $3$-part strictly less than $3^5$ and so $G$ is not isomorphic to a group of Lie type in characteristic $r$.

Finally, checking the orders of the sporadic groups, we have that $\mathrm{F}_3$ is the unique sporadic simple group with a Sylow $3$-subgroup of order $3^{10}$. Since the $3$-fusion category of $\mathrm{F}_3$ has $3$ classes of essential subgroups, $G\not\cong\mathrm{F}_3$ and we have a final contradiction. Hence, $\mathcal{H}$ is exotic.
\end{proof}

\begin{proposition}
$\mathcal{D}$ is a saturated fusion system with $\mathcal{D}^{frc}=\{E_1, E_3^\mathcal{D}, S\}$.
\end{proposition}
\begin{proof}
In the statement of \cref{JasonAdd}, letting $\mathcal{F}_0=N_{\mathcal{G}}(E_1)$, $V=E_3$ and $\Delta=\Aut_{\mathcal{G}}(E_3)$ we have that $\mathcal{D}^\dagger=\langle \fs_0, \Aut_{\mathcal{G}}(E_3)\rangle_S$ is a proper saturated subsystem of $\mathcal{G}$. Since $E_1$ is characteristic in $S$, we have that $\Aut_{\mathcal{D}^\dagger}(S)=\Aut_{\mathcal{G}}(S)$ and $\Aut_{\mathcal{D}^\dagger}(E_1)=\Aut_{\mathcal{G}}(E_1)$. Since $\Aut_{\mathcal{D}^\dagger}(E_2)\le \Aut_{\mathcal{G}}(E_2)$, if $E_2$ was $\mathcal{D}^\dagger$-essential, then $\mathcal{D}^\dagger=\mathcal{G}$, a contradiction.  Therefore, by the Alperin--Goldschmidt theorem, and using \cref{F3Essen}, we have that $\mathcal{D}^\dagger=\langle \Aut_{\mathcal{G}}(E_1), \Aut_{\mathcal{G}}(E_3), \Aut_{\mathcal{G}}(S)\rangle_S=\mathcal{D}$ is saturated. 

Let $R$ be a fully $\mathcal{D}$-normalized, radical, centric subgroup of $S$ not equal to $E_1$, $S$ or a $\mathcal{D}$-conjugate of $E_3$. If $R$ is contained in a $\mathcal{D}$-conjugate of $E_3$, then since $R$ is $\mathcal{D}$-centric and $E_3$ is elementary abelian, we have a contradiction. Hence $R$ is not contained in a $\mathcal{D}$-conjugate of $E_3$ and by \cref{F3Essen} and using that $E_2\not\in\mathcal{E}(\mathcal{D})$, $R$ is contained in at most one $\mathcal{D}$-essential, namely $E_1$. Then, as $E_1$ is $\Aut_{\mathcal{D}}(S)$-invariant, the extension axiom and the Alperin--Goldschmidt theorem imply that $\Out_{E_1}(R)\normaleq \Out_{\mathcal{D}}(R)$ and $E_1\le R\le S$, a contradiction.
\end{proof}

\begin{lemma}\label[lemma]{E1SC3}
$E_1$ is the unique proper non-trivial strongly closed subgroup of $\mathcal{D}$.
\end{lemma}
\begin{proof}
Since every essential subgroup of $\mathcal{D}$ is contained in $E_1$, and since $E_1$ is characteristic in $S$, we deduce that $E_1$ is strongly closed in $\mathcal{D}$. Assume that $T$ is any proper non-trivial strongly closed subgroup of $\mathcal{D}$. Then $T\normaleq S$ and so $Z(S)\le T$ and $Z_2(S)=\langle Z(S)^{\Aut_{\mathcal{D}}(E_1)}\rangle\le T$. Suppose first that $T\cap \Phi(E_1)=Z_2(S)$. Since $\Phi(E_1)\normaleq S$ we have that $[\Phi(E_1), T]=Z_2(S)$ so that $T\le E_1$. But then $[E_1, T]\le \Phi(E_1)\cap T=Z_2(S)=Z(E_1)$ and $T\le Z_2(E_1)=\Phi(E_1)$ so that $T=Z_2(S)$. However, then $T<\langle T^{\Aut_{\mathcal{D}}(E_3)}\rangle$, a contradiction.

Thus, $T\cap \Phi(E_1)>Z_2(S)$ and since ${\Out_{\mathcal{D}}(E_1)}$ acts irreducibly on $\Phi(E_1)/Z_2(S)$, we must have that $\Phi(E_1)\le T$. But now $E_3=\langle (\Phi(E_1)\cap E_3)^{\Aut_{\mathcal{D}}(E_3)}\rangle\le \langle (T\cap E_3)^{\Aut_{\mathcal{D}}(E_3)}\rangle\le T$. Finally, since $E_1=\langle E_3^S\rangle\le T$, we deduce that $T=E_1$, as desired.
\end{proof}

\begin{proposition}
$\mathcal{D}$ is a saturated exotic simple fusion system.
\end{proposition}
\begin{proof}
We note that $O^{3'}(\Out_{\mathcal{D}}(E_1))\cong \SL_2(3)$ and the extension axiom yields that $\Aut_{\mathcal{D}}^0(S)$ has index at most $2$ in $\Aut_{\mathcal{D}}(S)$. Suppose $\Aut_{\mathcal{D}}^0(S)$ has index exactly $2$ in $\Aut_{\mathcal{D}}(S)$. Then, since $\Out_{\mathcal{D}}(E_1)\cong \GL_2(3)$, an application of the extension axiom yields that $\Out_{O^{3'}(\mathcal{D})}(E_1)\cong \SL_2(3)$. Let $K$ be a Sylow $2$-subgroup of $N_{O^{3'}(\Aut_{\mathcal{D}}(E_3))}(\Aut_S(E_3))$ which is cyclic of order $8$ and contained in $O^{3'}(\mathcal{D})$. Then, by the extension axiom, the morphisms in $K$ lift to morphisms of larger subgroups of $S$ and as $E_1$ is $\Aut_{O^{3'}(\mathcal{D})}(S)$-invariant, and applying the Alperin--Goldschmidt theorem, we deduce that the morphisms in $K$ lift to morphisms in $\Aut_{O^{3'}(\mathcal{D})}(E_1)$. Hence, $\Out_{O^{3'}(\mathcal{D})}(E_1)$ contains a cyclic group of order $8$. Since $\Out_{O^{3'}(\mathcal{D})}(E_1)\cong \SL_2(3)$, this is a contradiction. Thus $\Aut_{\mathcal{D}}^0(S)=\Aut_{\mathcal{D}}(S)$ and applying \cref{p'lemma} we must have that $\mathcal{D}=O^{3'}(\mathcal{D})$.

Applying \cite[Theorem II.9.8(d)]{ako}, if $\mathcal{D}$ is not simple with $\mathcal{N}\normaleq \mathcal{D}$ then by \cref{E1SC3} we have that $\mathcal{N}$ is supported on $E_1$. Then by \cite[Proposition I.6.4]{ako}, $\Aut_{\mathcal{N}}(E_1)\normaleq \Aut_{\mathcal{D}}(E_1)$ so that $\Out_{\mathcal{N}}(E_1)$ is isomorphic to a normal $3'$-subgroup of $\Out_{\mathcal{D}}(E_1)\cong \GL_2(3)$ and hence is a subgroup of the quaternion group of order $8$. In particular, $E_3$ is not essential in $\mathcal{N}$ for otherwise we could again lift a cyclic subgroup of order $8$ to $\Aut_{\mathcal{N}}(E_1)$, using the extension axiom. Then, we apply \cref{EssenDeterD3} (or just perform the MAGMA calculation on which this relies) to deduce that $\mathcal{E}(\mathcal{N})=\emptyset$ and $E_1=O_3(\mathcal{N})$, and so $E_1\normaleq \mathcal{D}$, a contradiction by \cref{normalinF}.

Since $\mathcal{D}$ is a simple fusion system which contains a non-trivial proper strongly closed subgroup, we deduce by \cref{SCExotic} that $\mathcal{D}$ is exotic.
\end{proof}

It feels prudent at this point to draw comparisons with some of the other exotic fusion systems already documented in the literature. We remark that the set of essentials $\{E_3^{\mathcal{D}}\}$ in some ways behave similarly to \emph{pearls} as defined in \cite{grazian}, or the extensions of pearls as found in \cite{oliver}. In some ways, our class $\{E_3^{\mathcal{D}}\}$ motivates an examination of a generalization of pearls to \emph{$q$-pearls} $P$ where $O^{p'}(G)\cong q^2:\SL_2(q)$ for $G$ some model of $N_{\fs}(P)$ and $q=p^n$, as in \cite{ClellandExo}.

Perhaps one can investigate an even further generalization where we need only stipulate that $O^{p'}(G)/Z(O^{p'}(G))\cong q^2:\SL_2(q)$ and we allow for $Z(O^{p'}(G))\ne\{1\}$. All of these cases are linked with \emph{pushing up} problems more familiar in local group theory, and we speculate that all of these examples are special cases of a more general phenomenon in this setting.

We also record the following interesting observation. As shown in \cite[Theorem 3.6]{grazian}, $p$-pearls are \emph{never} contained in any larger essential subgroups, in direct contrast to situation in the fusion system $\mathcal{D}$. Perhaps the fusion systems where there is a class of $q$-pearls contained in a strictly larger essential subgroup have a more rigid structure and so may be organized in some suitable fashion.

We now delve into the study of all saturated fusion systems on $S$ and throughout the remainder of this section, we let $\fs$ be a saturated fusion system on $S$. As in the study of $\mathrm{Co}_1$, we first limit the possible combinations of essentials we have in a saturated fusion system supported on $S$, as well as the potential automizers.

\begin{lemma}\label[lemma]{3Conjugate}
Suppose that $\fs$ is a saturated fusion system on $S$ with $E_1\in\mathcal{E}(\fs)$. Then $\Aut_{\fs}(E_1)$ is $\Aut(E_1)$-conjugate to a subgroup of $\Aut_{\mathcal{G}}(E_1)$ and $O^{3'}(\Out_{\fs}(E_1))\cong \SL_2(3)$.
\end{lemma}
\begin{proof}
Since $Z(E_1)$ has order $9$, and from the actions present in $\mathrm{F}_3$, we deduce that $\Aut(E_1)/C_{\Aut(E_1)}(Z(E_1))\cong \GL_2(3)$. Indeed, we calculate (see \cref{code}) that $|\Aut(E_1)|_{3'}=16$ so that $C_{\Aut(E_1)}(Z(E_1))$ is a $3$-group. It follows that $\Out_{\fs}(E_1)$ is isomorphic to a subgroup of $\GL_2(3)$ which contains a strongly $3$-embedded subgroup and so $\Out_{\fs}(E_1)\cong \SL_2(3)$ or $\GL_2(3)$. Indeed, $\Out_{\fs}(E_1)$ is normal in a subgroup isomorphic to $\GL_2(3)$. We calculate that there is a two conjugacy classes of subgroups of $\Aut(E_1)$ containing $\Inn(E_1)$ whose quotient by $\Inn(E_1)$ is isomorphic to $\SL_2(3)$. Moreover, $\Aut_S(E_1)$ is a subgroup of a conjugate of exactly one of these classes (see \cref{code}). Since $\Aut_S(E_1)\le \Aut_{\fs}(E_1)\cap \Aut_{\mathcal{G}}(E_1)$, we conclude that $\Aut_{\fs}(E_1)$ is $\Aut(E_1)$-conjugate to a subgroup of $\Aut_{\mathcal{G}}(E_1)$.
\end{proof}

The following lemma uses several facts about the group $E_2$. These may be gleaned from \cite[Section 13]{Greenbook} ($E_2=Q_\beta$, $Z(\Phi(E_2))=V_\beta$ $C_2=C_\beta$ and $W_2=W_\beta$) but are also computed explicitly in \cref{code}.

\begin{lemma}\label[lemma]{F3SL2}
Suppose that $\fs$ is a saturated fusion system on $S$ with $E_2\in\mathcal{E}(\fs)$. Set $C_2:=C_{E_2}(Z_3(S))$ and $W_2:=C_{E_2}([E_2, C_2])$. Then $Z_3(S)=Z_2(E_2)$, $|W_2|=3^6$, $|Z(W_2)|=3^4$ and $O^{3'}(\Out_{\fs}(E_2))\cong \SL_2(3)$ acts irreducibly on $W_2/Z(W_2)$.
\end{lemma}
\begin{proof}
We calculate the following in MAGMA (see \cref{code}). We have that $Z(S)=Z(E_2)$ has order $3$ and $Z_3(S)=Z_2(E_2)$ has order $3^3$. Moreover, $C_2$ has order $3^7$ and so has index $3^2$ in $E_2$. We have $Z(W_2)=[E_2, C_2]$ has order $3^4$, $W_2$ has order $3^6$ and $Z_2(E_2)< Z(W_2)=C_{E_2}(W_2)$. Finally, we have that $C_2=C_{E_2}(W_2/Z_2(E_2))$. It remains to prove that $O^{3'}(\Out_{\fs}(E_2))\cong \SL_2(3)$ acts irreducibly on $Q_2/C_2$, $W_2/Z(W_2)$ and $Z_2(E_2)/Z(E_2)$.

We observe that as $Z_2(E_2)\le Z(W_2)$, we must have that $W_2\le C_2$. Then $|C_2/W_2|=|Z(W_2)/Z_2(E_2)|=|Z(E_2)|=3$ and so $O^{3'}(\Out_{\fs}(E_2))$ centralizes each of these chief factors. We note that $[E_2, W_2]\le [E_2, C_2]=Z(W_2)$. Let $R:=C_{O^{3'}(\Out_{\fs}(E_2))}(W_2/Z(W_2))\normaleq \Out_{\fs}(E_2)$. Assume that $R$ is non-trivial, and so as $O^{3'}(\Out_{\fs}(E_2))$ has a strongly $3$-embedded subgroup, there is $r\in R$ of $3'$-order. Then $[r, W_2]\le Z(W_2)$ and as $O^{3'}(\Out_{\fs}(E_2))$ centralizes $Z(W_2)/Z_2(E_2)$, we have that $[r, Z(W_2)]\le Z_2(E_2)$ and by coprime action we deduce that $[r, W_2]\le Z_2(E_2)$. We have that $[E_2, W_2]\le [E_2, C_2]=Z(W_2)$ and so $[E_2, W_2, r]\le Z_2(E_2)$ and so by the three subgroups lemma, we have that $[r, E, W_2]\le Z_2(E_2)$. But $C_{E_2}(W_2/Z_2(E_2))=C_2$ and so we deduce that $[r, E_2]\le C_2$. 

Again, as $O^{3'}(\Out_{\fs}(E_2))$ centralizes $C_2/W_2$, we have that $[r, C_2]\le W_2$ and $[r, W_2]\le Z_2(E_2)$, and by coprime action we deduce that $[r, E_2]\le Z_2(E_2)$. Hence, $[r, E_2, C_2]=\{1\}$, $[C_2, r, E_2]\le [Z_2(E_2), E_2]=Z(E_2)$ and by the three subgroups lemma we conclude that $[E_2, C_2, r]\le Z(E_2)$. But $[E_2, C_2]=Z(W_2)\ge Z_2(E_2)$ and so we ascertain that $[Z_2(E_2), r]\le Z(E_2)$ and as $r$ centralizes $Z(E_2)$, a final application of coprime action yields that $[r, E_2]=\{1\}$, a contradiction since $r$ is non-trivial. Hence, $R=\{1\}$ and $O^{3'}(\Out_{\fs}(E_2))$ acts faithfully on $W_2/Z(W_2)$. We conclude that $W_2/Z(W_2)$ is a natural module for $O^{3'}(\Out_{\fs}(E_2))\cong \SL_2(3)$.

Since $|Q_2/C_2|=|Z_2(E_2)/Z(E_2)|$, to complete the proof it remains to show that $1\ne r\in Z(O^{3'}(\Out_{\fs}(E_2)))$ acts non-trivially on $Q_2/C_2$ and $Z_2(E_2)/Z(E_2)$. Note that if $[r, Z_2(E_2)]\le Z(E_2)$ then by coprime action, $[r, Z_2(E_2)]=\{1\}$. An application of the three subgroups lemma would then yield that $[r, E_2]\le C_2$. Hence, it suffices to demonstrate that $[r, E_2]\not\le C_2$. Assume otherwise. Since $[r, C_2]\le W_2$, by coprime action we have that $[r, E_2]\le W_2$. Then $[r, E_2, W_2]\le [W_2, W_2]\le [C_2, W_2]=Z_2(E_2)$. Moreover, $[E_2, W_2, r]=[Z(W_2), r]\le Z_2(E_2)$. By the three subgroups lemma, we infer that $[r, W_2, E_2]\le Z_2(E_2)$. But $W_2=[r, W_2]Z(W_2)$ and $[E_2, Z(W_2)]\le Z_2(E_2)$ so that $[W_2, E_2]\le Z_2(E_2)$, a contradiction as $C_2=C_{E_2}(W_2/Z_2(E_2))$. Hence, $r$ acts non-trivially on $Q_2/C_2$, which completes the proof.
\end{proof}

\begin{lemma}\label[lemma]{3Triv}
Suppose that $\fs$ is a saturated fusion system on $S$ such that $\{E_1, E_2\}\subseteq \mathcal{E}(\fs)$. Then $O_3(\fs)=\{1\}$.
\end{lemma}
\begin{proof}
Assume that $\fs$ is a saturated fusion system on $S$ such that $\{E_1, E_2\}\subseteq \mathcal{E}(\fs)$ and suppose that $\{1\}\ne Q\normaleq \fs$. By \cref{normalinF}, we have that $Q\le E_1\cap E_2$. Then $Z(S)\le Q$ and as $\Out_{\fs}(E_1)$ acts irreducibly on $Z(E_1)$ by \cref{3Conjugate}, we deduce that $Z(E_1)=Z_2(S)\le Q$. Now, there is $t_2\in O^{3'}(\Aut_{\fs}(E_2))$ an involution which lifts to a morphism of $S$ by the extension axiom. Indeed, $t_2$ is such that $t_2\Inn(E_2)/\Inn(E_2)\le Z(O^{3'}(\Out_{\fs}(E_2)))$. Hence, the lift of $t_2$ to $\Aut(S)$ restricts to an automorphism of $E_1$. Since $t_2$ is non-trivial on $Z_2(S)$ but acts trivially on $Z(S)$, upon restriction, the image of this automorphism of $E_1$ together with $O^{3'}(\Out_{\fs}(E_1))\cong \SL_2(3)$ generates a group isomorphic to $\GL_2(3)$. Hence, by \cref{3Conjugate}, $\Out_{\fs}(E_1)$ is $\Aut(E_1)$-conjugate to $\Aut_{\mathcal{G}}(E_1)$ and so acts irreducibly on $\Phi(E_1)/Z_2(S)$. Thus, $\Phi(E_1)\le Q \le E_1\cap E_2$. 

Now, if $\Phi(O_3(\fs))$ is non-trivial then by the above argument we have that $\Phi(E_1)\le \Phi(O_3(\fs)) \le O_3(\fs) \le E_1\cap E_2$. But $\Phi(O_3(\fs))\le \Phi(\Phi(E_1\cap E_2))\le\Phi(E_1)$, and we conclude that $\Phi(E_1)\normaleq \fs$. If $\Phi(O_3(\fs))=\{1\}$ and $O_3(\fs)\ne\{1\}$ then $O_3(\fs)$ is elementary abelian and contains $\Phi(E_1)$, and since $\Phi(E_1)$ is elementary abelian of maximal order, the only possibility is that $O_3(\fs)=\Phi(E_1)$. Either way $\Phi(E_1)\normaleq \fs$. But in the language of \cref{F3SL2}, $Z(W_2)<\Phi(E_1)<W_2$ and $\Out_{\fs}(E_2)$ acts irreducibly on $W_2/Z(W_2)$, a contradiction.
\end{proof}

\begin{proposition}\label[proposition]{F31E}
Suppose that $\fs$ is a saturated fusion system on $S$ such that $O_3(\fs)\ne\{1\}$. Then either
\begin{enumerate}
    \item $\fs=N_{\fs}(S)$; or
    \item $\fs=N_{\fs}(E_i)$ where $O^{3'}(\Out_{\fs}(E_i))\cong \SL_2(3)$ for $i\in\{1,2\}$.
\end{enumerate}
\end{proposition}
\begin{proof}
If $\mathcal{E}(\fs)=\emptyset$, then outcome (i) is satisfied by the Alperin--Goldschmidt theorem. Thus, by \cref{E3inE1} and \cref{3Triv}, we may assume that $E_i$ is the unique essential subgroup of $\fs$ and apply \cref{3Conjugate} and \cref{F3SL2}.
\end{proof}

\begin{lemma}\label[lemma]{E3Unique}
Suppose that $\fs_1, \fs_2$ are two saturated fusion systems supported on $T$ where $E_1\le T\le S$. If $E_3\in\mathcal{E}(\fs_1)\cap \mathcal{E}(\fs_2)$ and $N_{\fs_1}(E_1)=N_{\fs_2}(E_1)$ then $\Aut_{\fs_1}(E_3)=\Aut_{\fs_2}(E_3)$.
\end{lemma}
\begin{proof}
 By \cref{E3inE1}, we have that $O^{3'}(\Aut_{\fs_i}(E_3))\cong\SL_2(9)$ for $i\in\{1,2\}$. Write $X:=O^{3'}(\Aut_{\fs_1}(E_3))$ and $Y:=O^{3'}(\Aut_{\fs_2}(E_3))$. Set $K:=N_{\Aut_{\fs_1}(E_3)}(\Aut_T(E_3))$ so that, by the extension axiom, all morphisms in $K$ lift to morphisms in $\Aut_{\fs_1}(E_1)=\Aut_{\fs_2}(E_1)$. In particular, \[K=N_{\Aut_{N_{\fs_1}(E_1)}(E_3)}(\Aut_T(E_3))=N_{\Aut_{N_{\fs_2}(E_1)}(E_3)}(\Aut_T(E_3))=N_{\Aut_{\fs_2}(E_3)}(\Aut_T(E_3)).\]
 
Let $L$ be the unique cyclic subgroup of order $8$ of a fixed Sylow $2$-subgroup of $K$ arranged such that $K_L:=L\Aut_S(E_3)=N_{O^{3'}(\Aut_{\fs_1}(E_3))}(\Aut_S(E_3))$. Then $K_L\le X\cap Y\le \Aut(E_1)\cong\GL_5(3).$ We record that there is a unique conjugacy class of subgroups isomorphic to $\SL_2(9)$ in $\GL_5(3)$ (see \cref{code}). Hence, there is $g\in\Aut(E_3)$ with $Y=X^g$.

Then $K_L, (K_L)^g\le Y$ and so there is $y\in Y$ such that $(K_L)^g=(K_L)^y$. Thus, we have that $X^{gy^{-1}}=X^g$ and we calculate that $gy^{-1}\le N_{\GL_5(3)}(K_L)\le N_{\GL_5(3)}(X)$ (see \cref{code}). But then $X=X^g=Y$. By a Frattini argument, $\Aut_{\fs_1}(E_3)=XK=YK=\Aut_{\fs_2}(E_3)$.
\end{proof}

\begin{theorem}\label[theorem]{ThmB2}
Suppose that $\fs$ is saturated fusion system on $S$ such that $O_3(\fs)=\{1\}$. If $E_2\not\in\mathcal{E}(\fs)$ then $\fs\cong \mathcal{D}$.
\end{theorem}
\begin{proof}
Suppose that $E_2\not\in\mathcal{E}(\fs)$. Since $O_3(\fs)=\{1\}$ and $E_1$ is $\Aut_{\fs}(S)$-invariant, an application of \cref{normalinF} using \cref{F3Essen} implies that some $E_3^\fs\cap \mathcal{E}(\fs)\ne \emptyset$. Hence by \cref{E3inE1}, we have that $E_3\in\mathcal{E}(\fs)$ and $O^{3'}(\Aut_{\fs}(E_3))\cong \SL_2(9)$. Then for $k$ an element of order $8$ in $N_{O^{3'}(\Aut_{\fs}(E_3))}(\Aut_S(E_3))$, by the extension axiom and since $E_1$ is $\Aut_{\fs}(S)$-invariant, $k$ lifts to an element of order $8$ in $\Aut(E_1)$. Now, by \cref{3Conjugate}, we have that $O^{3'}(\Out_{\fs}(E_1))\cong \SL_2(3)$ and since $\SL_2(3)$ has no cyclic subgroups of order $8$, the Sylow $2$-subgroups of $\Aut_{\fs}(E_1)$ have order at least $16$. We calculate that $|\Aut(E_1)|_{3'}=16$ so that $\Aut_{\fs}(E_1)$ contains a Sylow $2$-subgroup of $\Aut(E_1)$, and $\Out_{\fs}(E_1)\cong \GL_2(3)$.

Let $t$ be an element of $\Aut(S)$ of order coprime to $3$. Since $E_1$ is $\Aut(S)$-invariant, $t$ normalizes $E_1$. Since $E_1$ is self-centralizing in $S$, an application of the three subgroups lemma and coprime action reveals that $t$ acts non-trivially on $E_1$. Hence, a Hall $3'$-subgroup of $\Aut(S)$ restricts faithfully to $N_{\Aut(E_1)}(\Aut_S(E_1))$. As in \cref{3conjugate}, since $Z(E_1)$ has order $9$ and from the actions present in $\mathrm{F}_3$, we conclude that $\Aut(E_1)/C_{\Aut(E_1)}(Z(E_1))\cong \GL_2(3)$ and $C_{\Aut(E_1)}(Z(E_1))$ is a normal $3$-subgroup of $\Aut(E_1)$. Now, a Hall $3$-subgroup of $\Aut(S)$ also normalizes $Z(E_1)$ and so it restricts faithfully to $N_{\Aut(E_1)}(\Aut_S(E_1)C_{\Aut(E_1)}(Z(E_1)))$, and as $\Aut(E_1)/C_{\Aut(E_1)}(Z(E_1))\cong \GL_2(3)$ and $[\Aut_S(E_1), Z(E_1)]\ne \{1\}$, we conclude that a Hall $3'$-subgroup of $\Aut(S)$ is elementary abelian of order at most $4$. Since $N_{\Aut_{\fs}(E_1)}(\Aut_S(E_1))$ contains an elementary abelian subgroup of order $4$ which, by the extension axiom, lifts to $\Aut_{\fs}(S)\le \Aut(S)$ we conclude that a Hall $3'$-subgroup of $\Aut(S)$ and of $\Aut_{\fs}(S)$ is elementary abelian of order $4$. In particular, $\Out_{\fs}(S)$ in elementary abelian of order $4$.

By the Alperin--Goldschmidt theorem and using that $E_1$ is characteristic in $S$, we have that $\fs=\langle N_{\fs}(E_1), \Aut_{\fs}(E_3)\rangle_S$ and $\mathcal{D}=\langle N_{\mathcal{D}}(E_1), \Aut_{\mathcal{D}}(E_3)\rangle_S$. Hence, by \cref{E3Unique} to show that $\fs\cong \mathcal{D}$ it suffices to show that there is $\alpha\in\Aut(S)$ with $N_{\fs^\alpha}(E_1)=N_{\mathcal{D}}(E_1)$. 

Now, $\Aut_{\fs}(S)$ contains a Hall $3'$-subgroup of $\Aut(S)$. By Hall's theorem, there is $\alpha_1\in\Aut(S)$ such that $\Aut_{\fs^{\alpha_1}}(S)=\Aut_{\fs}(S)^{\alpha_1}=\Aut_{\mathcal{D}}(S)$. By the Alperin--Goldschmidt theorem, we see that $N_{\fs^{\alpha_1}}(S)=N_{\mathcal{D}}(S)$. Then 
\begin{align*}
K:=N_{\Aut_{\fs^{\alpha_1}}(E_1)}(\Aut_S(E_1))&=N_{\Aut_{N_{\fs^{\alpha_1}}(S)}(E_1)}(\Aut_S(E_1))\\
&=N_{\Aut_{N_{\mathcal{D}}(S)}(E_1)}(\Aut_S(E_1))\\
&=N_{\Aut_{\mathcal{D}}(E_1)}(\Aut_S(E_1)).
\end{align*}
We calculate that in $\Aut(E_1)$ there are three candidates for the group $\Aut_{\fs}(E_1)$ which contain $K$ appropriately and that there is an element which conjugates the three candidates and extends to an automorphism of $S$ which preserves the class $\{E_3^S\}$ (see \cref{code}). In particular, there is $\alpha_2\in \Aut(S)$ with $\Aut_{\fs^{\alpha_1\alpha_2}}(E_1)=\Aut_{\mathcal{D}}(E_1)$. Hence, by \cref{model} there is $\beta\in\Aut(S)$ with $\alpha:=\alpha_1\alpha_2\beta$ and $N_{\fs^\alpha}(E_1)=N_{\mathcal{D}}(E_1)$, as required.
\end{proof}

We are now in a position to determine all saturated fusion systems on $S$. For this, we recall the notion of an amalgam of type $\mathrm{F}_3$ from \cref{F3Unique}, referring to \cite[pg 100]{Greenbook} and \cite[pg 150 (a)-(f)]{F3} for the details. 

\begin{theorem}\label[theorem]{ThmB1}
Suppose that $\fs$ is saturated fusion system on $S$ such that $O_3(\fs)=\{1\}$. Then $\fs\cong \mathcal{D}, \mathcal{G}$ or $\mathcal{H}$.
\end{theorem}
\begin{proof}
By \cref{E3inE1} and \cref{F31E}, we have that $E_1\in\mathcal{E}(\fs)$. By \cref{ThmB2}, we may assume that that $\{E_1, E_2\}\subseteq \mathcal{E}(\fs)$ and form $\mathcal{T}:=\langle \Aut_{\fs}(E_1), \Aut_{\fs}(E_2), \Aut_{\fs}(S)\rangle_S$. If $E_3\in\mathcal{E}(\fs)$ then $\mathcal{T}$ is the $\fs$-analogue of $\mathcal{H}$ and the proof that $\mathcal{T}$ is saturated is the same as the proof that $\mathcal{H}$ is saturated, relying on \cref{Pruning}. If $E_3\not\in\mathcal{E}(\fs)$ then by the Alperin--Goldschmidt theorem we have that $\fs=\mathcal{T}$. In either case, $\mathcal{E}(\mathcal{T})=\{E_1, E_2\}$ and $O_3(\mathcal{T})=\{1\}$ by \cref{3Triv}. 

Let $G_i$ be a model for $N_{\fs}(E_i)$ and since $E_i$ is $\Aut_{\fs}(S)$-invariant, by the uniqueness of models provided by \cref{model}, we may embed the model for $N_{\fs}(S)$, which we denote $G_{12}$, into $G_i$ for $i\in\{1,2\}$. Applying \cite[Theorem 1]{Robinson}, we ascertain that $\mathcal{T}=\langle \fs_S(G_1), \fs_S(G_2)\rangle_S=\fs_S(G_1\ast_{G_{12}} G_2)$. Furthermore, the tuple $(G_1, G_2, G_{12})$ along with the appropriate injective maps forms an amalgam of type $\mathrm{F}_3$. By \cref{F3Unique}, this amalgamated product is determined up to isomorphism, and so $\mathcal{T}$ is unique up to isomorphism. In particular, $\mathcal{T}$ is the unique (up to isomorphism) saturated fusion system on $S$ with $O_3(\mathcal{T})=\{1\}$ and $\mathcal{E}(\mathcal{T})=\{E_1, E_2\}$. Since $\mathcal{H}$ satisfies these conditions, we must have that $\mathcal{T}\cong\mathcal{H}$.

We may as well assume now that $E_3\in\mathcal{E}(\fs)$ and by the Alperin--Goldschmidt theorem, that $\fs=\langle \mathcal{T}, \Aut_{\fs}(E_3)\rangle_S$. Indeed, replacing $\fs$ by $\fs^\alpha$ for some $\alpha\in\Aut(S)$, we have that $\fs=\langle \mathcal{H}, \Aut_{\fs}(E_3)\rangle_S$ and upon demonstrating that $\Aut_{\fs}(E_3)=\Aut_{\mathcal{G}}(E_3)$ we will have shown that $\fs=\mathcal{G}$. But $N_{\fs}(E_1)=N_{\mathcal{H}}(E_1)=N_{\mathcal{G}}(E_1)$ and so \cref{E3Unique} gives $\Aut_{\fs}(E_3)=\Aut_{\mathcal{G}}(E_3)$, as desired.
\end{proof}

We provide the following table summarizing the actions induced by the fusion systems described in \cref{ThmB1} on their centric-radical subgroups. The entry ``-" indicates that the subgroup is no longer centric-radical in the subsystem.

\begin{table}[H]
    \caption{$\mathcal{G}$-conjugacy classes of radical-centric subgroups of $S$}
        \begin{tabular}{|c|c|c|c|c|}\hline
            $P$ & $|P|$ & $\Out_{\mathcal{G}}(P)$ & $\Out_{\mathcal{H}}(P)$ & $\Out_{\mathcal{D}}(P)$ \\\hline
            $S$ & $3^{10}$ & $2\times 2$ & $2\times 2$ & $2\times 2$ \\\hline
            $E_1$ & $3^9$ & $\GL_2(3)$ & $\GL_2(3)$ & $\GL_2(3)$ \\\hline
            $E_2$ & $3^9$ & $\GL_2(3)$ & $\GL_2(3)$ & - \\\hline            
            $E_3$ & $3^5$ & $\SL_2(9).2$ & - & $\SL_2(9).2$\\ \hline
        \end{tabular}
    \label{F3Table}
\end{table}

We describe a pair of bonus exotic fusion systems related to the exotic system $\mathcal{D}$. Using that $E_1$ is characteristic in $S$, and applying the Alperin--Goldschmidt theorem in tandem with the extension axiom, the morphisms in $N_{\Aut_{\mathcal{D}}(E_3)}(\Aut_S(E_3))$ extend to a group of morphisms in $\Aut_{\mathcal{D}}(E_1)$ which we denote by $K$. Then $|K|_{3'}=16$. Let $G$ be a model for $N_{\mathcal{D}}(E_1)$ and let $H$ be a subgroup of $G$ chosen such that $\Aut_H(E_1)=K\Inn(E_1)$. In particular, $H$ is the product of $E_1$ with some Sylow $2$-subgroup of $G$. We define the subsystem 
\[\mathcal{D^*}:=\langle \Aut_{\mathcal{D}}(E_3), \fs_{E_1}(H)\rangle_{E_1} \le \mathcal{D}.\]

Note that the conjugacy class of $E_3$ in $S$ splits into three distinct classes upon restricting only to $E_1$. Indeed, in this way we have three choices for the construction of $\mathcal{D}^*$, corresponding to the three $E_1$-conjugacy classes of $S$-conjugates of $E_3$, which in turn correspond to the three choices of Sylow $2$-subgroups of $\Out_{\mathcal{D}}(E_1)$. Since the choice is induced by an element of $\Aut(E_1)$, all choices give rise to isomorphic fusion systems.

\begin{proposition}
$\mathcal{D^*}$ is saturated fusion system on $E_1$ and $O^{3'}(\mathcal{D^*})$ has index $2$ in $\mathcal{D^*}$.
\end{proposition}
\begin{proof}
We create $H$ as in the construction of $\mathcal{D}^*$ and consider $\fs_{E_1}(H)$. Since $\fs_{E_1}(H)\subseteq \mathcal{D}$, and as $E_3$ is fully $\mathcal{D}$-normalized and $N_S(E_3)\le E_1$, $E_3$ is also fully $\fs_{E_1}(H)$-normalized. Since $C_{E_1}(E_3)\le E_3$ we see that $E_3$ is also $\fs_{E_1}(H)$-centric. Finally, since $E_3$ is abelian, it is minimal among $S$-centric subgroups with respect to inclusion and has no proper subgroup of $E_3$ is essential in $\fs_{E_1}(H)$. In the statement of \cref{JasonAdd}, letting $\fs_0=\fs_{E_1}(H)$, $V=E_3$ and $\Delta=\Aut_{\mathcal{D}}(E_3)$, we have that $\wt\Delta:=\Aut_{\fs_{E_1}(H)}(E_3)=N_{\Aut_{\mathcal{D}}(E_3)}(\Aut_S(E_3))$ is strongly $3$-embedded in $\Delta$. By that result, $\mathcal{D^*}=\langle \Aut_{\mathcal{D}}(E_3), \fs_{E_1}(H)\rangle_{E_1}$ is a saturated fusion system.

In the construction of $\mathcal{D}^*$, we may have taken in place of $K$ the group obtained by lifting the morphisms in $N_{\Aut_{O^{3'}(\mathcal{D}}(E_3)}(\Aut_S(E_3))$ to $\Aut_{\mathcal{D}}(E_1)$ and forming $\hat{H}$ of index $2$ in $H$ with $\Aut_{\hat{H}}(E_3)=N_{\Aut_{O^{3'}(\mathcal{D})}(E_3)}(\Aut_S(E_3))$. Letting $\fs_0=\fs_{E_1}(\hat{H})$, $\wt\Delta=N_{\Aut_{O^{3'}(\mathcal{D})}(E_3)}(\Aut_S(E_3))$ and $\Delta=\Aut_{O^{3'}(\mathcal{D})}(E_3)$ and applying \cref{JasonAdd}, the fusion system $\hat{\mathcal{D}^*}=\langle O^{3'}(\Aut_{\mathcal{D}}(E_3)), \fs_{E_1}(\hat{H})\rangle_{E_1}$ is a saturated fusion subsystem of $\mathcal{D}$.

By construction, $\mathcal{D}^*=\langle \hat{\mathcal{D}^*}, \Aut_{\mathcal{D}^*}(E_1)\rangle_{E_1}$ and it is clear that for all $\alpha\in \Aut_{\mathcal{D}^*}(E_1)$, $\hat{\mathcal{D}^*}^\alpha=\hat{\mathcal{D}^*}$. Hence, applying \cite[Proposition I.6.4]{ako}, we have that $\hat{\mathcal{D}^*}$ is weakly normal in $\mathcal{D}^*$ in the sense of \cite[Definition I.6.1]{ako} and \cite[Theorem A]{CravenSub} yields that $O^{3'}(\hat{\mathcal{D}^*})\normaleq \mathcal{D}^*$. Then $O^{3'}(\Aut_{\mathcal{D}^*}(T))\le \Aut_{O^{3'}(\hat{\mathcal{D}^*})}(T)\normaleq \Aut_{\mathcal{D}^*}(T)$ by \cite[Proposition I.6.4]{ako} for all $T\le E_1$, and we deduce that $O^{3'}(\hat{\mathcal{D}^*})$ has index prime to $3$ in $\mathcal{D}^*$. It quickly follows that $\hat{\mathcal{D}^*}$ has index prime to $3$ in $\mathcal{D}^*$ and as $\Aut_{\hat{\mathcal{D}^*}}(E_1)\le \Aut_{\mathcal{D}^*}^0(E_1)$, we see that $\Aut_{\hat{\mathcal{D}^*}}(E_1)=\Aut_{\mathcal{D}^*}^0(E_1)$ has index $2$ in $\Aut_{\mathcal{D}^*}(E_1)$. A final application of \cref{p'lemma} gives that $O^{3'}(\mathcal{D}^*)=\hat{\mathcal{D}^*}$ has index $2$ in $\mathcal{D}^*$, as desired.
\end{proof}

We provide some more generic results regarding all possible saturated fusion systems supported on $E_1$. Although we do not formally prove the following proposition, its conclusion merits some explanation. Let $\fs$ be a saturated fusion system on $E_1$. It is fairly easy to show that $\mathcal{E}(\fs)\subseteq \{E_3^S\}$ so we take this as a starting point.

For $E_3^s$ some $S$-conjugate of $E_3$ with $s\not\in E_1$, if $E_3, E_3^s\in\mathcal{E}(\fs)$ then it quickly follows that $O^{3'}(\Aut_{\fs}(E_3))\cong O^{3'}(\Aut_{\fs}(E_3^s))\cong \SL_2(9)$ (as witnessed in \cref{TrivCore}). Applying the extension axiom to $E_3$ and $E_3^s$, we have that for $T$ a cyclic subgroup of $O^{3'}(\Aut_{\fs}(E_3))$ of order $8$ which normalizes $N_{E_1}(E_3)$ the morphisms in $T$ lift to morphisms in $\Aut_{\fs}(E_1)$. Similarly, $T^s$ also lifts. Note that no element of $T$ centralizes $Z_2(S)=Z(E_1)$ and so both $T$ and $T^s$ project to cyclic subgroups of order $8$ in $\Aut(E_1)/C_{\Aut(E_1)}(Z(E_1))\cong \GL_2(3)$. But then the projection of $\langle T, T^s\rangle$ is divisible by $3$, a contradiction since $\Inn(E_1)\in\syl_3(\Aut_{\fs}(E_1))$ and $\Inn(E_1)\le C_{\Aut(E_1)}(Z(E_1))$.

We conclude that if $E_3\in\mathcal{E}(\fs)$ then the only $S$-conjugates of $E_3$ in $\mathcal{E}(\fs)$ are the $E_1$ conjugates of $E_3$. Moreover, for $s\in S\setminus E_1$ and $\alpha_s$ the automorphism of $E_1$ induced by conjugation by $s$, $\fs^{\alpha_s}\cong \fs$ and if $\{E_3^{E_1}\}\subseteq \mathcal{E}(\fs)$ then $\{(E_3^s)^{E_1}\}\in\mathcal{E}(\fs^{\alpha_s})$. Since we only care about classifying fusion systems up to isomorphism, we may as well assume that $E_3\in\mathcal{E}(\fs)$, leading to the following result (which is verified computationally \cref{code}).

\begin{proposition}\label[proposition]{EssenDeterD3}
Let $\fs$ be a saturated fusion system supported on $E_1$. Then $\mathcal{E}(\fs)\subseteq \{E_3^{E_1}\}$.
\end{proposition}

We return to some properties of the systems $\mathcal{D}$ and $O^{3'}(\mathcal{D})$.

\begin{proposition}\label[proposition]{SimpleSub3}
$N_{E_1}(E_3)$ is the unique, proper, non-trivial, strongly closed subgroup in both $O^{3'}(\mathcal{D}^*)$ and $\mathcal{D^*}$,  and ${\mathcal{D}^*}^{frc}=O^{3'}(\mathcal{D}^*)^{frc}=\{E_3^{\mathcal{D}^*}, E_1\}$.
\end{proposition}
\begin{proof}
Since $N_{E_1}(E_3)$ is normalized by $\Aut_{\mathcal{D}^*}(E_1)$ and contains all $\mathcal{D}^*$-conjugates of $E_3$, we have that $N_{E_1}(E_3)$ is strongly closed in $\mathcal{D}^*$ and $O^{3'}(\mathcal{D}^*)$. Assume that $T$ is a proper non-trivial strongly closed subgroup of $\mathcal{D}^*$. Then $\langle (T\cap E_3)^{\Aut_{\mathcal{D}}(E_3)}\rangle\le T$ and since $T\normaleq E_1$, we must have that $[E_3, \Aut_{\mathcal{D}^*}(E_3)]\le T$. But then $\langle [E_3, \Aut_{\mathcal{D}^*}(E_3)]^{E_1}\rangle\le T$ and one can calculate that this implies that $N_{E_1}(E_3)\le T$. 

Let $\tau$ be a non-trivial involution in $Z(O^{3'}(\Aut_{\mathcal{D}^*}(E_3)))$. By the extension axiom, $\tau$ lifts to $\wt \tau\in\Aut_{O^{3'}(\mathcal{D})}(E_1)$. Suppose that $[\wt \tau, E_1]\le N_{E_1}(E_3)$. Since $\wt \tau$ is the extension of $\tau$ to $\Aut_{\mathcal{D}^*}(N_{E_1}(E_3))$ and $\tau$ centralizes $\Aut_{E_1}(E_3)$, we conclude that $[\wt \tau, N_{E_1}(E_3)]\le E_3$ and so by coprime action, we have that $[\wt \tau, E_1]\le E_3$. Since $E_3$ is abelian and $[E_1, E_3, \wt \tau]\le [\Phi(E_1), \wt \tau]=Z(E_1)$, the three subgroups lemma implies that $[E_3, \wt \tau, E_1]\le Z(E_1)$. But then, as $E_3=[E_3, \tau]Z(N_{E_1}(E_3))$ and $[E_1, Z(N_{E_1}(E_3))]\le [E_1, \Phi(E_1)]=Z(E_1)$, we have that $E_3\normaleq E_1$, a contradiction. Hence, $\wt \tau$ acts non-trivially on $E_1/N_{E_1}(E_3)$ and since the Sylow $2$-subgroups of $\Aut_{O^{5'}(\mathcal{D}^*)}(E_1)$ are cyclic of order $8$, we deduce that a Sylow $2$-subgroup of $\Aut_{O^{5'}(\mathcal{D}^*)}(E_1)$ acts faithfully and irreducibly on $E_1/N_{E_1}(E_3)$. We conclude that $T=N_{E_1}(E_3)$ is the unique proper non-trivial strongly closed subgroup of both $O^{3'}(\mathcal{D^*})$ and $\mathcal{D^*}$.

Let $\fs\in\{\mathcal{D}, O^{3'}(\mathcal{D})\}$ and assume that $R$ in $\fs^{frc}$ but not equal to $E_1$. Applying the extension axiom and the Alperin--Goldschmidt theorem, since $R$ is $\fs$-radical, $R$ is contained in at least one $\fs$-essential subgroup. But \cref{EssenDeterD3} then implies that $R$ is contained in an $E_1$-conjugate of $E_3$. Since $E_3$ is elementary abelian and $R$ is $\fs$-centric, we must have that $R$ is $E_1$-conjugate to $E_3$, as required.
\end{proof}

\begin{proposition}\label[proposition]{SimpleSub3b}
$O^{3'}(\mathcal{D}^*)$ is simple and both $O^{3'}(\mathcal{D}^*)$ and $\mathcal{D^*}$ are exotic fusion systems.
\end{proposition}
\begin{proof}
Let $\mathcal{N}\normaleq O^{3'}(\mathcal{D^*})$ supported on $\{1\}\le P\le E_1$. By \cite[Theorem II.9.8(d)]{ako} we may assume that $P<E_1$, and $P$ is strongly closed in $O^{3'}(\mathcal{D^*})$. Hence, $\mathcal{N}$ is supported on $N_{E_1}(E_3)$ and we have that $\Aut_{\mathcal{N}}(E_3)\normaleq \Aut_{O^{3'}(\mathcal{D})}(E_3)$ by \cite[Proposition I.6.4(c)]{ako} so that $\Aut_{\mathcal{N}}(E_3)=\Aut_{O^{3'}(\mathcal{D})}(E_3)\cong \SL_2(9)$.

Let $\tau$ be a non-trivial involution in $Z(\Aut_{\mathcal{N}}(E_3))$. By the extension axiom, $\tau$ lifts to $\wt \tau\in\Aut_{O^{3'}(\mathcal{D})}(E_1)$ and restricts to $\hat{\tau}\in \Aut_{O^{3'}(\mathcal{D})}(N_{E_1}(E_3))$. Indeed, $\hat{\tau}\in \Aut_{\mathcal{N}}(N_{E_1}(E_3))\normaleq \Aut_{O^{3'}(\mathcal{D})}(N_{E_1}(E_3))$ and we ascertain that $[\hat{\tau}, \Aut_{E_1}(N_{E_1}(E_3))]\le \Inn(N_{E_1}(E_3))$. By the extension axiom, we infer that $[\wt \tau, E_1]\le N_{E_1}(E_3)$. But by \cref{SimpleSub3}, the Sylow $2$-subgroups of $\Aut_{O^{3'}(\mathcal{D})}(E_1)$ act faithfully on $E_1/N_{E_1}(E_3)$, a contradiction. Hence, $O^{3'}(\mathcal{D}^*)$ is simple.

Assume that there is $\mathcal{N}$ is a proper non-trivial normal subsystem of $\mathcal{D}^*$. Applying \cite[Theorem II.9.1]{ako} and using that $O^{3'}(\mathcal{D}^*)$ is simple, we deduce that $O^{3'}(\mathcal{D}^*)\le \mathcal{N}$ and it quickly follows that $O^{3'}(\mathcal{D}^*)=\mathcal{N}$. Since $N_{E_1}(E_3)$ is a strongly closed subgroup of both $\mathcal{D}$ and $O^{3'}(\mathcal{D^*})$, by \cref{SCExotic}, we conclude that both $\mathcal{D}$ and $O^{3'}(\mathcal{D^*})$ are exotic. 
\end{proof}

We now classify all saturated fusion systems supported on $E_1$. We preface this classification with the following lemma.

\begin{lemma}\label[lemma]{TrivCore}
Suppose that $\fs$ is saturated fusion system on $E_1$ with $E_3\in\mathcal{E}(\fs)$. Then $O_3(\fs)=\{1\}$ and $O^{3'}(\Aut_{\fs}(E_3))\cong\SL_2(9)$.
\end{lemma}
\begin{proof}
As in \cref{E3inE1}, since $\Phi(E_1)$ induces an FF-action on $E_3$, an application of \cref{SEFF} implies that $O^{3'}(\Aut_{\fs}(E_3))\cong\SL_2(9)$ and $E_3=[E_3, O^{3'}(\Aut_{\fs}(E_3))]\times C_{E_3}(O^{3'}(\Aut_{\fs}(E_3)))$. Moreover, $Z(E_1)\le [E_3, O^{3'}(\Aut_{\fs}(E_3))]\not\normaleq E_1$ and $|C_{E_3}(O^{3'}(\Aut_{\fs}(E_3)))|=3$. Since $E_3\in\mathcal{E}(\fs)$, by \cref{normalinF}, $O_3(\fs)$ is an $\Aut_{\fs}(E_3)$-invariant subgroup of $E_3$ which is also normal in $E_1$, so that $O_3(\fs)=\{1\}$. 
\end{proof}

\begin{ThmD}\hypertarget{ThmD2}{}
Suppose that $\fs$ is saturated fusion system on $E_1$ such that $E_1\not\normaleq \fs$. Then $\fs\cong O^{3'}(\mathcal{D}^*)$ or $\mathcal{D^*}$.
\end{ThmD}
\begin{proof}
Since $E_1\not\normaleq \fs$, we must have that $E_3$ is essential in $\fs$ by \cref{EssenDeterD3}. By \cref{TrivCore}, we have that $O_3(\fs)=\{1\}$ and $O^{3'}(\Aut_{\fs}(E_3))\cong\SL_2(9)$. Let $K$ be a Hall $3'$-subgroup of $N_{\Aut_{\fs}(E_3)}(\Aut_S(E_3))$ so that by the extension axiom, $K$ lifts to a group of automorphisms of $E_1$, which we denote by $\hat{K}$. As in \cref{3Conjugate}, we calculate that $|\Aut(E_1)|_{3'}=16$ and so $\Out_{\mathcal{D}^*}(E_1)$ is a Sylow $2$-subgroup of $\Out(E_1)$. Set $L:=K\cap O^{3'}(\Aut_{\fs}(E_3))$ and $\hat{L}$ the lift of $L$ to $\Aut_{\fs}(E_1)$. Then $\hat{L}$ is the unique cyclic subgroup of $\hat{K}$ of order $8$ and has index at most $2$ in $\hat{K}$. We may choose $\alpha\in\Aut(E_1)$ so that $\hat{K}^\alpha\Inn(E_1)\le \Aut_{\fs^\alpha}(E_1)\le \Aut_{\mathcal{D}^*}(E_1)$. Indeed, $\hat{L}^\alpha\Inn(E_1)=\Aut_{O^{3'}(\mathcal{D}^*)}(E_1)$. Applying \cref{model}, we deduce that there is $\beta\in\Aut(E_1)$ with \[N_{O^{3'}(\mathcal{D}^*)}(E_1)\le N_{\fs^{\alpha\beta}}(E_1)\le N_{\mathcal{D}^*}(E_1).\]

Either way, we invoke \cref{E3Unique} so that $\Aut_{\fs^{\alpha\beta}}(E_3)=\Aut_{O^{3'}(\mathcal{D}^*)}(E_3)$ if $N_{\fs^{\alpha\beta}}(E_1)=N_{O^{3'}(\mathcal{D}^*)}(E_1)$, while $\Aut_{\fs^{\alpha\beta}}(E_3)=\Aut_{\mathcal{D}^*}(E_3)$ if $N_{\fs^{\alpha\beta}}(E_1)=N_{\mathcal{D}^*}(E_1)$. Then the Alperin--Goldschmidt theorem implies that $\fs^{\alpha\beta}=O^{3'}(\mathcal{D}^*)$ or $\mathcal{D}^*$ and the theorem holds.
\end{proof}

The following table summarizes the actions induced by the fusion systems described in \hyperlink{ThmD2}{Theorem D} on their centric-radical subgroups.

\begin{table}[H]
    \label{SubF3Table}
        \caption{$\mathcal{D}$-conjugacy classes of radical-centric subgroups of $E_1$}
        \begin{tabular}{|c|c|c|c|}\hline
            $P$ & $|P|$ & $\Out_{\mathcal{D}^*}(P)$ & $\Out_{O^{3'}(\mathcal{D}^*)}(P)$ \\\hline
            $E_1$ & $3^9$ & $\mathrm{SD}_{16}$ & $\mathrm{C}_8$ \\\hline          
            $E_3$ & $3^5$ & $\SL_2(9).2$ & $\SL_2(9)$\\ \hline
        \end{tabular}
\end{table}

\section{Fusion Systems related to a Sylow $5$-subgroup of $\mathrm{M}$}\label{MonSec}

In this final section, we investigate saturated fusion systems on a $5$-group $S$ which is isomorphic to a Sylow $5$-subgroup of the Monster sporadic simple group $\mathrm{M}$. As in the previous section, we document some exotic fusion systems supported on $S$ and some exotic fusion systems supported on a particular index $5$ subgroup of $S$. Once again, the Atlas \cite{atlas} is an invaluable tool in illustrating the structure of $\mathrm{M}$ and its actions. As a starting point, we consider the following maximal $5$-local subgroups of $\mathrm{M}$:

\begin{center}
$M_1\cong 5^2.5^2.5^4:(\Sym(3)\times \GL_2(5))$\\\vspace{0.5em}
$M_2\cong 5^{1+6}_+:4.\mathrm{J}_2.2$\\\vspace{0.5em}
$M_3\cong 5^4:(3\times \SL_2(25)).2$\\\vspace{0.5em}
$M_4\cong 5^{3+3}.(2\times \PSL_3(5))$
\end{center}

remarking that $|S|=5^9$, and for a given $S\in\syl_5(\mathrm{M})$ each $M_i$ be chosen such that $S\cap M_i\in\syl_5(M_i)$. Choose $M_i$ such that this holds. 

Let $E_1:=O_5(M_1)=C_S(Z_2(S))$ of order $5^8$, and $E_3:=O_5(M_3)$ elementary abelian of order $5^4$. Furthermore, note that $\mathbf{Q}:=O_5(M_2)$ is the unique extraspecial subgroup of $S$ of order $5^7$ and so is characteristic in $S$. 

We appeal to the online version of the Atlas of Finite Group Representations \cite{OnlineAtlas} for representations of $M_i$ for $i\in\{1,2,3,4\}$. These are accessible without the need to construct the Monster computationally. We consider $M_1$ as a permutation group on $750$ points, $M_2$ in its $8$-dimensional matrix representation over $\GF(5)$, and $M_4$ as a permutation group on $7750$ points. Naturally, we access $S$ and $E_1$ computationally via $M_1$.

We note some important structural properties of $M_1$ which will be used later. Namely, we have that $\Phi(E_1)$ is of order $5^4$ and $Z(E_1)=Z_2(S)$ is of order $5^2$. Moreover, we can choose a subgroup isomorphic to $\Sym(3)$ in $M_1/E_1$ which acts trivially on $Z(E_1)$. We shall denote this subgroup $A_1$ and refer to as the ``pure" $\Sym(3)$ in $M_1/E_1$. We record that the unique normal subgroup of $M_1/E_1$ isomorphic to $\GL_2(5)$ acts faithfully on $Z(E_1)$ and centralizes $A_1$. In this way, we have that $M_1/E_1=A_1\times B_1\cong \Sym(3)\times \GL_2(5)$. Moreover, $O^{5'}(M_1)=C_{M_1}(\Phi(E_1)/\Inn(E_1))$, $O^{5'}(M_1/E_1)=O^{5'}(M_1)/E_1\cong \SL_2(5)$ and $O^{5'}(M_1/E_1)\le B_1$.

We desire more candidates for essentials subgroups of the $5$-fusion category of $\mathrm{M}$ and for this we examine the structure of $M_2$. Let $X\normaleq M_2$ with $M_2/X\cong \mathrm{J}_{2}.2$ and consider the maximal subgroup $H\cong (\Alt(5)\times \Dih(10)).2$ of $M_2/X$. Define $E_2$ to be the largest normal $5$-subgroup of the preimage of $H$ in $M_2$ so that 
\[N_{\mathrm{M}}(E_2)=N_{M_2}(E_2)\cong 5^{1+6}_+.5:(2\times\GL_2(5)).\]
Then $E_2$ is an essential subgroup of $\fs_S(\mathrm{M})$ of order $5^8$, $\mathbf{Q}$ is characteristic in $E_2$ and $[N_{\mathrm{M}}(S): N_{N_{\mathrm{M}}(S)}(E_2)]=3$.

We remark that $M_2=\langle N_{\mathrm{M}}(S), N_{\mathrm{M}}(E_2)\rangle$ and we can arrange, up to conjugacy, that $M_4=\langle O^{5'}(N_{\mathrm{M}}(E_1)), N_{\mathrm{M}}(E_2)\rangle$. In particular, setting $\mathbf{R}:=O_5(M_4)$, we have that $[N_{\mathrm{M}}(S): N_{N_{\mathrm{M}}(S)}(\mathbf{R})]=3$. For ease of notation, we fix $\mathcal{G}:=\fs_S(\mathrm{M})$ for the remainder of this section.

\begin{proposition}\label[proposition]{MonsterRad}
$\mathcal{E}(\mathcal{G})=\{E_1, E_2^\mathcal{G}, E_3^\mathcal{G}\}$ and $\mathcal{G}^{frc}=\{E_1, E_2^\mathcal{G}, E_3^\mathcal{G}, \mathbf{Q}, \mathbf{R}^\mathcal{G}, S\}$.
\end{proposition}
\begin{proof}
See \cite[Theorem 5]{MSub}.
\end{proof}

As in \cref{F3Sec}, before describing any exotic subsystems of $\mathcal{G}$, we require an observation regarding the containment of some essentials in others and lean on MAGMA for the determination of all possible essential subgroups of a saturated fusion system $\fs$ supported on $S$. The following proposition is verified computationally (see \cref{code}).

\begin{proposition}\label[proposition]{MonsterEssen}
Suppose that $\fs$ is saturated fusion system on $S$. Then $\mathcal{E}(\fs)\subseteq \{E_1, E_2^\mathcal{G}, E_3^\mathcal{G}\}$.
\end{proposition}

We remark that each of the three $\mathcal{G}$-conjugates of $E_2$ is normal in $S$. We record that upon restricting to $S$, the $\mathcal{G}$-conjugates of $E_3$ split into four distinct classes, fused by elements of $N_{\Aut_{\mathcal{G}}(S)}(E_2)$. We provide some generic results regarding all saturated fusion systems on $S$ which also elucidate some of the structure of $\mathcal{G}$.

\begin{lemma}\label[lemma]{E3inE151}
Every $\mathcal{G}$-conjugate of $E_3$ is contained in $E_1$ and not contained in any $\mathcal{G}$-conjugate of $E_2$.
\end{lemma}
\begin{proof}
Since $\{E_3^\mathcal{G}\}=\{E_3^{N_{\Aut_{\mathcal{G}}(S)}(E_2)}\}$ (see \cref{code}) and $E_1$ and $E_2\alpha$ are normalized by $N_{\Aut_{\mathcal{G}}(S)}(E_2)$ for all $\alpha\in \Aut_{\mathcal{G}}(S)$, for the first statement of the lemma it suffices to prove that $E_3\le E_1$ and $E_3\not\le E_2\alpha$ for all $\alpha\in \Aut_{\mathcal{G}}(S)$. To this end, we note that $[Z_2(S), E_3]=\{1\}$ so that $E_3\le E_1$. One can see this in $\mathcal{G}$, for otherwise since $E_3$ is elementary abelian we would have that $Z_2(S)\not\le E_3$ and $[Z_2(S), E_3]\le Z(S)$, a contradiction since $O^{5'}(\Out_{\mathcal{G}}(E_3))\cong\SL_2(25)$ has no non-trivial modules exhibiting this behaviour. If $E_3\le E_2\alpha$ for some $\alpha\in \Aut_{\mathcal{G}}(S)$, then as $E_2\alpha\normaleq S$, we have that $E_1=\langle E_3^S\rangle\le E_2\alpha$, an obvious contradiction.
\end{proof}

\begin{lemma}\label[lemma]{MSL21}
Suppose that $\fs$ is a saturated fusion system on $S$ with $E_2\in\mathcal{E}(\fs)$. Then $|\Phi(E_2)|=5^5$, $O^{5'}(\Aut_{\fs}(E_2))$ acts trivially on $E_2/\mathbf{Q_2}$, and both $\mathbf{Q}/\Phi(E_2)$ and $\Phi(E_2)/\Phi(\mathbf{R})$ are natural modules for $O^{5'}(\Out_{\fs}(E_2))\cong \SL_2(5)$.
\end{lemma}
\begin{proof}
We compute (see \cref{code}) that $\Phi(E_2)$ is of order $5^5$ and so has index $5^3$ in $E_2$. Then $\mathbf{Q}$ has index $5$ in $E_2$ and $\Phi(E_2)$ has index $5^2$ in $\mathbf{Q}$. Thus, $O^{5'}(\Aut_{\fs}(E_2))$ acts trivially on $E_2/\mathbf{Q}$ and since $O^5(O^{5'}(\Aut_{\fs}(E_2))$ acts faithfully on $E_2/\Phi(E_2)$, we deduce that $O^{5'}(\Out_{\fs}(E_2))\cong \SL_2(5)$ and $\mathbf{Q}/\Phi(E_2)$ is a natural module. 

We have that $\Phi(E_2)<\mathbf{R}\le E_2$ so that $\Phi(\mathbf{R})< \Phi(E_2)$. We calculate (see \cref{code}) that $\Phi(\mathbf{R})=[E_2, \Phi(E_2)]$ is characteristic in $E_2$ and so $\mathbf{R}=C_{E_2}(\Phi(\mathbf{R}))$ is also characteristic in $E_2$. Since $S$ centralizes $\mathbf{R}/\Phi(E_2)$, we either have that $O^{5'}(\Out_{\fs}(E_2))$ acts faithfully on $\Phi(E_2)/\Phi(\mathbf{R})$ of order $5^2$; or $O^{5'}(\Out_{\fs}(E_2))$ acts trivially on $\mathbf{R}$. Since $R$ is self-centralizing in $S$ the latter case clearly gives a contradiction. Hence, the former case holds and $\Phi(E_2)/\Phi(\mathbf{R})$ is a natural module for $O^{5'}(\Out_{\fs}(E_2))\cong \SL_2(5)$.
\end{proof}

The above lemma also holds for any $\mathcal{G}$-conjugate of $E_2\alpha$ which is essential in $\fs$, with $\mathbf{R}$ replaced by $\mathbf{R}\alpha$.

\begin{lemma}\label[lemma]{E3inE152}
Let $\fs$ be a saturated fusion system on $S$. Let $P$ be some $\mathcal{G}$-conjugate of $E_3$. If $P\in\mathcal{E}(\fs)$, then $O^{5'}(\Aut_{\fs}(P))\cong\SL_2(25)$, $E_1\in\mathcal{E}(\fs)$ and $O_5(\fs)=\{1\}$.
\end{lemma}
\begin{proof}
Let $P$ be some $\mathcal{G}$-conjugate of $E_3$ and suppose that $P\in\mathcal{E}(\fs)$. Then $\Phi(E_1)$ is elementary abelian of order $5^4$ and is not contained in $P$. Furthermore, by \cref{E3inE151} $[P,\Phi(E_1)]\le [E_1, \Phi(E_1)]=Z_2(S)\le P$ so that $\Phi(E_1)\le N_S(P)$. Since $P$ is $\mathcal{G}$-essential, and $|N_S(P)/P|=5^2$, applying \cref{SEFF} we see that $N_S(P)=P\Phi(E_1)$, $P\cap \Phi(E_1)=Z_2(S)$ and $\Phi(E_1)$ induces an FF-action on $P$. Then for $L:=O^{5'}(\Aut_{\fs}(P))$, \cref{SEFF} implies that $L\cong \SL_2(25)$ and $P=[P, L]$ is a natural module. 

Let $K$ be a Hall $5'$-subgroup of $N_L(\Aut_S(P))$ so that $K$ is cyclic of order $24$ and acts irreducibly on $Z_2(S)$. Then the morphisms in $K$ lift to a larger subgroup of $S$ by the extension axiom and if $E_1$ is not essential then, using the Alperin--Goldschmidt theorem, \cref{MonsterEssen} and \cref{E3inE151}, the morphisms in $K$ must lift to automorphisms of $S$. But then, upon restriction, the morphisms in $K$ would normalize $Z(S)$, contradicting the irreducibility of $Z_2(S)$ under the action of $K$. Hence, $E_1\in\mathcal{E}(\fs)$. Since $O_5(\fs)\normaleq S$, $P$ is irreducible under the action of $\Aut_{\fs}(P)$ and, by \cref{normalinF}, $O_5(\fs)\le P$, we conclude that $O_5(\fs)=\{1\}$.
\end{proof}

\begin{lemma}\label[lemma]{MSL2a}
Suppose that $\fs$ is a saturated fusion system on $S$ with $E_1\in\mathcal{E}(\fs)$. Then $|\Phi(E_1)|=5^4$, $O^{5'}(\Aut_{\fs}(E_1))$ acts trivially on $\Phi(E_1)/Z(E_1)$, $Z(E_1)$ is a natural module for $O^{5'}(\Out_{\fs}(E_1))\cong \SL_2(5)$, and $O^{5'}(\Aut_{\fs}(E_1))$ normalizes every $\Aut_{\mathcal{G}}(S)$-conjugate of $\mathbf{R}$.
\end{lemma}
\begin{proof}
We compute that $\Phi(E_1)$ is elementary abelian of order $5^4$ and that $S$ centralizes $\Phi(E_1)/Z(E_1)$. In particular, $O^{5'}(\Aut_{\fs}(E_1))$ acts trivially on $\Phi(E_1)/Z(E_1)$. Set $L:=O^{5'}(\Out_{\fs}(E_1))$ and notice that for $r\in L$ of $5'$-order, if $r$ acts trivially on $\Phi(E_1)$ then, by the three subgroups lemma, $r$ centralizes $E_1/C_{E_1}(\Phi(E_1))$. Since $\Phi(E_1)$ is self-centralizing in $E_1$, we deduce that $L$ acts faithfully on $\Phi(E_1)$. In particular, $C_L(Z(E_1))=\{1\}$. Since $Z(E_1)$ has order $5^2$, we conclude that $Z(E_1)$ is natural module for $L\cong \SL_2(5)$.

We note that $Z(E_1)$, $\Phi(E_1)$ and $E_1$ are all invariant under $\Aut_{\mathcal{G}}(S)$. Hence, for $\alpha\in \Aut_{\mathcal{G}}(S)$, $\mathbf{R}\alpha\le E_1$ so that $\Phi(\mathbf{R}\alpha)\le \Phi(E_1)$. Since $Z(E_1)$ centralizes $\mathbf{R}\alpha$, we deduce that $Z(E_1)\le Z(\mathbf{R}\alpha)=\Phi(\mathbf{R}\alpha)$ and as $O^{5'}(\Aut_{\fs}(E_1))$ centralizes $\Phi(E_1)/Z(E_1)$, $O^{5'}(\Aut_{\fs}(E_1))$ normalizes $\Phi(\mathbf{R}\alpha)$ and so normalizes $C_{E_1}(\Phi(\mathbf{R}\alpha))=\mathbf{R}\alpha$. 
\end{proof}

\begin{lemma}\label[lemma]{MSL2}
Suppose that $\fs$ is a saturated fusion system on $S$ with $E_1\in\mathcal{E}(\fs)$. Then there is $\gamma\in\Aut(E_1)$ with $\Aut_{\fs}(E_1)^\gamma\le \Aut_{\mathcal{G}}(E_1)$ and we may choose $A,B\le \Out(E_1)$ such that $A=A_1^\gamma\cong \Sym(3)$ with $[A, Z(E_1)]=\{1\}$, $B=B_1^\gamma\cong\GL_2(5)$ with $[B, A]=\{1\})$, and $\Out_{\fs}(E_1)\le A\times B$ with $O^{5'}(\Out_{\fs}(E_1))\le B$. 
\end{lemma}
\begin{proof}
Let $T$ be a Sylow $2$-subgroup of $O^{5'}(\Aut_{\fs}(E_1))$ so that $|T|=2^3$. Then $N_T(\Aut_S(E_1))$ is cyclic of order $4$ and $T$ centralizes $\Phi(E_1)/Z(E_1)$ and so centralizes $\Aut_{\Phi(E_1)}(E_1)$. We calculate (\cref{code}) that
$N_T(\Aut_S(E_1))$ is a Sylow $2$-subgroup of $C_{N_{\Aut(E_1)}(\Aut_S(E_1))}(\Aut_{\Phi(E_1)}(E_1))$ and $N_T(\Aut_S(E_1))$ is conjugate by an element of $N_{\Aut(E_1)}(\Aut_S(E_1))$ to a Sylow $2$-subgroup $N_{O^{5'}(\Aut_{\mathcal{G}}(E_1))}(\Aut_S(E_1))$.

We have that there is a unique $\Aut(E_1)$-conjugacy class of subgroups which contain $N_{O^{5'}(\Aut_{\mathcal{G}}(E_1))}(\Aut_S(E_1))$ with quotient by $\Inn(E_1)$ isomorphic to $\SL_2(5)$. Indeed, $O^{5'}(\Aut_{\mathcal{G}}(E_1))$ satisfies these conditions and tracing backwards, we ascertain that $O^{5'}(\Aut_{\fs}(E_1))$ is $\Aut(E_1)$-conjugate to $O^{5'}(\Aut_{\mathcal{G}}(E_1))$. Finally, the normalizer in $\Out(E_1)$ of $O^{5'}(\Out_{\mathcal{G}}(E_1))$ is $\Out_{\mathcal{G}}(E_1)$ so that $\Out_{\fs}(E_1)\le N_{\Out(E_1)}(O^{5'}(\Out_{\fs}(E_1)))$ and $N_{\Out(E_1)}(O^{5'}(\Out_{\fs}(E_1)))$ is $\Out(E_1)$ conjugate to $\Out_{\mathcal{G}}(E_1)$. Hence, $\Aut_{\fs}(E_1)$ is $\Aut(E_1)$-conjugate to a subgroup of $\Aut_{\mathcal{G}}(E_1)$ and the rest of the result follows from the description of $\Out_{\mathcal{G}}(E_1)\cong M_1/E_1$.
\end{proof}

\begin{lemma}\label[lemma]{E3inE15a}
Let $\fs$ be a saturated fusion system on $S$. Let $P$ be some $\mathcal{G}$-conjugate of $E_3$. If $P\in\mathcal{E}(\fs)$, then $\Out_{\fs}(E_1)=C_{\Out_{\fs}(E_1)}(Z(E_1))\times B$, where $B\cong \GL_2(5)$, $[B, C_{\Out_{\fs}(E_1)}(Z(E_1))]=\{1\}$ and $|C_{\Out_{\fs}(E_1)}(Z(E_1))|\in\{3,6\}$.
\end{lemma}
\begin{proof}
Let $P\in\mathcal{E}(\fs)$ where $P$ is a $\mathcal{G}$-conjugate of $E_3$. Applying \cref{E3inE152}, $E_1\in\mathcal{E}(\fs)$ and $O^{5'}(\Aut_{\fs}(P))\cong\SL_2(25)$, and following the notation from the proof of that result, we set $K$ to be a Hall $5'$-subgroup of $N_{O^{5'}(\Aut_{\fs}(P))}(\Aut_S(P))$. Then $K$ is cyclic of order $24$ and using that $E_1$ is $\Aut_{\fs}(S)$-invariant and applying \cref{MonsterEssen} and the Alperin--Goldschmidt theorem, by the extension axiom we see that $K$ lifts to a group of automorphisms of $E_1$ which we denote $\hat{K}$. Then $\hat{K}$ acts on $Z(E_1)$ as $K$ does. In particular, $\hat{K}$ acts faithfully on $Z(E_1)$. By \cref{MSL2}, $\Out_{\fs}(E_1)$ is $\Out(E_1)$-conjugate to a subgroup of $\Out_{\mathcal{G}}(E_1)$ so that $\Out_{\fs}(E_1)/C_{\Out_{\fs}(E_1)}(Z(E_1))$ is isomorphic to a subgroup of $\GL_2(5)$. But $O^{5'}(\Out_{\fs}(E_1))\cap C_{\Out_{\fs}(E_1)}(Z(E_1))=\hat{K}\Inn(E_1)/\Inn(E_1)\cap C_{\Out_{\fs}(E_1)}(Z(E_1))=\{1\}$ and we deduce that $\Out_{\fs}(E_1)/C_{\Out_{\fs}(E_1)}(Z(E_1))\cong \GL_2(5)$.

Furthermore, again using that $\Out_{\fs}(E_1)$ is $\Out(E_1)$-conjugate to a subgroup of $\Out_{\mathcal{G}}(E_1)$, we deduce that $|C_{\Out_{\fs}(E_1)}(Z(E_1))|\leq 6$. Now, a Sylow $3$-subgroup of $O^{5'}(\Out_{\fs}(E_1))$ centralizes $\Phi(E_1)/Z(E_1)$. Since a Sylow $3$-subgroup of $\hat{K}$ acts on $\Phi(E_1)/Z(E_1)$ as $K$ acts on $\Aut_S(E_1)\cong \Phi(E_1)P/P\cong \Phi(E_1)/Z(E_1)$, we have that a Sylow $3$-subgroup of $\hat{K}$ acts non-trivially on $\Phi(E_1)/\Inn(E_1)$. Hence, $9\divides |\Out_{\fs}(E_1)|$ and it follows that $|C_{\Out_{\fs}(E_1)}(Z(E_1))|\in\{3,6\}$.

Since $\Out_{\fs}(E_1)$ is $\Out(E_1)$-conjugate to a subgroup of $\Out_{\mathcal{G}}(E_1)$, if $|C_{\Out_{\fs}(E_1)}(Z(E_1))|=6$ then $\Out_{\fs}(E_1)$ is $\Out(E_1)$-conjugate to $\Out_{\mathcal{G}}(E_1)$ and the result is easily seen to hold. Hence, we assume that $|C_{\Out_{\fs}(E_1)}(Z(E_1))|=3$ so that $\Out_{\fs}(E_1)=\langle \hat{K}\Inn(E_1)/\Inn(E_1), O^{5'}(\Out_{\fs}(E_1))\rangle$. We note that $\{E_3^{\mathcal{G}}\}$ is the unique class of elementary abelian subgroups $H$ of $E_1$ of order $5^4$ with $|N_{E_1}(H)|=5^6$ and $[N_{E_1}(H), E_1]=Z(E_1)$ (see \cref{code}). In particular, this class is invariant under $\Aut(E_1)$. Since $\Aut_{\fs}(E_1)$ is $\Out(E_1)$-conjugate to a subgroup of $\Aut_{\mathcal{G}}(E_1)$, and we can choose a Sylow $3$-subgroup of $\Aut_{\mathcal{G}}(E_1)$ to normalize $P$, we can also choose a Sylow $3$-subgroup of $\Aut_{\fs}(E_1)$ to normalize $P$. In particular, there is a $3$-element $t$ of $\Aut_{\fs}(E_1)$ which normalizes $P$ and centralizes $\Phi(E_1)/Z(E_1)\cong \Aut_S(P)$. Since $O^{5'}(\Aut_{\fs}(P))\cong \SL_2(25)$, we must that $t|_P$ centralizes $O^{5'}(\Aut_{\fs}(P))\cong \SL_2(25)$. Hence, $\hat{K}$ centralizes a Sylow $3$-subgroup of $\Aut_{\fs}(E_1)$ and so $\hat{K}\Inn(E_1)/\Inn(E_1)$ centralizes $C_{\Out_{\fs}(E_1)}(Z(E_1))$. But $O^{5'}(\Out_{\fs}(E_1))$ centralizes $C_{\Out_{\fs}(E_1)}(Z(E_1))$ and so we see that $\Out_{\fs}(E_1)$ centralizes $C_{\Out_{\fs}(E_1)}(Z(E_1))$. Finally, since $\Out_{\fs}(E_1)$ is $\Out(E_1)$-conjugate to a subgroup of $\Out_{\mathcal{G}}(E_1)$, the lemma holds.
\end{proof}

\begin{lemma}\label[lemma]{E3inE15b}
Let $\fs$ be a saturated fusion system on $S$. Let $P$ be some $\mathcal{G}$-conjugate of $E_3$ and set $\Aut_{\fs}^*(S)$ the subgroup of $\Aut_{\fs}(S)$ generated by all morphisms in $O^{5'}(\Aut_{\fs}(R))$ which extend to automorphisms of $S$, where $R\in\{E_1, P^{\fs}, S\}$. If $P\in\mathcal{E}(\fs)$, then
\begin{enumerate}
    \item $|\Aut_{\fs}(S)/\Aut_{\fs}^*(S)|=|C_{\Out_{\fs}(E_1)}(Z(E_1))|/3$; 
    \item $\{P^{\fs}\}=\{E_3^{\mathcal{G}}\}$; and
    \item $[\Aut_{\fs}(P): O^{5'}(\Aut_{\fs}(P))]=|C_{\Out_{\fs}(E_1)}(Z(E_1))|$.
\end{enumerate}
Moreover, if $\{E_2^{\mathcal{G}}\}\cap \mathcal{E}(\fs)=\emptyset$ then $\Aut_{\fs}^*=\Aut_{\fs}^0(S)$, $\Out_{O^{5'}(\fs)}(E_1)\cong 3\times \GL_2(5)$ and $\Aut_{O^{5'}(\fs)}(E_3)\cong 3\times \SL_2(25)$.
\end{lemma}
\begin{proof}
Note that any morphism in $O^{5'}(\Aut_{\fs}(R))$ which extends to automorphisms of $S$ restricts faithfully to a morphism in $\Aut_{\fs}(E_1)$. Indeed, it follows that $|\Aut_{\fs}(S)|/|\Aut_{\fs}^*(S)|=|\Aut_{\fs}(E_1)|/|\langle \Aut_{\fs}^*(S)|_{E_1}, O^{5'}(\Aut_{\fs}(E_1))\rangle|$. By the proof of \cref{E3inE15a}, and using that $\Out_{\fs}(E_1)\le A\times B$ in the language of \cref{MSL2}, we have that \[\langle \Aut_{\fs}^*(S)|_{E_1}, O^{5'}(\Aut_{\fs}(E_1))\rangle\Inn(E_1)/\Inn(E_1)=O_3(A)\times B\cong 3\times \GL_2(5)\]. Hence, (i) holds.

We observe that the subgroup of $\Aut_{\mathcal{G}}(E_1)$ with quotient by $\Inn(E_1)$ isomorphic to $3\times \GL_2(5)$ acts transitively on the set $\{E_3^{\mathcal{G}}\}$ (see \cref{code}), and is conjugate by $\Aut(E_1)$ to a subgroup of $\Aut_{\fs}(E_1)$ by \cref{E3inE15a}. Since $\{E_3^{\mathcal{G}}\}$ is preserved by $\Aut(E_1)$ as in \cref{E3inE15a} and $P$ is $\mathcal{G}$-conjugate to $E_3$, we have that $\{P^{\fs}\}=\{E_3^{\mathcal{G}}\}$ and so (ii) holds.

Now, it follows by a Frattini argument that $|N_{\Aut_{\fs}(E_3)}(\Aut_S(E_3))|=[\Aut_{\fs}(P): O^{5'}(\Aut_{\fs}(P))]|N_{O^{5'}(\Aut_{\fs}(E_3))}(\Aut_S(E_3))|$. By the extension axiom, and using that $E_1$ is characteristic in $S$, we see that all morphisms in $N_{\Aut_{\fs}(E_3)}(\Aut_S(E_3))$ lift to morphisms in $\Aut_{\fs}(E_1)$ which normalize $E_3$. But $\Aut_{\fs}(E_1)$ is $\Aut(E_1)$-conjugate to a subgroup of $\Aut_{\mathcal{G}}(E_1)$ and as $\Inn(E_1)$ preserves the class $\{E_3^{\mathcal{G}}\}$, we may calculate $|N_{\Aut_{\fs}(E_3)}(\Aut_S(E_3))|$ from $N_{\Aut_{\mathcal{G}}(E_1)}(E_3)$. Writing $H$ for the subgroup of $\Aut_{\mathcal{G}}(E_1)$ with $H/\Inn(E_1)\cong 3\times \GL_2(5)$, we calculate (see \cref{code}) that $N_H(E_3)$ has index $2$ in $N_{\Aut_{\mathcal{G}}(E_1)}(E_3)$ and so (iii) holds. 

Finally, assume that $\{E_2^{\mathcal{G}}\}\cap \mathcal{E}(\fs)=\emptyset$. We clearly have that $\Aut_{\fs}^*(S)\le \Aut_{\fs}^0(S)\le \Aut_{\fs}(S)$. Aiming for a contradiction, suppose that $\Aut_{\fs}^*(S)<\Aut_{\fs}^0(S)$ so that $\Aut_{\fs}^0(S)=\Aut_{\fs}(S)$. Then we see that $\Out_{\fs}(E_1)=A\times B\cong \Sym(3)\times \GL_2(5)$. By \cref{model}, we let $H$ be a model for $N_{\fs}(E_1)$ and let $H^*\normaleq H$ such that $H^*/E_1\cong 3\times \GL_2(5)$. Indeed, $H^*$ is unique with respect to this property. Form the fusion system $\fs^*:=\langle O^{5'}(\Aut_{\fs}(P)), \fs_S(H^*)\rangle_S$. Applying \cref{JasonAdd} and by the definition of $\Aut_{\fs}^*(S)$, we have that $\fs^*$ is saturated and $\fs=\langle \fs^*, \Aut_{\fs}(S)\rangle_S$. Moreover, since $\fs=\langle \Aut_{\fs}(E_1), \Aut_{\fs}(P), \Aut_{\fs}(S)\rangle_S$ by the Alperin--Goldschmidt theorem and \cref{MonsterEssen}, for all $\alpha\in\fs$ we have that ${\fs^*}^\alpha=\fs^*$. Hence, applying \cite[Proposition I.6.4]{ako}, we have that $\fs^*$ is weakly normal in $\fs$ in the sense of \cite[Definition I.6.1]{ako} and \cite[Theorem A]{CravenSub} yields that $O^{5'}(\fs^*)\normaleq \fs$. Then $O^{5'}(\Aut_{\fs}(T))\le \Aut_{O^{5'}(\fs^*)}(T)\normaleq \Aut_{\fs}(T)$ by \cite[Proposition I.6.4]{ako} for all $T\le S$, and we deduce that $O^{5'}(\fs^*)$ has index prime to $5$ in a $\fs$, a contradiction by \cref{p'lemma} since $\Aut_{\fs}^0(S)=\Aut_{\fs}(S)$. Hence, $\Aut_{\fs}^*=\Aut_{\fs}^0(S)$. 

Then $\Out_{O^{5'}(\fs)}(E_1)=\langle \Aut_{\fs}^0(S)|_{E_1}, O^{5'}(\Aut_{\fs}(E_1))\rangle\Inn(E_1)/\Inn(E_1)=O_3(A)\times B\cong 3\times \GL_2(5)$. As in \cref{E3inE15a}, we see that we may choose $t\in \Aut_{\fs}(E_1)$ to normalizes $E_3$ with $[t, \Phi(E_1)]\le Z(E_1)$ so that $t|_{E_3}$ centralizes $O^{5'}(\Aut_{\fs}(E_3))$. Then part (iii) implies that $\Aut_{O^{5'}(\fs)}(E_3)\cong 3\times \SL_2(25)$.
\end{proof}

As a consequence of \cref{E3inE152} and \cref{E3inE15b}, if any $\mathcal{G}$-conjugate of $E_3$ is essential in $\fs$, then $\{E_1, \{E_3^{\mathcal{G}}\}\}\subseteq \mathcal{E}(\fs)$.

We now construct some exotic fusion subsystems of $\mathcal{G}$ in a similar manner to \cref{F3Sec}, and persist with the same notations. That is, we set \[\mathcal{H}=\langle \Aut_{\mathcal{G}}(E_1), \Aut_{\mathcal{G}}(E_2), \Aut_{\mathcal{G}}(S)\rangle_S\] and \[\mathcal{D}=\langle \Aut_{\mathcal{G}}(E_1), \Aut_{\mathcal{G}}(E_3), \Aut_{\mathcal{G}}(S)\rangle_S.\]

\begin{proposition}\label[proposition]{HExotic51}
$\mathcal{H}$ is a saturated fusion system with with $\mathcal{E}(\mathcal{H})=\{E_1, E_2^\mathcal{H}\}$ and $\mathcal{H}^{frc}=\{E_1, E_2^\mathcal{H}, \mathbf{Q}, \mathbf{R}^\mathcal{H}, S\}$.
\end{proposition}
\begin{proof}
Set $\mathcal{H}^*:=\langle \mathcal{H}, H_{\mathcal{G}}(E_3)\rangle_S$, the saturated fusion system determined by applying \cref{Pruning} to $\mathcal{G}$ with $P=E_3$ and $K=H_\fs(E_3)$. Then $\mathcal{H}^*$ is saturated and $E_3\not\in\mathcal{E}(\mathcal{H}^*)$ by \cref{Pruning}. Hence, by \cref{MonsterEssen}, $\mathcal{E}(\mathcal{H}^*)=\{E_1, E_2^\mathcal{G}\}$ and the Alperin--Goldschmidt theorem implies that $\mathcal{H}^*=\langle \Aut_{\mathcal{G}}(E_1), \Aut_{\mathcal{G}}(E_2), \Aut_{\mathcal{G}}(S)\rangle_S=\mathcal{H}$ is saturated.

Let $R$ be a fully $\mathcal{H}$-normalized, radical, centric subgroup of $S$ not equal to one described in the conclusion of the proposition. Then $R$ must be contained in an $\mathcal{H}$-essential subgroup for otherwise, by the extension axiom and the Alperin--Goldschmidt theorem, we infer that $\Out_S(R)\normaleq \Out_{\mathcal{H}}(R)$ and $R$ is not $\mathcal{H}$-radical. If $R$ is contained in an $\mathcal{G}$-conjugate of $E_3$, $A$ say, then since $R$ is $\mathcal{H}$-centric, $R=A$. Then $\Out_S(R)\le O^{5'}(\Out_{\mathcal{H}}(R))\le O^{5'}(\Out_{\mathcal{G}}(R))\cong \SL_2(25)$. Since $R$ is not $\mathcal{H}$-essential, it follows that $O^{5'}(\Out_{\mathcal{H}}(R))$ is contained in the unique maximal subgroup of $O^{5'}(\Out_{\mathcal{G}}(R))$ which contains $\Out_S(R)$ and so $\Out_S(R)\normaleq O^{5'}(\Out_{\mathcal{H}}(R))$. Then the Frattini argument implies that $\Out_S(R)\normaleq \Out_{\mathcal{H}}(R)$, a contradiction. 

Thus, $R$ is not contained in an $\mathcal{G}$-conjugate of $E_3$. Hence, by the Alperin--Goldschmidt theorem and using \cref{MonsterEssen}, since $\mathcal{H}=\langle \Aut_{\mathcal{G}}(E_1), \Aut_{\mathcal{G}}(E_2), \Aut_{\mathcal{G}}(S)\rangle_S$ and $R$ is fully $\mathcal{H}$-normalized, $R$ is fully $\mathcal{G}$-normalized and so is $\mathcal{G}$-centric. Finally, since $O_5(\Out_{\mathcal{G}}(R))\le O_5(\Out_{\mathcal{H}}(R))=\{1\}$, we conclude that $R$ is $\mathcal{G}$-centric-radical and comparing with \cref{MonsterRad}, we have a contradiction.
\end{proof}

\begin{proposition}\label[proposition]{HExotic52}
$\mathcal{H}$ is simple.
\end{proposition}
\begin{proof}
Assume that $\mathcal{N}\normaleq \mathcal{H}$ and $\mathcal{N}$ is supported on $T$. Then $T$ is a strongly closed subgroup of $\mathcal{H}$. In particular, $T\normaleq S$ and $Z(S)\le T$. Observe that since $N_{\mathcal{G}}(\mathbf{Q})=\langle N_{\mathcal{G}}(S), N_{\mathcal{G}}(E_2)\rangle_S\le \mathcal{H}$, we have that $N_{\mathcal{H}}(\mathbf{Q})=N_{\mathcal{G}}(\mathbf{Q})$. In particular, $\Aut_{\mathcal{H}}(\mathbf{Q})$ is irreducible on $\mathbf{Q}/Z(S)$. Since $\Aut_{\mathcal{H}}(E_1)=\Aut_{\mathcal{G}}(E_1)$ is irreducible on $Z_2(S)$, we have that $\mathbf{Q}\le T$. Then $E_1=\langle (E_1\cap \mathbf{Q})^{\Aut_{\mathcal{G}}(E_1)}\rangle\le T$ and so $S=E_1\mathbf{Q}=T$. Since $\Aut_{\mathcal{H}}(S)$ is generated by lifted morphisms from $O^{5'}(\Aut_{\mathcal{H}}(E_1))$ and $O^{5'}(\Aut_{\mathcal{H}}(\mathbf{Q}))$, in the language of \cref{p'lemma} we have that $\Aut_{\mathcal{H}}^0(S)=\Aut_{\mathcal{H}}(S)$. Then \cite[Theorem II.9.8(d)]{ako} implies that $\mathcal{H}$ is simple.
\end{proof}

\begin{proposition}\label[proposition]{HExotic5}
$\mathcal{H}$ is exotic.
\end{proposition}
\begin{proof}
Aiming for a contradiction, suppose that $\mathcal{H}=\fs_S(G)$ for some finite group $G$ with $S\in\syl_5(G)$. We may as well assume that $O_5(G)=O_{5'}(G)=\{1\}$ so that $F^*(G)=E(G)$ is a direct product of non-abelian simple groups, all divisible by $5$. Then, as $\mathcal{H}$ is simple and $\fs_{S\cap F^*(G)}(F^*(G))\normaleq \fs_S(G)$, we have that $G$ is simple.

If $G\cong \Alt(n)$ for some $n$ then $m_5(\Alt(n))=\lfloor\frac{n}{5}\rfloor$ by \cite[Proposition 5.2.10]{GLS3} and so $n<25$. But a Sylow $5$-subgroup of $\Alt(25)$ has order $5^6$ and so $G\not\cong\Alt(n)$ for any $n$. If $G$ is isomorphic to a group of Lie type in characteristic $5$, then comparing with \cite[Table 3.3.1]{GLS3}, we see that the groups with a Sylow $5$-subgroup which has $5$-rank $4$ are $\PSL_2(5^4)$, $\PSL_3(25)$, $\PSU_3(25)$, $\PSL_4(5)$ or $\PSU_4(5)$ and none of these examples have a Sylow $5$-subgroup of order $5^9$.

Assume now that $G$ is a group of Lie type in characteristic $r\ne 5$. By \cite[Theorem 4.10.3]{GLS3}, $S$ has a unique elementary abelian subgroup of $5$-rank $4$ unless $G\cong\mathrm{G}_2(r^a), {}^2\mathrm{F}_4(r^a), {}^3\mathrm{D}_4(r^a), \PSU_3(r^a)$ or $\PSL_3(r^a)$.  Moreover, by \cite[Theorem 4.10.2]{GLS3}, there is a normal abelian subgroup $S_T$ of $S$ such that $S/S_T$ is isomorphic to a subgroup of the Weyl group of $G$. But $|S_T|\leq 5^4$ so that $|S/S_T|\geq 5^5$. All of the candidate groups above have Weyl group with $5$-part strictly less than $5^5$ and so $G$ is not isomorphic to a group of Lie type in characteristic $r$.

Finally, checking the orders of the sporadic groups, we have that $\mathrm{M}$ is the unique sporadic simple group with a Sylow $5$-subgroup of order $5^9$. Since $\mathrm{M}$ has $3$ classes of essential subgroups, $G\not\cong\mathrm{M}$ and $\mathcal{H}$ is exotic.
\end{proof}

\begin{proposition}\label[proposition]{DSimple1}
$\mathcal{D}$ is a saturated fusion system and $O^{5'}(\mathcal{D})$ has index $2$ in $\mathcal{D}$.
\end{proposition}
\begin{proof}
In the statement of \cref{JasonAdd}, letting $\mathcal{F}_0=N_{\mathcal{G}}(E_1)$, $V=E_3$ and $\Delta=\Aut_{\mathcal{G}}(E_3)$ we have that $\mathcal{D}^\dagger=\langle \fs_0, \Aut_{\mathcal{G}}(E_3)\rangle_S$ is a proper saturated subsystem of $\mathcal{G}$. Since $E_1$ is characteristic in $S$, we have that $\Aut_{\mathcal{D}^\dagger}(S)=\Aut_{\mathcal{G}}(S)$ and $\Aut_{\mathcal{D}^\dagger}(E_1)=\Aut_{\mathcal{G}}(E_1)$. Since $\Aut_{\mathcal{D}^\dagger}(E_2)\le \Aut_{\mathcal{G}}(E_2)$, if $E_2$ was $\mathcal{D}^\dagger$-essential, then $\mathcal{D}^\dagger=\mathcal{G}$, a contradiction.  Therefore, by the Alperin--Goldschmidt theorem, and using \cref{MonsterEssen}, we have that $\mathcal{D}^\dagger=\langle \Aut_{\mathcal{G}}(E_1), \Aut_{\mathcal{G}}(E_3), \Aut_{\mathcal{G}}(S)\rangle_S=\mathcal{D}$ is saturated.

We note that as $\mathcal{D}<\mathcal{G}$, no $\mathcal{G}$-conjugate of $E_2$ is essential in $\mathcal{D}$. Applying \cref{E3inE15b}, we have that $\Aut_{\mathcal{D}}^*(S)=\Aut_{\mathcal{D}}^0(S)$ and $|\Aut_{\mathcal{D}}(S)/\Aut_{\mathcal{D}}^0(S)|=|C_{\Aut_{\mathcal{D}}(E_1)}(Z(E_1))|/3=2$. Hence, \cref{p'lemma} implies that $O^{5'}(\mathcal{D})$ is the unique proper subsystem of $\mathcal{D}$ of $p'$-index and has index $2$ in $\mathcal{D}$.
\end{proof}

\begin{proposition}\label[proposition]{DSimple2}
$\mathcal{D}^{frc}=O^{5'}(\mathcal{D})^{frc}=\{E_1, E_3^{\mathcal{G}}, S\}$.
\end{proposition}
\begin{proof}
Let $\fs$ be one of $\mathcal{D}$ or $O^{5'}(\mathcal{D})$ and $R$ be a fully $\fs$-normalized, radical, centric subgroup of $S$ not equal to $E_1$, $S$ or a $\mathcal{D}$-conjugate of $E_3$. If $R$ is contained in a $\mathcal{G}$-conjugate of $E_3$, then since $R$ is $\fs$-centric and $E_3$ is elementary abelian, we have a contradiction. Hence $R$ is not contained in a $\mathcal{G}$-conjugate of $E_3$ and by \cref{MonsterEssen} and using that $E_2\not\in\mathcal{E}(\fs)$, $R$ is contained in at most one $\fs$-essential, namely $E_1$. Then, as $E_1$ is $\Aut_{\fs}(S)$-invariant, the extension axiom and the Alperin--Goldschmidt theorem imply that $\Out_{E_1}(R)\normaleq \Out_{\fs}(R)$ and $E_1\le R\le S$, a contradiction.
\end{proof}

\begin{lemma}\label[lemma]{SC5}
$E_1$ is the unique proper non-trivial strongly closed subgroup of $\mathcal{D}$ and $O^{5'}(\mathcal{D})$
\end{lemma}
\begin{proof}
Assume that $T$ is a proper non-trivial strongly closed subgroup of $\fs$, where $\fs$ is one of $\mathcal{D}$ or $O^{5'}(\mathcal{D})$. Then $T\normaleq S$ and so $Z(S)\le T$. Then the irreducibility of $O^{5'}(\Aut_{\mathcal{D}}(E_3))\le \Aut_\fs(E_3)$ on $E_3$ implies that $E_3\le T$. Since $E_1=\langle E_3^S\rangle$ we deduce that $E_1\le T$. Since $E_1$ is $\Aut_{\fs}(S)$-invariant and every essential subgroup of $\fs$ is contained in $E_1$ by \cref{MonsterEssen}, it follows from the Alperin--Goldschmidt theorem that $E_1$ is strongly closed in $\fs$.
\end{proof}

\begin{proposition}\label[proposition]{DSimple}
$O^{5'}(\mathcal{D})$ is a simple saturated fusion system on $S$ and both $\mathcal{D}$ and $O^{5'}(\mathcal{D})$ are exotic.
\end{proposition}
\begin{proof}
If $O^{5'}(\mathcal{D})$ is not simple with $\mathcal{N}\normaleq O^{5'}(\mathcal{D})$ and $\mathcal{N}$ supported on $T<S$ then by \cref{SC5}, $\mathcal{N}$ is supported on $E_1$. By \cite[Proposition I.6.4]{ako}, $\Aut_{\mathcal{N}}(E_1)\normaleq \Aut_{O^{5'}(\mathcal{D})}(E_1)$ so that $\Out_{\mathcal{N}}(E_1)$ is isomorphic to a normal $5'$-subgroup of $\Out_{O^{5'}(\mathcal{D})}(E_1)\cong 3\times \GL_2(5)$. In particular, no $\mathcal{D}$-conjugate of $E_3$ is essential in $\mathcal{N}$ for otherwise we could again lift a cyclic subgroup of order $24$ to $\Aut_{\mathcal{N}}(E_1)$, using the extension axiom. Thus, applying \cref{EssenDeterD5} (or performing the MAGMA calculation on which this relies), we deduce that $\mathcal{E}(\mathcal{N})=\emptyset$ and $E_1=O_5(\mathcal{N})$ so that $E_1\normaleq O^{5'}(\mathcal{D})$, a contradiction by \cref{normalinF}.

Hence, if $O^{5'}(\mathcal{D})$ is not simple then $\mathcal{N}$ is supported on $S$. But then by \cite[Theorem II.9.8(d)]{ako}, we have that $O^{5'}(O^{5'}(\mathcal{D}))<O^{5'}(\mathcal{D})$, a contradiction. Thus $O^{5'}(\mathcal{D})$ is simple.

Assume that there is $\mathcal{N}$, a proper non-trivial normal subsystem of $\mathcal{D}$. Applying \cite[Theorem II.9.1]{ako} and using that $O^{5'}(\mathcal{D})$ is simple, we deduce that $O^{5'}(\mathcal{D})\le \mathcal{N}$ and it quickly follows that $O^{5'}(\mathcal{D})=\mathcal{N}$. Since $E_1$ is strongly closed in $\mathcal{D}$ and $O^{5'}(\mathcal{D})$, by \cref{SCExotic}, we conclude that $\mathcal{D}$ and $O^{5'}(\mathcal{D})$ are exotic. 
\end{proof}

We now begin the task of determining all saturated fusion systems supported on $S$. We first record a lemma limiting the possible combinations of essential subgroups in $\fs$. 

\begin{lemma}\label[lemma]{E3ForceE2}
Let $\fs$ be a saturated fusion system on $S$ with $E_3\in\mathcal{E}(\fs)$. If $E_2\in\mathcal{E}(\fs)$ then $E_2$ is not $\Aut_{\fs}(S)$-invariant and $\{E_2^\mathcal{G}\}\subseteq \mathcal{E}(\fs)$.
\end{lemma}
\begin{proof}
Assume that $E_2\in\mathcal{E}(\fs)$. Then, there is a $3$-element in $\Aut_{\fs}(E_1)$ which centralizes $S/E_1$ and $Z(E_1)$, and normalizes $E_3$, and lifts to some $\alpha\in\Aut_{\fs}(S)$ by the extension axiom. Then $\alpha$ normalizes $S/\mathbf{Q}$ and by \cref{E3inE15b}, we may choose $\alpha$ to normalize $E_3$. Thus, if $E_2$ is $\Aut_{\fs}(S)$-invariant, by coprime action, $\alpha$ centralizes $S/\mathbf{Q}$ and so centralizes $E_3\mathbf{Q}/\mathbf{Q}$. But as an $\langle \alpha\rangle$-module, $E_3\mathbf{Q}/\mathbf{Q}\cong E_3/Z(E_1)$ and coprime action yields that $\alpha$ centralizes $E_3$, a contradiction. Thus, $E_2$ is not $\Aut_{\fs}(S)$ invariant and we deduce that all $\mathcal{G}$-conjugates of $E_2$ are essential in $\fs$.
\end{proof}

\begin{proposition}\label[proposition]{5Triv}
Suppose that $\fs$ is a saturated fusion system on $S$ such that $\mathcal{E}(\fs)\subseteq \{E_i\alpha\}$ for some $i\in\{1,2\}$ and $\alpha\in \Aut_{\mathcal{G}}(S)$. Then either
\begin{enumerate}
    \item $\fs=N_{\fs}(S)$; or
    \item $\fs=N_{\fs}(E_i\alpha)$ where $O^{5'}(\Out_{\fs}(E_i\alpha))\cong \SL_2(5)$ for $i\in\{1,2\}$.
\end{enumerate}
\end{proposition}
\begin{proof}
If $\mathcal{E}(\fs)=\emptyset$, then (i) holds by the Alperin--Goldschmidt theorem. If $\mathcal{E}(\fs)=\{E_1\}$ or $\{E_2\alpha\}$ for some $\alpha\in \Aut_{\mathcal{G}}(S)$, then (ii) holds by \cref{MSL21} and \cref{MSL2}.
\end{proof}

\begin{proposition}\label[proposition]{5QAct}
Suppose that $\fs$ is a saturated fusion system on $S$ with $\{E_2^{\mathcal{G}}\}\subseteq \mathcal{E}(\fs)$. Then $O^{5'}(\Out_{\fs}(\mathbf{Q}))\cong 2.\mathrm{J}_2$, $\mathcal{E}(N_{\fs}(\mathbf{Q}))=\{E_2^{\mathcal{G}}\}$ and if $\{E_2^{\mathcal{G}}\}=\mathcal{E}(\fs)$ then $\fs=N_{\fs}(\mathbf{Q})$.
\end{proposition}
\begin{proof}
Assume that $\{E_2^{\mathcal{G}}\}\subseteq \mathcal{E}(\fs)$. Note that $\mathbf{Q}\normaleq N_{\fs}(E_2\alpha)\le N_{\fs}(\mathbf{Q})$ for all $\alpha\in\Aut_{\mathcal{G}}(S)$. Then \cref{normalinF} implies that $\{E_2^{\mathcal{G}}\}\subseteq\mathcal{E}(N_{\fs}(\mathbf{Q}))$, $O_5(N_{\fs}(\mathbf{Q}))=\mathbf{Q}$ and $\fs=N_{\fs}(\mathbf{Q})$ whenever $\{E_2^{\mathcal{G}}\}=\mathcal{E}(\fs)$. Furthermore, any essential subgroup of $N_{\fs}(\mathbf{Q})$ contains $\mathbf{Q}$ as $\mathbf{Q}\not\le E_1$ an appeal to \cref{MonsterEssen} gives $\mathcal{E}(N_{\fs}(\mathbf{Q}))=\{E_2^{\mathcal{G}}\}$. Finally, $O^{5'}(\Out_{\fs}(\mathbf{Q}))$ satisfies the hypothesis of \cref{J2Iden} so that $O^{5'}(\Out_{\fs}(\mathbf{Q}))\cong 2.\mathrm{J}_2$.
\end{proof}

\begin{proposition}\label[proposition]{5RAct}
Suppose that $\fs$ is a saturated fusion system on $S$ with $\{E_1, E_2\alpha\}\subseteq \mathcal{E}(\fs)$ for some $\alpha\in\mathcal{G}$. Then $O^{5'}(\Out_{\fs}(\mathbf{R}\alpha))\cong \PSL_3(5)$, $\mathcal{E}(N_{\fs}(\mathbf{R}\alpha))=\{E_1, E_2\alpha\}$ and if $\{E_1, E_2\alpha\}=\mathcal{E}(\fs)$ then $\fs=N_{\fs}(\mathbf{R}\alpha)$.
\end{proposition}
\begin{proof}
Assume that $\mathcal{E}(\fs)=\{E_1, E_2\alpha\}$ for some $\alpha\in \Aut_{\mathcal{G}}(S)$. Adjusting by an automorphism of $S$ if necessary, we may as well assume that $\mathcal{E}(\fs)=\{E_1, E_2\}$. Since any $\Aut_{\fs}(S)$-conjugate of $E_2$ is also essential in $\fs$, we infer from this that $E_2$ is $\Aut_{\fs}(S)$-invariant. In particular, since $\mathbf{R}\normaleq N_{\fs}(E_2)$ by \cref{MSL21}, $\mathbf{R}$ is normalized by $\Aut_{\fs}(S)$. By \cref{MSL2}, $O^{5'}(\Aut_{\fs}(E_1))$ normalizes $\mathbf{R}$ and so applying the extension axiom and a Frattini argument to $\Aut_{\fs}(E_1)$, we deduce that $\mathbf{R}$ is normalized by $\Aut_{\fs}(E_1)$. In particular, $\{E_1, E_2\}\subseteq \mathcal{E}(N_{\fs}(\mathbf{R}))$. 

Note that if $\mathbf{R}\le E_2\alpha\ne E_2$ for some $\alpha\in \Aut_{\mathcal{G}}(S)$, we have that $\mathbf{R}\le E_2\cap E_2\alpha=\mathbf{Q}$, a contradiction. Hence, by \cref{normalinF}, we see that $\mathcal{E}(N_{\fs}(\mathbf{R}\alpha))=\{E_1, E_2\alpha\}$. Since $\Aut_{\fs}(E_2)$ acts irreducibly on $\mathbf{Q}/\Phi(E_2)\cong E_2/\mathbf{R}$, we conclude that $\mathbf{R}=O_5(N_{\fs}(\mathbf{R}))$ and so if $\{E_1, E_2\}=\mathcal{E}(\fs)$ then $\fs=N_{\fs}(\mathbf{R})$. Moreover, the actions described in \cref{MSL21} and \cref{MSL2} imply that the only non-trivial normal subgroups of $\fs$ are $\mathbf{R}$ and $\Phi(\mathbf{R})$. Since $\mathrm{M}=\langle N_{\mathrm{M}}(S), M_4\rangle$, where $\mathrm{M}$ is the Monster, we see that $\Phi(\mathbf{R})$ is not characteristic in $S$. In particular, no non-trivial characteristic subgroup of $S$ is normal in $\fs$.

By \cref{model}, there is a finite group $G$ with $S\in\syl_5(G)$, $N_{\fs}(\mathbf{R})=\fs_S(G)$ and $F^*(G)=\mathbf{R}$. Moreover, by the uniqueness of models provided in \cref{model} we can embed the models of $N_{\fs}(S)$, $N_{\fs}(E_1)$ and  $N_{\fs}(E_2)$, which we write as $G_{12}$, $G_1$ and $G_2$ respectively, into $G$. Indeed, by the Alperin--Goldschmidt theorem, we may as well assume that $G=\langle G_1, G_2\rangle$ and $G_{12}=G_1\cap G_2$. Then the triple $(G_1/\mathbf{R}, G_2/\mathbf{R}, G_{12}/\mathbf{R})$ along with the appropriate induced injective maps forms a weak BN-pair of rank $2$, and since $S/\mathbf{R}\cong 5^{1+2}_+$, applying \cite[Theorem A]{Greenbook} and using the terminology there, we deduce that $O^{5'}(G)/\mathbf{R}$ is locally isomorphic to $\PSL_3(5)$. By \cite[Theorem 1]{PushWBN}, $O^{5'}(G)/\mathbf{R}\cong \PSL_3(5)$, and $\mathbf{R}/Z(\mathbf{R})$ and $Z(\mathbf{R})$ are dual natural modules for $O^{5'}(G)/\mathbf{R}$. Hence, we have that $O^{5'}(\Out_{\fs}(\mathbf{R}))\cong \PSL_3(5)$, as desired.
\end{proof}

\begin{remark}
In the above, the groups of shape $5^{3+3}.\PSL_3(5)$ come from a situation where a weak BN-pair of rank $2$ of type $\PSL_3(5)$ is pushed up. Indeed, this case occurs as outcome (12) of \cite[Theorem 1]{PushWBN} with the stipulation that $q=5$. There, this phenomena could also occur for $q=3^n$ for all $n\in\N$. We speculate that these cases could result in a class of interesting fusion systems. In particular, when $q=3$, a similar Sylow subgroup already supports the $3$-fusion categories of $\Omega_7(3)$, $\mathrm{Fi}_{22}$ and ${}^2\mathrm{E}_6(2)$. We note however that in our case $S$ is \emph{not} isomorphic to a Sylow $5$-subgroup of $\Omega_7(5)$.
\end{remark}

\begin{proposition}\label[proposition]{5nontriv}
Suppose that $\fs$ is a saturated fusion system on $S$. Then $O_5(\fs)=\{1\}$ if and only if $\mathcal{E}(\fs)=\{E_1, E_2^{\mathcal{G}}\}$ or $E_3\in\mathcal{E}(\fs)$.
\end{proposition}
\begin{proof}
Suppose first that $\mathcal{E}(\fs)=\{E_1, E_2^{\mathcal{G}}\}$. By \cref{5QAct}, we have that $\Out_{\fs}(\mathbf{Q})\cong 2.\mathrm{J}_2$ acts irreducibly on $Q_2/Z(S)$. As $Q_2\not\le E_1$, by \cref{normalinF}, we conclude that $O_5(\fs)\le Z(S)$. But by \cref{MSL2}, $\Out_{\fs}(E_1)$ acts irreducibly on $Z(E_1)$ and we conclude that $O_5(\fs)=\{1\}$. If $E_3\in\mathcal{E}(\fs)$ for some $\alpha\in\mathcal{G}$ then $O_5(\fs)=\{1\}$ by \cref{E3inE152}.

Suppose that $O_5(\fs)=\{1\}$. By \cref{E3inE15b}, if $E_3\not\in\mathcal{E}(\fs)$ then no $\mathcal{G}$-conjugate of $E_3$ is contained in $\mathcal{E}(\fs)$. Then \cref{MonsterEssen} and \cref{5Triv}-\cref{5RAct} imply that $\mathcal{E}(\fs)=\{E_1, E_2^{\mathcal{G}}\}$ as desired.
\end{proof}

As a consequence of this result, if $O_5(\fs)\ne\{1\}$ then $\fs$ is described in \cref{5Triv}-\cref{5RAct}. We additionally note that if $E_3\alpha\in\mathcal{E}(\fs)$ then \cref{E3inE152} implies that $E_1\in\mathcal{E}(\fs)$ and \cref{E3ForceE2} implies that either $\{E_2^{\mathcal{G}}\}\cap \mathcal{E}(\fs)=\emptyset$ or $\{E_2^{\mathcal{G}}\}\subset \mathcal{E}(\fs)$. 

\begin{lemma}\label[lemma]{E1Unique5}
Suppose that $\fs_1, \fs_2$ are two saturated fusion systems supported on $S$. If $E_1\in\mathcal{E}(\fs_1)\cap \mathcal{E}(\fs_2)$ and $N_{\fs_1}(S)=N_{\fs_2}(S)$ then $N_{\fs_1}(E_1)=N_{\fs_2}(E_1)$.
\end{lemma}
\begin{proof}
We know that $E_1$ is characteristic in $S$ so that $N_{\fs_1}(E_1)\ge N_{\fs_1}(S)\le N_{\fs_2}(E_1)$. By \cite[Proposition 2.11]{BobTodd}, it suffices to show that $\Aut_{\fs_1}(E_1)=\Aut_{\fs_2}(E_1)$ and that the homomorphism $H^1(\Out_{\fs}(E_1); Z(E_1))\to H^1(\Out_{N_{\fs_1}(S)}(E_1); Z(E_1))$ induced by restriction is surjective. We observe by \cref{MSL2} and \cref{E3inE15a} that $\Out_{\fs_i}(E_1)$ contains a subgroup isomorphic to $3\times \GL_2(5)$ of index at most $2$ for $i\in\{1,2\}$. Moreover, since $E_1$ is characteristic in $S$, all morphisms in $\Aut_{\fs_i}(S)$ restrict faithfully to morphisms in $\Aut_{\fs_i}(E_1)$ for $i\in\{1,2\}$. In particular, $\Aut_{\fs_i}(S)$ is generated by lifted morphisms in $N_{\Aut_{\fs_i}(E_1)}(\Aut_S(E_1))$.

Let $K$ be a Hall $5'$-subgroup of $N_{\Aut_{\fs_1}(E_1)}(\Aut_S(E_1))=\Aut_{N_{\fs_1}(S)}(E_1)=\Aut_{N_{\fs_2}(S)}(E_1)$ so that $K\cong 3\times C_4\times C_4$ or $\Sym(3)\times C_4\times C_4$. Then $K$ lifts to a group of automorphisms $\hat{K}$ of $\Aut(S)$ with $\hat{K}\Inn(S)=\Aut_{\fs_1}(S)=\Aut_{\fs_2}(S)$. We calculate that $|\Aut(S)|_{5'}=|\Out_{\mathcal{G}}(S)|$ in \cref{code} (but \cref{ThmC1} also provides a genuine proof). In particular, $\hat{K}$ is either a Hall $5'$-subgroup itself or is the centralizer of the unique Sylow $3$-subgroup of a Hall $5'$-subgroup of $\Aut(S)$. Either way set $L:=C_K(O_3(K))$ so that $L\cong 3\times C_4\times C_4$.

We note that $C_{\Out_{\fs_i}(E_1)}(O_3(\Out_{\fs_i}(E_1)))\cong 3\times \GL_2(5)$ and $Z(C_{\Out_{\fs_i}(E_1)}(O_3(\Out_{\fs_i}(E_1))))$ is cyclic of order $12$ for $i\in\{1,2\}$. Indeed, there is a unique subgroup $L^*$ of $L$ cyclic of order $12$ such that $[L^*|_{E_1}\Inn(E_1)/\Inn(E_1), C_{\Out_{\fs_i}(E_1)}(O_3(\Out_{\fs_i}(E_1)))]=\{1\}$. In particular, by a Frattini argument, we see that $C_{\Aut_{\fs_i}(E_1)}(L^*|_{E_1})\Inn(E_1)$ is the preimage in $\Aut_{\fs_i}(E_1)$ of $C_{\Out_{\fs_i}(E_1)}(O_3(\Out_{\fs_i}(E_1)))$. We verify using MAGMA (See \cref{code}) that $C_{\Aut(E_1)}(L^*|_{E_1})\cong 3\times \GL_2(5)$ and so we have that $C_{\Aut_{\fs_1}(E_1)}(L^*|_{E_1})\Inn(E_1)=C_{\Aut_{\fs_2}(E_1)}(L^*|_{E_1})\Inn(E_1)$. Finally, since $N_{\Aut_{\fs_1}(E_1)}(\Aut_S(E_1))=\Aut_{N_{\fs_1}(S)}(E_1)=\Aut_{N_{\fs_2}(S)}(E_1)$, a Frattini argument implies that $\Aut_{\fs_1}(E_1)=\Aut_{\fs_2}(E_1)$.

It remains to prove that the homomorphism $H^1(\Out_{\fs}(E_1); Z(E_1))\to H^1(\Out_{N_{\fs_1}(S)}(E_1); Z(E_1))$ induced by restriction is surjective. We observe by \cref{MSL2} and \cref{E3inE15a} that $\Out_{\fs_1}(E_1)\cong 3\times \GL_2(5)$ or $\Sym(3)\times \GL_2(5)$. One can compute (e.g. in MAGMA as in \cref{code}) that $H^1(\Out_{N_{\fs_1}(S)}(E_1); Z(E_1))=\{1\}$. Hence, the result.
\end{proof}

\begin{lemma}\label[lemma]{E3Unique5}
Suppose that $\fs_1, \fs_2$ are two saturated fusion systems supported on $T$ where $E_1\le T\le S$. If $E_3\in\mathcal{E}(\fs_1)\cap \mathcal{E}(\fs_2)$ and $N_{\fs_1}(E_1)=N_{\fs_2}(E_1)$ then $\Aut_{\fs_1}(E_3)=\Aut_{\fs_2}(E_3)$.
\end{lemma}
\begin{proof}
By \cref{E3inE152}, we have that $O^{5'}(\Aut_{\fs_i}(E_3))\cong\SL_2(25)$ for $i\in\{1,2\}$. Write $X:=O^{5'}(\Aut_{\fs_1}(E_3))$ and $Y:=O^{5'}(\Aut_{\fs_2}(E_3))$. Set $K:=N_{\Aut_{\fs_1}(E_3)}(\Aut_T(E_3))$ so that, by the extension axiom, all morphisms in $K$ lift to morphisms in $\Aut_{\fs_1}(E_1)=\Aut_{\fs_2}(E_1)$. In particular, by \cref{E3inE15a} \[K=N_{\Aut_{N_{\fs_1}(E_1)}(E_3)}(\Aut_T(E_3))=N_{\Aut_{N_{\fs_2}(E_1)}(E_3)}(\Aut_T(E_3))=N_{\Aut_{\fs_2}(E_3)}(\Aut_T(E_3)).\]
 
Let $L$ be a cyclic subgroup of order $24$ in a Hall $5'$-subgroup of $K$ arranged such that $K_L:=L\Aut_S(E_3)=N_{O^{5'}(\Aut_{\fs_1}(E_3))}(\Aut_S(E_3))$. Then $K_L\le X\cap Y\le \Aut(E_1)\cong\GL_4(5).$ We record that there is a unique conjugacy class of subgroups isomorphic to $\SL_2(25)$ in $\GL_4(5)$ (see \cref{code}). Hence, there is $g\in\Aut(E_3)$ with $Y=X^g$.

Then $K_L, (K_L)^g\le Y$ and so there is $y\in Y$ such that $(K_L)^g=(K_L)^y$. Thus, we have that $X^{gy^{-1}}=X^g$ and we calculate that $gy^{-1}\le N_{\GL_4(5)}(K_L)\le N_{\GL_4(5)}(X)$ (see \cref{code}). But then $X=X^g=Y$. By a Frattini argument, $\Aut_{\fs_1}(E_3)=XK=YK=\Aut_{\fs_2}(E_3)$.
\end{proof}

\begin{theorem}\label[theorem]{ThmC1}
Suppose that $\fs$ is saturated fusion system on $S$. If $\{E_1, E_2^{\mathcal{G}}\}\subseteq\mathcal{E}(\fs)$ then $\fs\cong \mathcal{G}$ or $\mathcal{H}$.
\end{theorem}
\begin{proof}
We observe by the Alperin--Goldschmidt theorem, \cref{MonsterEssen} and \cref{E3inE15b} that either $\mathcal{E}(\fs)=\{E_1, E_2^{\mathcal{G}}\}$ or $\mathcal{E}(\fs)=\{E_1, E_2^{\mathcal{G}}, E_3^{\mathcal{G}}\}$. Moreover, applying \cref{E1Unique5} and \cref{E3Unique5}, if $N_{\fs}(S)=N_{\mathcal{G}}(S)$ and $N_{\fs}(\mathbf{Q})=N_{\mathcal{G}}(\mathbf{Q})$ then the Alperin--Goldschmidt theorem and \cref{MonsterEssen} yields that $\fs=\mathcal{G}$ or $\fs=\mathcal{H}$ depending on whether or not $E_3\in\mathcal{E}(\fs)$.

Since $\{E_2^{\mathcal{G}}\}\subseteq\mathcal{E}(\fs)$, we have by \cref{5QAct} that $\mathbf{J}=O_5(N_{\fs}(\mathbf{J}))$ and $O^{5'}(\Out_{\fs}(\mathbf{J}))\cong 2.\mathrm{J}_2$ acts trivially on $Z(S)$. Let $L$ be a complement to $\Aut_S(E_1)$ in $N_{O^{5'}(\Aut_{\fs}(E_1))}(\Aut_S(E_1))$, recalling that $O^{5'}(\Out_{\fs}(E_1))\cong \SL_2(5)$ by \cref{MSL2}. Then $L$ acts faithfully on $Z(S)$ and lifts by the extension axiom to $\hat{L}\le \Aut_{\fs}(S)$ which also acts faithfully on $Z(S)$. Since $\mathbf{J}$ is characteristic in $S$, $\hat{L}|_{\mathbf{Q}}$ induces a cyclic subgroup of $\Aut_{\fs}(\mathbf{J})$ of order $4$ which acts faithfully on $Z(S)$. Indeed, we have that $|\Out_{\fs}(\mathbf{J})/O^{5'}(\Out_{\fs}(\mathbf{J}))|\geq 4$. 

Since a maximal subgroup of $\Out(\mathbf{Q})\cong \Sp_6(5):4$ containing $O^{5'}(\Out_{\fs}(\mathbf{Q}))$ has shape $4.\mathrm{J}_2:2$ by \cite{Winter} and \cite[Table 8.28]{LowMax}, we deduce that $\Out_{\fs}(\mathbf{Q})\cong 4.\mathrm{J}_2:2$ is maximal in $\Out(\mathbf{Q})$. Since any automorphism of $S$ restricts faithfully to an automorphism of $\mathbf{Q}$ which normalizes $\Aut_S(\mathbf{Q})$ and $\Out_{\fs}(\mathbf{Q})$ is maximal, we must have, by the extension axiom, that a Hall $5'$-subgroup of $N_{\Aut_{\fs}(\mathbf{Q})}(\Aut_S(\mathbf{Q}))$ lifts to a Hall $5'$-subgroup of $\Aut(S)$. Indeed, $\Aut_{\fs}(S)$ contains a Hall $5'$-subgroup of $\Aut(S)$, and by a similar reasoning, $\Aut_{\mathcal{G}}(S)$ contains a Hall $5'$-subgroup of $\Aut(S)$. Therefore, there is $\alpha\in\Aut(S)$ such that $\Aut_{\fs}(S)^\alpha=\Aut_{\fs^\alpha}(S)=\Aut_{\mathcal{G}}(S)$ and by \cref{model}, there is $\beta\in\Aut(S)$ with $N_{\fs}(S)^\beta=N_{\fs^\beta}(S)=N_{\mathcal{G}}(S)$.

Let $K$ be the embedding of the restriction of $\Aut_{\fs^\beta}(S)$ to $\mathbf{Q}$ into $\Aut(\mathbf{Q})\cong 5^6:\Sp_6(5).4$. Set $X=\Aut_{\fs^\beta}(\mathbf{Q})$ and $Y=\Aut_{\mathcal{G}}(\mathbf{Q})$ so that $K\le X\cap Y$. We observe that there is one conjugacy class of subgroups isomorphic to $\Aut_{\fs^\beta}(\mathbf{Q})$ in $\Aut(\mathbf{Q})$ and so there is $g\in \Aut(\mathbf{Q})$ with $Y=X^g$. Then $K, K^g\le Y$ and $K, K^g$ are both Sylow $5$-subgroup normalizers in $Y$. Thus, there is $m\in Y$ with $K^m=K^g$ so that $gm^{-1}\in N_{\Aut(\mathbf{Q})}(K)$ and $X^{gm^{-1}}=Y$. We calculate in MAGMA that $N_{\Aut(\mathbf{Q})}(K)\le N_{\Aut(\mathbf{Q})}(X)$ (\cref{code}) so that $X=Y$ and $\Aut_{\fs^\beta}(\mathbf{Q})=\Aut_{\mathcal{G}}(\mathbf{Q})$. Hence, by \cref{model} there is $\beta_2\in\Aut(S)$ such that $N_{\fs^\beta}(\mathbf{Q})^{\beta_2}=N_{\fs^{\beta\beta_2}}(\mathbf{Q})=N_{\mathcal{G}}(\mathbf{Q})$ and $\gamma:=\beta\beta_2\in\Aut(S)$. Then $N_{\fs^\gamma}(S)=N_{N_{\fs^\gamma}(\mathbf{Q})}(S)=N_{N_{\mathcal{G}}(\mathbf{Q})}(S)=N_{\mathcal{G}}(S)$ and we conclude that $\fs^\gamma=\mathcal{H}$ or $\mathcal{G}$. Hence, $\fs\cong \mathcal{H}$ or $\mathcal{G}$, as desired.
\end{proof}

\begin{remark}
The techniques in the above proof can be used to show that the symplectic amalgam $\mathcal{A}_{53}$ found in \cite{parkerSymp} is determined up to isomorphism.
\end{remark}

\begin{theorem}\label[theorem]{ThmC2}
Suppose that $\fs$ is saturated fusion system on $S$ such that $E_2^{\mathcal{G}}\cap \mathcal{E}(\fs)=\emptyset$ and $E_3^{\mathcal{G}}\cap \mathcal{E}(\fs)\ne \emptyset$. Then $\fs\cong \mathcal{D}$ or $O^{5'}(\mathcal{D})$.
\end{theorem}
\begin{proof}
By \cref{MSL2} and \cref{E3inE15a}, we have that $C_{\Out_{\fs}(E_1)}(O_3(\Out_{\fs}(E_1)))\cong 3\times \GL_2(5)$ has index at most $2$ in $\Out_{\fs}(E_1)$ and $\{E_3^\fs\}=\{E_3^{\mathcal{G}}\}$. Let $A$ be the preimage in $\Aut_{\fs}(E_1)$ of this subgroup and consider the group $K:=N_A(\Aut_S(E_1))$. Then, by the extension axiom, $K$ lifts to a subgroup $\hat{K}\le \Aut_{\fs}(S)$ such that $\hat{K}\Inn(S)/\Inn(S)\cong 3\times C_4\times C_4$. In particular, $|\hat{K}\Inn(S)/\Inn(S)|=|\Out_{\mathcal{G}}(S)|/2$. As observed in \cref{ThmC1}, $\Out_{\mathcal{G}}(S)$ is a Hall $5'$-subgroup of $\Out(S)$ and so $\hat{K}\Inn(S)/\Inn(S)$ has index $2$ in some Hall $5'$-subgroup $Y$ of $\Out(S)$ which is conjugate in $\Out(S)$ to $\Out_{\mathcal{G}}(S)$. Indeed, $\hat{K}\Inn(S)/\Inn(S)=C_{Y}(O_3(Y))$. Hence, we have that $\hat{K}$ is conjugate in $\Aut(S)$ to $\Aut_{O^{5'}(\mathcal{D})}(S)=C_{\Aut_{\mathcal{D}}(S)}(O_3(\Aut_{\mathcal{D}}(S)))$. Since $N_{\Aut(S)}(\Aut_{O^{5'}(\mathcal{D})}(S))=\Aut_{\mathcal{G}}(S)$ by \cref{code}, we see that either $\Aut_{\fs}(S)$ is conjugate to $\Aut_{O^{5'}(\mathcal{D})}(S)$ or $\Aut_{\mathcal{G}}(S)$. In particular, applying \cref{model}, $N_{\fs}(S)$ is $\Aut(S)$-conjugate to $N_{O^{5'}(\mathcal{D})}(S)$ or $N_{\mathcal{D}}(S)$.

Suppose that there is $\alpha\in\Aut(S)$ with $N_{\fs}(S)^\alpha=N_{\fs^\alpha}(S)=N_{\mathcal{D}}(S)$. Applying \cref{E1Unique5} we have that $N_{\fs^\alpha}(E_1)=N_{\mathcal{D}}(E_1)$. We note that by \cref{E3inE15b} that $\{E_3^{\mathcal{G}}\}=\{E_3^{\fs^\alpha}\}$. Hence, $\fs^\alpha$ and $\mathcal{D}$ have the same essential subgroups and we deduce by the Alperin--Goldschmidt theorem that $\fs^\alpha=\langle N_{\fs^\alpha}(E_1), \Aut_{\fs^\alpha}(E_3)\rangle_S$. Then \cref{E3Unique5} implies that $\fs^\alpha=\mathcal{D}$ and $\fs\cong \mathcal{D}$.

Suppose now that there is $\alpha\in\Aut(S)$ with $N_{\fs}(S)^\alpha=N_{\fs^\alpha}(S)=N_{O^{5'}(\mathcal{D})}(S)$ so that by \cref{E1Unique5}, we have $N_{\fs^\alpha}(E_1)=N_{O^{5'}(\mathcal{D})}(E_1)$. Then by the Alperin--Goldschmidt theorem, \cref{MonsterEssen} and \cref{E3inE15b}, we have that $\fs^\alpha=\langle N_{\fs^\alpha}(E_1), \Aut_{\fs^\alpha}(E_3)\rangle_S$ and \cref{E3Unique5} implies that $\fs^\alpha=O^{5'}(\mathcal{D})$ and $\fs\cong O^{5'}(\mathcal{D})$, completing the proof.
\end{proof}

We provide the following table summarizing the actions induced by the fusion systems described in \cref{ThmC1} and \cref{ThmC2} on their centric-radical subgroups. The entry ``-" indicates that the subgroup is no longer centric-radical in the subsystem.

\begin{table}[H]
    \caption{$\mathcal{G}$-conjugacy classes of radical-centric subgroups of $S$}
        \begin{tabular}{|c|c|c|c|}\hline
            $P$ & $|P|$ & $\Out_{\mathcal{G}}(P)$ & $\Out_{\mathcal{H}}(P)$\\\hline
            $S$ & $5^{9}$ & $\Sym(3)\times 4\times 4$ & $\Sym(3)\times 4\times 4$ \\\hline
            $E_1$ & $5^8$ & $\Sym(3)\times \GL_2(5)$ & $\Sym(3)\times \GL_2(5)$ \\\hline
            $E_2$ & $5^8$ & $2\times \GL_2(5)$ & $2\times \GL_2(5)$ \\\hline            
            $E_3$ & $5^4$ & $(3\times \SL_2(25)).2$ & - \\ \hline
            $\mathbf{Q}$ & $5^7$ & $4.\mathrm{J}_2:2$ & $4.\mathrm{J}_2:2$\\ \hline
            $\mathbf{R}$ & $5^6$ & $2\times \PSL_3(5)$ & $2\times \PSL_3(5)$\\ \hline            
        \end{tabular}
\end{table}

\begin{table}[H]
    \caption{$\mathcal{G}$-conjugacy classes of radical-centric subgroups of $S$}
\begin{tabular}{|c|c|c|c|}\hline
            $P$ & $|P|$ & $\Out_{\mathcal{D}}(P)$ & $\Out_{O^{5'}(\mathcal{D})}(P)$\\\hline
            $S$ & $5^{9}$ & $\Sym(3)\times 4\times 4$  & $3\times 4\times 4$ \\\hline
            $E_1$ & $5^8$ & $\Sym(3)\times \GL_2(5)$ & $3\times \GL_2(5)$\\\hline
            $E_2$ & $5^8$ & - & - \\\hline            
            $E_3$ & $5^4$ & $(3\times \SL_2(25)).2$ & $3\times \SL_2(25)$\\ \hline
            $\mathbf{Q}$ & $5^7$ & - & - \\ \hline
            $\mathbf{R}$ & $5^6$ & - & - \\ \hline            
        \end{tabular}        
    \label{MTable}
\end{table}

In a similar manner to \cref{F3Sec}, we now construct some additional exotic fusion systems related to the system $\mathcal{D}$ and supported on $E_1$. We note that the lift to $\Aut_{\mathcal{D}}(E_1)$ of a cyclic subgroup of order $24$ in $N_{O^{5'}(\Aut_{\mathcal{D}}(E_3))}(\Aut_{E_1}(E_3))$ projects as a group of order $24$ in the unique normal subgroup of $\Out_{\mathcal{D}}(E_1)\cong \Sym(3)\times \GL_2(5)$ which is isomorphic to $\GL_2(5)$. Indeed, there is a unique up to conjugacy cyclic subgroup of $\GL_2(5)$ of order $24$ which is contained in a unique $5'$-order overgroup, in which it has index $2$. We set $K^*\cong \Sym(3)\times (\mathrm{C}_{24}:2)$ to be the unique $5'$-order overgroup of a chosen cyclic subgroup of order $24$ in $\Out_{\mathcal{D}}(E_1)$ and denote by $K$ its preimage in $\Aut_{\mathcal{D}}(E_1)$. Indeed, $N_{\Aut_{\mathcal{D}}(E_1)}(E_3)$ has index $50$ in $K\Inn(E_1)$.

Let $G$ be a model for $N_{\mathcal{D}}(E_1)$ and let $H$ be a subgroup of $G$ chosen such that $\Aut_H(E_1)=K\Inn(E_1)$. We define the subsystem
\[\mathcal{D^*}=\langle \fs_{E_1}(H), \Aut_{\mathcal{D}}(E_3)\rangle_{E_1} \le \mathcal{D}.\]

We observe that we could have chosen any of the $10$ $\Aut_{\mathcal{D}}(E_1)$-conjugates of $K$ to form a saturated fusion system. By definition, all of the created fusion systems are isomorphic. Moreover, the $\mathcal{D}$-conjugacy class of $E_3$ splits into $10$ classes upon restricting to $\mathcal{D}^*$, which in turn correspond the possible choices of a cyclic subgroup of order $24$.

\begin{proposition}
$\mathcal{D^*}$ is saturated fusion system on $E_1$ and $\mathcal{E}(\mathcal{D}^*)=\{E_3^{\mathcal{D}^*}\}$.
\end{proposition}
\begin{proof}
We create $H$ as in the construction of $\mathcal{D}^*$ and consider $\fs_{E_1}(H)$. Since $\fs_{E_1}(H)\le \mathcal{D}$, and as $E_3$ is fully $\mathcal{D}$-normalized and $N_S(E_3)\le E_1$, $E_3$ is also fully $\fs_{E_1}(H)$-normalized. Since $C_{E_1}(E_3)\le E_3$ we see that $E_3$ is also $\fs_{E_1}(H)$-centric. Finally, since $E_3$ is abelian, it is minimal among $S$-centric subgroups with respect to inclusion and has no proper subgroup of $E_3$ is essential in $\fs_{E_1}(H)$. In the statement of \cref{JasonAdd}, letting $\fs_0=\fs_{E_1}(H)$, $V=E_3$ and $\Delta=\Aut_{\mathcal{D}}(E_3)$, we have that $\wt\Delta:=\Aut_{\fs_{E_1}(H)}(E_3)=N_{\Aut_{\mathcal{D}}(E_3)}(\Aut_S(E_3))$ is strongly $5$-embedded in $\Delta$. By that result, $\mathcal{D^*}=\langle \fs_{E_1}(H), \Aut_{\mathcal{D}}(E_3)\rangle_{E_1}$ is a saturated fusion system.

Since each morphism in $\mathcal{D}^*$ is a composite of morphisms in $\fs_{E_1}(H)$ and $\Aut_{\mathcal{D}}(E_3)$, we must have that an essential subgroup of $\mathcal{D}^*$ is contained in some $H$-conjugate of $E_3$ and so  $\mathcal{E}(\mathcal{D}^*)=\{E_3^{\mathcal{D}^*}\}$.
\end{proof}

\begin{proposition}\label[proposition]{DStar5Pa}
$O^{5'}(\mathcal{D^*})$ has index $4$ in $\mathcal{D^*}$.
\end{proposition}
\begin{proof}
Let $K$ be a Hall $5'$-subgroup of $N_{O^{5'}(\Aut_{\mathcal{D}^*}(E_3))}(\Aut_{E_1}(E_3))$ so that $K$ is cyclic of order $24$. Then $K$ centralizes a Sylow $3$-subgroup of $\Aut_{\mathcal{D}^*}(E_3)$ and, by the extension axiom, lifts to a group of morphisms in $\Aut_{\mathcal{D}^*}(E_1)$ which we denote by $\hat{K}$. Indeed, it follows that $\hat{K}$ centralizes a Sylow $3$-subgroup of $\Aut_{\mathcal{D}^*}(E_1)$, and this holds for all $\mathcal{D}^*$ conjugates of $E_3$. Now, by the definition of $\mathcal{D}^*$, if $R$ is a $\mathcal{D}^*$-centric subgroup which is not equal to a $\mathcal{D}^*$-conjugate of $E_3$ then by the extension axiom, it follows that $\Aut_{E_1}(R)\normaleq \Aut_{\mathcal{D}^*}(R)$ and so $O^{5'}(\Aut_{\mathcal{D}^*}(R))$ is a $5$-group. Hence, we have by definition that $\Out_{\mathcal{D}^*}^0(E_1)$ centralizes a Sylow $3$-subgroup of $\Out_{\mathcal{D}^*}(E_1)$. We observe that the centralizer in $\Out_{\mathcal{D}^*}(E_1)\cong \Sym(3)\times \mathrm{C}_{24}:2$ of a Sylow $3$-subgroup is isomorphic to $3\times \mathrm{C}_{24}$ and so $O^{5'}(\mathcal{D}^*)$ has index at least $4$ in $\mathcal{D}^*$ by \cref{p'lemma}.

Since $\hat{K}$ is cyclic of order $24$, we have that $\Out_{\mathcal{D}^*}^0(E_1)$ is of order at least $24$ and $O^{5'}(\mathcal{D}^*)$ has index at most $8$ in $\mathcal{D}^*$ by \cref{p'lemma}. Aiming for a contradiction, assume that $\Out_{\mathcal{D}^*}^0(E_1)=\hat{K}\Inn(E_1)/\Inn(E_1)$ is cyclic of order $24$. Then $\hat{K}\Inn(E_1)/\Inn(E_1)\normaleq \Out_{\mathcal{D}^*}(E_1)$. But then for $T\in\syl_3(\hat{K}\Inn(E_1)/\Inn(E_1))$, $T\normaleq \Out_{\mathcal{D}^*}(E_1)$, and in the language of \cref{MSL2} we have that $T\in\syl_3(A)$ or $T\in\syl_3(B)$. That is, either $T$ centralizes $Z(E_1)$ or $T$ centralizes $\Phi(E_1)/Z(E_1)$. Since $T$ is induced by the lift of a morphism in $K$, this is a contradiction. Hence, $\Out_{\mathcal{D}^*}^0(E_1)\cong 3\times \mathrm{C}_{24}$ and by \cref{p'lemma}, we have that $O^{5'}(\mathcal{D}^*)$ has index $4$ in $\mathcal{D}^*$. 
\end{proof}

\begin{proposition}\label[proposition]{DStar5P}
$O^{5'}(\mathcal{D}^*)$ is simple and $\mathcal{E}(O^{5'}(\mathcal{D}^*))=\{E_3^{\mathcal{D}^*}\}=\{E_3^{O^{5'}(\mathcal{D}^*)}, E_3\alpha^{O^{5'}(\mathcal{D}^*)}\}$ for some $\alpha\in\mathcal{D}^*$.
\end{proposition}
\begin{proof}
Let $\mathcal{N}\normaleq O^{5'}(\mathcal{D^*})$ supported on $P\le E_1$. By \cref{p'lemma} we may assume that $P<E_1$, and $P$ is strongly closed in $\mathcal{D^*}$. By the irreducible action of $O^{5'}(\Aut_{\mathcal{D}^*}(E_3))$ on $E_3$, we deduce that $E_3\le P$ and since $P\normaleq E_1$, we have that $N_{E_1}(E_3)\le P$. Indeed, as $\Out_{O^{5'}(\mathcal{D}^*)}(E_1)$ acts irreducibly on $E_1/N_{E_1}(E_3)$ we see that $P=N_{E_1}(E_3)$. By \cite[Proposition I.6.4(c)]{ako} we have that $O^{5'}(\Aut_{\mathcal{N}}(E_3))=O^{5'}(\Aut_{O^{5'}(\mathcal{D})}(E_3))\cong \SL_2(25)$.

Let $\tau$ be a non-trivial involution in $Z(O^{5'}(\Aut_{\mathcal{N}}(E_3)))$. By the extension axiom, $\tau$ lifts to $\wt \tau\in\Aut_{O^{5'}(\mathcal{D})}(E_1)$ and restricts to $\hat{\tau}\in \Aut_{O^{5'}(\mathcal{D})}(P)$. Indeed, $\hat{\tau}\in \Aut_{\mathcal{N}}(P)\normaleq \Aut_{O^{5'}(\mathcal{D})}(P)$ and we ascertain that $[\hat{\tau}, \Aut_{E_1}(P)]\le \Inn(P)$. Since $\hat{\tau}$ is the extension of $\tau\in \Aut_{\mathcal{N}}(E_3)$ to $P$, we have that $[\hat{\tau}, \Aut_{E_1}(P)]\le \Aut_{E_3}(P)$. By the extension axiom, we infer that $[\wt \tau, E_1]\le E_3$. But then, as $E_3$ is abelian and $[E_1, E_3, \wt \tau]\le [\Phi(E_1), \hat{\tau}]\le Z(E_1)$, the three subgroups lemma implies that $[E_3, \wt{\tau}, E_1]\le Z(E_1)$ and as $E_3=[E_3, \tau]$ and $Z(E_1)\le E_3$, we have that $E_3\normaleq E_1$, a contradiction. Hence, $O^{3'}(\mathcal{D}^*)$ is simple.

We note that $\Aut_{O^{5'}(\mathcal{D}^*)}(E_1)\le N_{\Aut_{\mathcal{D}^*}(E_1)}(E_3)\Inn(E_1)$ and it follows that $\{E_3^{\mathcal{D}^*}\}$ splits into two conjugacy classes upon restricting to $O^{5'}(\mathcal{D}^*)$. The claim regarding the essential subgroups is then clear.
\end{proof}

\begin{proposition}\label[proposition]{DStar}
There are three proper saturated subsystems of $\mathcal{D}^*$ which properly contain $O^{5'}(\mathcal{D^*})$. Moreover, every saturated subsystem $\fs$ of $\mathcal{D}^*$ of index prime to $5$ is an exotic fusion system, and satisfies $\fs^{frc}=\{E_3^{\mathcal{D}^*}, E_1\}$.
\end{proposition}
\begin{proof}
Applying \cref{p'lemma}, we simply enumerate the proper subgroups of $\Out_{\mathcal{D}^*}(E_1)$ which properly contain $\Out_{O^{5'}(\mathcal{D}^*)}(E_1)$, which gives three non-isomorphic subgroups of shapes $\Sym(3)\times \mathrm{C}_{24}$, $3\times (\mathrm{C}_{24}:2)$ and $(3\times \mathrm{C}_{24}):2$. 

Let $\fs$ be a fusion subsystem of $\mathcal{D}^*$ of index prime to $5$ and assume that $R\in\fs^{frc}$ with $R\ne E_1$. Applying the extension axiom and the Alperin--Goldschmidt theorem, since $R$ is $\fs$-radical, $R$ is contained in at least one $\fs$-essential subgroup. But \cref{DStar5P} then implies that $R$ is contained in a $\mathcal{D}^*$-conjugate of $E_3$. Since $E_3$ is elementary abelian and $R$ is $\fs$-centric, we must have that $R$ is $\mathcal{D}^*$-conjugate to $E_3$, as required.
\end{proof}

\begin{proposition}\label[proposition]{DStarb}
Every saturated subsystem $\fs$ of $\mathcal{D}^*$ of index prime to $5$ is an exotic fusion system.
\end{proposition}
\begin{proof}
Assume that there is $\mathcal{N}$ is a non-trivial normal subsystem of $\fs$. Applying \cite[Theorem II.9.1]{ako} and using that $O^{5'}(\mathcal{D}^*)$ is simple and normal in $\fs$, we deduce that $O^{5'}(\mathcal{D}^*)\le \mathcal{N}$. Hence, every normal subsystem of $\fs$ is supported on $E_1$. 

Suppose that there is a finite group $G$ containing $E_1$ as a Sylow $5$-subgroup with $\fs=\fs_{E_1}(G)$. We may as well assume that $O_5(G)=O_{5'}(G)=\{1\}$, and since $\fs_{F^*(G)\cap E_1}(F^*(G))\normaleq \fs$, we have that $E_1\in\syl_5(F^*(G))$. Since $|\Omega_1(Z(E_1))|=25$, we conclude that $F^*(G)=E(G)$ is a direct product of at most two non-abelian simple groups, which are isomorphic.

If $F^*(G)$ is a direct product of exactly two simple groups, $K_1$ and $K_2$ say, then $N_{N_G(E_1)}(K_i\cap \Omega(Z(E_1)))$ has index at most $2$ in $N_G(E_1)$. But a $3$-element of $\Aut_{\fs}(E_1)$ acts irreducibly on $\Omega(Z(E_1))$ and we have a contradiction. Thus, $F^*(G)$ is simple.

If $F^*(G)\cong \Alt(n)$ for some $n$ then $m_5(\Alt(n))=\lfloor\frac{n}{5}\rfloor$ by \cite[Proposition 5.2.10]{GLS3} and so $n<25$. But a Sylow $5$-subgroup of $\Alt(25)$ has order $5^6$ and so $F^*(G)\not\cong\Alt(n)$ for any $n$. If $F^*(G)$ is isomorphic to a group of Lie type in characteristic $5$, then comparing with \cite[Table 3.3.1]{GLS3}, we see that the groups with a Sylow $5$-subgroup which has $5$-rank $4$ are $\PSL_2(5^4)$, $\PSL_3(25)$, $\PSU_3(25)$, $\PSL_4(5)$ or $\PSU_4(5)$ and none of these examples have a Sylow $5$-subgroup of order $5^8$.

Assume now that $F^*(G)$ is a group of Lie type in characteristic $r\ne 5$. By \cite[Theorem 4.10.3]{GLS3}, $S$ has a unique elementary abelian subgroup of $5$-rank $4$ unless $F^*(G)\cong\mathrm{G}_2(r^a), {}^2\mathrm{F}_4(r^a), {}^3\mathrm{D}_4(r^a), \PSU_3(r^a)$ or $\PSL_3(r^a)$.  Moreover, by \cite[Theorem 4.10.2]{GLS3}, there is a normal abelian subgroup $T$ of $E_1$ such that $E_1/T$ is isomorphic to a subgroup of the Weyl group of $F^*(G)$. But $|T|\leq 5^4$ so that $|E_1/T|\geq 5^4$. All of the candidate groups above have Weyl group with $5$-part strictly less than $5^4$ and so $F^*(G)$ is not isomorphic to a group of Lie type in characteristic $r$.

Finally, no sporadic groups have Sylow $5$-subgroup of order $5^8$ and we conclude that $\fs$ is exotic.
\end{proof}

As observed in \cref{DStar5P}, the $\mathcal{D}^*$-classes of $E_3$ split into two distinct classes upon restriction to $O^{5'}(\mathcal{D}^*)$ (in fact, this holds restricting to $\fs_{E_1}(E_1)$). Indeed, there is a system of index $2$ in $\mathcal{D}^*$ in which this happens and this is the largest subsystem of $\mathcal{D}^*$ in which this happens. This subsystem, which we denote by $\mathcal{L}$, contains $O^{5'}(\mathcal{D}^*)$ with index $2$ and has $\Out_{\mathcal{L}}(E_1)=N_{\Out_{\mathcal{D}^*}(E_1)}(E_3)\cong (3\times \mathrm{C}_{24}):2$.

We may apply \cref{Pruning} to $\mathcal{L}$ and $O^{5'}(\mathcal{D}^*)$, and as the two classes of essential subgroups are fused by an element of $\Aut(E_1)$, regardless of the choice of class we obtain a saturated subsystem defined up to isomorphism. We denote the subsystems obtained by $\mathcal{L}_{\mathcal{P}}$ and $O^{5'}(\mathcal{D}^*)_{\mathcal{P}}$ and the convention we adopt is that $E_3\in\mathcal{E}(\mathcal{L}_{\mathcal{P} })\cap \mathcal{E}(O^{5'}(\mathcal{D}^*)_{\mathcal{P}})$. It is clear from \cref{Pruning} that $\mathcal{E}(\mathcal{L}_{\mathcal{P}})=\mathcal{E}(O^{5'}(\mathcal{D}^*)_{\mathcal{P}})=\{E_3^{\mathcal{L}}\}$.

\begin{proposition}\label[proposition]{LPrune5}
$O^{5'}(\mathcal{L}_{\mathcal{P}})$ has index $6$ in $\mathcal{L}_{\mathcal{P}}$ and is simple. Moreover, $N_{E_1}(E_3)$ is the unique proper non-trivial strongly closed subgroup in every
saturated subsystem $\fs$ of $\mathcal{L}_{\mathcal{P}}$ which contains $O^{5'}(\mathcal{L}_{\mathcal{P}})$.
\end{proposition}
\begin{proof}
It is immediate from \cref{p'lemma} and \cref{DStar5Pa} that $O^{5'}(\mathcal{D}^*)_{\mathcal{P}}$ has index $2$ in $\mathcal{L}_{\mathcal{P}}$ and $O^{5'}(\mathcal{L}_{\mathcal{P}})$ has index prime to $5$ in $O^{5'}(\mathcal{D}^*)_{\mathcal{P}}$. Hence, $O^{5'}(\mathcal{L}_{\mathcal{P}})=O^{5'}(O^{5'}(\mathcal{D}^*)_{\mathcal{P}})$ and for the first part of the lemma, it suffices to prove that $O^{5'}(\mathcal{L}_{\mathcal{P}})$ has index $3$ in $O^{5'}(\mathcal{D}^*)_{\mathcal{P}}$. Note that $O^{5'}(\Aut_{O^{5'}(\mathcal{D}^*)_{\mathcal{P}}}(E_3))\cong \SL_2(25)$ and that, as in \cref{DStar5Pa} we can select a cyclic subgroup of order $24$ labeled $K$ which lifts to a subgroup $\hat{K}$ of $\Aut_{O^{5'}(\mathcal{D}^*)_{\mathcal{P}}}(E_1)$ and $\hat{K}\Inn(E_1)\normaleq \Aut_{O^{5'}(\mathcal{D}^*)_{\mathcal{P}}}(E_1)$. For $R$ a $O^{5'}(\mathcal{D}^*)_{\mathcal{P}}$-centric subgroup, we have that either $\Aut_{E_1}(R)\normaleq \Aut_{O^{5'}(\mathcal{D}^*)_{\mathcal{P}}}(R)$ or that $R$ is $O^{5'}(\mathcal{D}^*)_{\mathcal{P}}$-conjugate to $E_3$. Then as $\hat{K}/\Inn(E_1)\normaleq \Aut_{O^{5'}(\mathcal{D}^*)_{\mathcal{P}}}(E_1)$, it follows that $\hat{K}\Inn(E_1)=\Aut_{O^{5'}(\mathcal{D}^*)_{\mathcal{P}}}^0(E_1)$ and $O^{5'}(\mathcal{L}_{\mathcal{P}})$ has index $6$ in $\mathcal{L}_{\mathcal{P}}$.

Let $\fs$ be a saturated subsystem $\fs$ of $\mathcal{L}_{\mathcal{P}}$ which contains $O^{5'}(\mathcal{L}_{\mathcal{P}})$. Then $O^{5'}(\Aut_{\fs}(E_3))\cong \SL_2(25)$ acts irreducibly on $E_3$. Hence, if $P$ is a non-trivial strongly closed subgroup of $\fs$, then since $P\cap Z(E_1)\ne\{1\}$, we infer that $E_3\le P$. Indeed, $N_{E_1}(E_3)=\langle E_3^{E_1}\rangle\le P$. Note that $\Aut_{\fs}(E_1)=N_{\Aut_{\fs}(E_1)}(E_3)\Inn(E_1)$ and so $N_{E_1}(E_3)$ contains all essential subgroups of $\fs$ and is normalized by $\Aut_{\fs}(E_1)$. Hence, $N_{E_1}(E_3)$ is strongly closed in $\fs$ and as $\Aut_{\fs}(E_1)$ acts irreducibly on $E_1/N_{E_1}(E_3)$, $N_{E_1}(E_3)$ is the unique proper non-trivial strongly closed subgroup of $\fs$.

Let $\mathcal{N}$ be a proper non-trivial normal subsystem of $O^{5'}(\mathcal{L}_{\mathcal{P}})$. Then by \cref{p'lemma}, we may assume that $\mathcal{N}$ is supported on $N_{E_1}(E_3)$. We then repeat parts of the proof of \cref{DStar5P} with $O^{5'}(\mathcal{L}_{\mathcal{P}})$ in place of $O^{5'}(\mathcal{D}^*)$ to see that $E_3\normaleq E_1$, a contradiction. Hence, $O^{5'}(\mathcal{L}_{\mathcal{P}})$ is simple, completing the proof.
\end{proof}

\begin{proposition}\label[proposition]{DStarPrune}
Up to isomorphism, there are two proper saturated subsystems of $\mathcal{L}_{\mathcal{P}}$ which properly contain $O^{5'}(\mathcal{L}_{\mathcal{P}})$, one of which has index $3$ while the other, $O^{5'}(\mathcal{D}^*)_{\mathcal{P}}$, has index $2$. Furthermore, every saturated subsystem $\fs$ of $\mathcal{L}_{\mathcal{P}}$ which contains $O^{5'}(\mathcal{L}_{\mathcal{P}})$ is an exotic fusion system, and satisfies $\fs^{frc}=\{E_3^{O^{5'}(\mathcal{D^*})}, E_1\}$.
\end{proposition}
\begin{proof}
As in \cref{DStar}, applying \cite[Theorem I.7.7]{ako}, we enumerate proper subgroups of $\Out_{\mathcal{L}_{\mathcal{P}}}(E_1)$ which properly contain $\Out_{O^{5'}(\mathcal{L}_{\mathcal{P}})}(E_1)$, noting that this corresponds to calculating subgroups of $\Sym(3)$. Thus, there is a unique subsystem of index $2$ and three systems of index $3$ which, since they are all conjugate under an automorphism of $E_1$, are pairwise isomorphic. Since $O^{5'}(\mathcal{D}^*)_{\mathcal{P}}$ has index $2$ in $\mathcal{L}_{\mathcal{P}}$, we have verified the first part of the proposition. We now let $\fs$ be a fusion subsystem of $\mathcal{L}_{\mathcal{P}}$ of index prime to $5$.

Assume that $R$ in $\fs^{frc}$ but not equal to $E_1$. Applying the extension axiom and the Alperin--Goldschmidt theorem, since $R$ is $\fs$-radical, $R$ is contained in at least one $\fs$-essential subgroup. But then $R$ is contained in a $\mathcal{L}$-conjugate of $E_3$. Since $E_3$ is elementary abelian and $R$ is $\fs$-centric, we must have that $R$ is $\mathcal{L}$-conjugate to $E_3$, as required.

Assume that there is $\mathcal{N}$ is a proper non-trivial normal subsystem of $\fs$. Applying \cite[Theorem II.9.1]{ako} and using that $O^{5'}(\fs)=O^{5'}(\mathcal{L}_{\mathcal{P}})$ is simple, we deduce that $O^{5'}(\mathcal{L}_{\mathcal{P}})\le \mathcal{N}$ and so $N_{E_1}(E_3)$ supports no normal subsystem of $\fs$. Hence, applying \cref{SCExotic}, we see that $\fs$ is exotic.
\end{proof}

We now determine all fusion systems supported on $E_1$ up to isomorphism. We begin with the following general lemmas.

\begin{lemma}\label[lemma]{TrivCore5}
Suppose that $\fs$ is saturated fusion system on $E_1$ with $P\in E_3^{\mathcal{G}}\cap\mathcal{E}(\fs)\ne \emptyset$. Then $O_5(\fs)=\{1\}$ and $O^{5'}(\Aut_{\fs}(P))\cong\SL_2(25)$.
\end{lemma}
\begin{proof}
Let $P\in E_3^{\mathcal{G}}\cap\mathcal{E}(\fs)$. The proof that $O^{5'}(\Aut_{\fs}(P))\cong\SL_2(25)$ is the same as \cref{E3inE152}. Then by \cref{normalinF}, $O_5(\fs)$ is an $\Aut_{\fs}(P)$-invariant subgroup of $P$ which is also normal in $E_1$, so that $O_5(\fs)=\{1\}$. 
\end{proof}

\begin{lemma}\label[lemma]{DStarAut}
Suppose that $\fs$ is saturated fusion system on $E_1$ with $E_3^{\mathcal{G}}\cap\mathcal{E}(\fs)\ne \emptyset$. Then $\Out_{\fs}(E_1)$ is $\Aut(E_1)$-conjugate to a subgroup of $\Out_{\mathcal{D}^*}(E_1)$.
\end{lemma}
\begin{proof}
Let $P\in E_3^{\mathcal{G}}\cap\mathcal{E}(\fs)$. By \cref{TrivCore5} and the extension axiom, we may lift a cyclic subgroup of order $24$ from $N_{O^{5'}(\Aut_{\fs}(P))}(\Aut_{E_1}(P))$ to $\Aut_{\fs}(E_1)$. This subgroup acts faithfully on $Z(E_1)$ and so injects into $\Aut(E_1)/C_{\Aut(E_1)}(Z(E_1))\cong \GL_2(5)$. Since $\Aut_{\fs}(E_1)/C_{\Aut_{\fs}(E_1)}(Z(E_1))$ is a $5'$-group containing a cyclic subgroup of order $24$, we deduce that $\Aut_{\fs}(E_1)/C_{\Aut_{\fs}(E_1)}(Z(E_1))$ has order at most $48$ and contains a cyclic subgroup of order $24$ of index at most $2$. 

Write $N:=N_{\Aut(E_1)}(\Aut_{\fs}(E_1))C_{\Aut(E_1)}(Z(E_1))$ so that $N/C_{\Aut(E_1)}(Z(E_1))$ has order $48$, and $N$ contains $\Aut_{\fs}(E_1)$. Since $|\Aut_{\mathcal{G}}(E_1)|_{5'}=|\Aut(E_1)|_{5'}$, we have that $|C_{\Aut(E_1)}(Z(E_1))|_{5'}=6$. In particular, $C_{\Aut(E_1)}(Z(E_1))$ is solvable and we conclude that $N$ is solvable.

Since $\Aut_{\mathcal{D}^*}(E_1)C_{\Aut(E_1)}(Z(E_1))/C_{\Aut(E_1)}(Z(E_1))$ has order $48$ (and $\GL_2(5)$ has a unique conjugacy class of groups of order $48$ with a cyclic subgroup of index $2$), we deduce that $\Aut_{\mathcal{D}^*}(E_1)$ is $\Aut(E_1)$-conjugate to a subgroup of $N$. Hence, $\Out_{\mathcal{D}^*}(E_1)$ is $\Out(E_1)$-conjugate to a subgroup of $N/\Inn(E_1)$. But $|N/\Inn(E_1)|_{5'}=2^5.3^2=|\Out_{\mathcal{D}^*}(E_1)|$ and so $\Out_{\mathcal{D}^*}(E_1)$ is $\Out(E_1)$-conjugate to a Hall $5'$-subgroup of $N/\Inn(E_1)$. Since $\Out_{\fs}(E_1)$ is a $5'$-group, $\Out_{\fs}(E)$ lies in a Hall $5'$-subgroup of $N/\Inn(E_1)$ and we deduce that $\Out_{\fs}(E_1)$ is $\Out(E_1)$-conjugate to a subgroup of $\Out_{\mathcal{D}^*}(E_1)$.
\end{proof}

\begin{lemma}\label[lemma]{UniqueCyclic}
There is a unique conjugacy class of cyclic subgroups of order $24$ in $\Out_{\mathcal{D}^*}(E_1)$ whose Sylow $3$-subgroups act non-trivially on $Z(E_1)$ and $\Phi(E_1)/Z(E_1)$. This class contains two subgroups.
\end{lemma}
\begin{proof}
We note that the Sylow $3$-subgroups of $C_{\Out_{\mathcal{D}^*}(E_1)}(Z(E_1))$ and of $C_{\Out_{\mathcal{D}^*}(E_1)}(\Phi(E_1)/Z(E_1))$ are normal in $\Out_{\mathcal{D}^*}(E_1)$. Indeed, these are the unique subgroups of order $3$ which are normal in $\Out_{\mathcal{D}^*}(E_1)$. The rest of the calculation is performed computationally (see \cref{code}).
\end{proof}

The next result is computed in MAGMA (see \cref{code}).

\begin{proposition}\label[proposition]{EssenDeterD5}
Let $\fs$ be a saturated fusion system supported on $E_1$. Then $\mathcal{E}(\fs)\subseteq \{E_3^{\mathcal{D}^*}\}$.
\end{proposition}

Again, we provide some explanation for this without formal proof. The MAGMA calculation performed, as documented in \cref{code}, and the existence of $\mathcal{D}^*$ yields that $\mathcal{E}(\fs)\subseteq \{E_3^{\Aut(E_1)}\}=\{E_3^{\Aut_{\mathcal{D}}(E_1)}\}$. By \cref{DStarAut}, and as we are only interested in classifying fusion systems up to isomorphism, we arrange that $\Aut_{\fs}(E_1)$ is contained in $\Aut_{\mathcal{D}^*}(E_1)$.

Let $P_1, P_2\in\mathcal{E}(\fs)$ so that $P_1$ and $P_2$ are $\mathcal{D}$-conjugate to $E_3$. Further, suppose that $P_1$ and $P_2$ are not $\mathcal{D}^*$-conjugate. Writing $K_{P_i}$ for the lift to $\Aut_{\fs}(E_1)$ of $N_{O^{5'}(\Aut_{\fs}(P_i))}(\Aut_{E_1}(P_i))$, we see that \[K_{P_i}C_{\Aut_{\mathcal{D}}(E_1)}(Z(E_1))\le N_{\Aut_{\mathcal{D}}(E_1)}(\{P_i^\fs\})\le \Aut_{\mathcal{D}^*}(E_1)\] for $i\in\{1,2\}$. In particular, we see that $K_{P_1}C_{\Aut_{\mathcal{D}}(E_1)}(Z(E_1))=K_{P_2}C_{\Aut_{\mathcal{D}}(E_1)}(Z(E_1))$. Let $\alpha\in\Aut_{\mathcal{D}}(E_1)\setminus \Aut_{\mathcal{D}^*}(E_1)$ such that $P_1\alpha=P_2$. Then $N_{\Aut_{\mathcal{D}}(E_1)}(\{P_1^\fs\})\alpha=N_{\Aut_{\mathcal{D}}(E_1)}(\{P_2^\fs\})$. Hence, either $N_{\Aut_{\mathcal{D}}(E_1)}(\{P_i^\fs\})=\Aut_{\mathcal{D}^*}(E_1)$ and $\alpha$ normalizes $\Aut_{\mathcal{D}^*}(E_1)$ or $N_{\Aut_{\mathcal{D}}(E_1)}(\{P_i^\fs\})=K_{P_i}C_{\Aut_{\mathcal{D}}(E_1)}(Z(E_1))$ and $\alpha$ normalizes $K_{P_1}C_{\Aut_{\mathcal{D}}(E_1)}(Z(E_1))$. Either way, we have that $\alpha\in\Aut_{\mathcal{D}^*}(E_1)$, a contradiction.

Hence, $\mathcal{E}(\fs) \subseteq \{P^{\mathcal{D}^*}\}$ where $P$ is some $\mathcal{D}$-conjugate of $E_3$. It remains to show that $P$ and $E_3$ are $\mathcal{D}^*$-conjugate. Assume for a contradiction that this is not the case. We may lift a cyclic subgroup of order $24$ from $N_{O^{5'}(\Aut_{\fs}(P))}(\Aut_{E_1}(P))$ to $\Aut_{\fs}(E_1)$, and denote it $K_P$. Then, by \cref{UniqueCyclic}, $K_P$ is $\mathcal{D}^*$ conjugate the cyclic subgroup of order $24$ which is induced by lifted morphisms from $N_{O^{5'}(\Aut_{\fs}(E_3))}(\Aut_{E_1}(E_3))$. Since $\mathcal{E}(\fs) \subseteq \{P^{\mathcal{D}^*}\}$, we may as well assume that these groups are equal. Hence, we may apply \cref{JasonAdd} to $\fs$, with $V=E_3$ and $\Delta=O^{5'}(\Aut_{\mathcal{D}^*}(E_3))$. It easy to see that we verify the hypothesis there, and so we may construct a saturated fusion system on $E_1$ in which both $E_1$ and $P$ are essential. But by the above, this is a contradiction and we see that $\mathcal{E}(\fs)\subseteq \{E_3^{\mathcal{D}^*}\}$.

\begin{ThmE}\hypertarget{ThmE2}{}
Suppose that $\fs$ is saturated fusion system on $E_1$ such that $E_1\not\normaleq \fs$. Then $\fs$ is either isomorphic to a subsystem of $\mathcal{D}^*$ of $5'$-index, of which there are five, or isomorphic to a subsystem of $\mathcal{L}_{\mathcal{P}}$ of $5'$-index, of which there are four.
\end{ThmE}
\begin{proof}
Since we are only interested in determining $\fs$ up to isomorphism, and as $E_1\not\normaleq \fs$, applying \cref{EssenDeterD5}, we have that $E_3\in\mathcal{E}(\fs)\subseteq \{E_3^{\mathcal{D}^*}\}$. Note that \cref{E3Unique5} holds upon replacing $E_3$ by any $\mathcal{D}^*$ conjugate of $E_3$ and so $\mathcal{E}(\fs)$ and $N_{\fs}(E_1)$ determines $\fs$ completely, by the Alperin--Goldschmidt theorem. By \cref{DStarAut}, we arrange that $\Aut_{\fs}(E_1)$ is a subgroup of $\Aut_{\mathcal{D}^*}(E_1)$. Let $P\in\{E_3^{\mathcal{D}^*}\}$ with $P$ not conjugate to $E_3$ by any element of $E_1$.

Let $K$ be a Hall $5'$-subgroup of $N_{O^{5'}(\Aut_{\fs}(E_3))}(\Aut_{E_1}(E_3))$ so that $K$ is cyclic of order $24$. We note that a Sylow $3$-subgroup acts non-trivially on $Z(E_1)$ and $\Aut_{E_1}(E_3)\cong \Phi(E_1)/Z(E_1)$. By the extension axiom, we let $\hat{K}$ be the lift of $K$ to $\Aut_{\fs}(E_1)$. Then, by \cref{UniqueCyclic}, in $\Out_{\mathcal{D}^*}(E_1)$ there is a unique conjugacy class of cyclic subgroups of order $24$ whose Sylow $3$-subgroup is not contained in $C_{\Out_{\mathcal{D}^*}(E_1)}(Z(E_1))$ or $C_{\Out_{\mathcal{D}^*}(E_1)}(\Phi(E_1)/Z(E_1))$. Indeed, $\hat{K}\Inn(E_1)/\Inn(E_1)$ belongs in this class and again by \cref{UniqueCyclic}, we have two candidates for $\hat{K}\Inn(E_1)$ in $\Aut_{\mathcal{D}^*}(E_1)$ (one coming from the lift of automorphisms of $E_3$ and one coming from the lift of automorphisms of $P$).

We enumerate the possible overgroups of $\hat{K}\Inn(E_1)/\Inn(E_1)\cong \mathrm{C}_{24}$ in $\Out_{\mathcal{D}^*}(E_1)\cong \Sym(3)\times \mathrm{C}_{24}:2$. These are the groups of shape
\[\mathrm{C}_{24},\, \mathrm{C}_{24}:2,\, 3\times \mathrm{C}_{24},\, (3\times \mathrm{C}_{24}):2,\, \Sym(3)\times \mathrm{C}_{24},\, 3\times \mathrm{C}_{24}:2\,\,\text{and}\,\,\Sym(3)\times \mathrm{C}_{24}:2.\]
Note that there are three subgroups of shape $\mathrm{C}_{24}:2$, all conjugate, and every other group is unique. Finally, we note that $\Out_{O^{5'}(\mathcal{L}_{\mathcal{P}})}(E_1)\cong \mathrm{C}_{24}$, $\Out_{O^{5'}(\mathcal{D}^*)}(E_1)\cong 3 \times \mathrm{C}_{24}$ and $\Out_{\mathcal{L}}(E_1)\cong 3 \times \mathrm{C}_{24}:2$.

Suppose first that $P\not\in\mathcal{E}(\fs)$. Hence, $\mathcal{E}(\fs)=\{E_3^{E_1}\}$ and so $\Aut_{\fs}(E_1)\le N_{\Aut_{\fs}(E_1)}(E_3)\Inn(E_1)$. In particular, $\Out_{\fs}(E_1)\le \Out_{\mathcal{L}}(E_1)\cong 3 \times \mathrm{C}_{24}:2$. There are four choices for $\Out_{\fs}(E_1)$ up to conjugacy, and so there are four choices for $\Aut_{\fs}(E_1)$ and these choices correspond exactly with $\Aut_{\mathcal{Y}}(E_1)$ where $\mathcal{Y}$ is a subsystem of $5'$-index in $\mathcal{L}_{\mathcal{P}}$ described in \cref{LPrune5} and \cref{DStarPrune}. By \cref{model}, there is $\alpha\in\Aut(E_1)$ such that $N_{\fs}(E_1)^\alpha=N_{\fs^\alpha}(E_1)=N_{\mathcal{Y}}(E_1)$. If $\mathcal{E}(\fs^\alpha)=\{E_3^{E_1}\}$ then we have that $\fs^\alpha=\mathcal{Y}$ by an earlier observation so that $\fs\cong \mathcal{Y}$. Hence, we have that $P\in\mathcal{E}(\fs^\alpha)$. Then there is $\beta\in\Aut_{\mathcal{D}^*}(E_1)$ such that $P\beta=E_3$ and $N_{\mathcal{Y}}(E_1)^\beta=N_{\mathcal{Y}}(E_1)$. Hence, by an earlier observation using the Alperin--Goldschmidt theorem, we have that $\fs^{\alpha\beta}=\mathcal{Y}$ and so $\fs\cong \mathcal{Y}$. 

Therefore, we may assume that $\mathcal{E}(\fs)=\{E_3^{\mathcal{D}^*}\}$. Let $\beta\in\Aut_{\mathcal{D}^*}(E_1)$ with $E_3\beta=P$. Then $\beta\not\in\Aut_{\mathcal{L}}(E_1)$, $\hat{K}\Inn(E_1)\ne \hat{K}\beta\Inn(E_1)$ and $\hat{K}\beta|_{P}\le O^{5'}(\Aut_{\fs}(P))$. In particular, by \cref{TrivCore5}, we have that $\langle \hat{K}, \hat{K}\beta\rangle \le \Aut_{\fs}(E_1)$ and we infer that $\Out_{\fs}(E_1)$ is an overgroup of $\Out_{O^{5'}(\mathcal{D}^*)}(E_1)\cong 3 \times \mathrm{C}_{24}$
Thus, there are five choices for $\Out_{\fs}(E_1)$ up to conjugacy, and so there are five choices for $\Aut_{\fs}(E_1)$ and these choices correspond exactly with $\Aut_{\mathcal{Y}}(E_1)$ where $\mathcal{Y}$ is a subsystem of $5'$-index in $\mathcal{D}^*$ described in and \cref{DStar}. By \cref{model}, there is $\alpha\in\Aut(S)$ such that $N_{\fs}(E_1)^\alpha=N_{\fs^\alpha}(E_1)=N_{\mathcal{Y}}(E_1)$. Since $\mathcal{E}(\fs^\alpha)=\mathcal{E}(\mathcal{Y})=\{E_3^{\mathcal{D}^*}\}$, by an earlier observation we have that $\fs^\alpha=\mathcal{Y}$ so that $\fs\cong \mathcal{Y}$.
\end{proof}

We provide the following tables summarizing the actions induced by the fusion systems described in \hyperlink{ThmE2}{Theorem E} on their centric-radical subgroups. \cref{SubMTablea} and \cref{SubMTableI} treat those subsystems of $\mathcal{D}$ which are not ``pruned", while \cref{SubMTableII} and \cref{SubMTableIII} deals with the remainder. The entry ``-" indicates that the subgroup is no longer centric-radical in the subsystem, and an entry decorated with ``${}^\dagger$" specifies that there are two conjugacy classes of $E_3$ in this subsystem which are fused upon enlarging to $\mathcal{D}$.

\begin{table}[H]
    \caption{$\mathcal{D}$-conjugacy classes of radical-centric subgroups of $E_1$}
        \begin{tabular}{|c|c|c|c|c|}\hline
            $P$ & $|P|$ & $\Out_{\mathcal{D}^*}(P)$ & $\Out_{O^{5'}(\mathcal{D}^*).2_1}(P)$ & $\Out_{O^{5'}(\mathcal{D}^*).2_2}(P)$ \\\hline
            $E_1$ & $5^8$ & $\Sym(3)\times (\mathrm{C}_{24}:2)$ & $\Sym(3)\times \mathrm{C}_{24}$ & $3\times (\mathrm{C}_{24}:2)$ \\\hline          
            $E_3$ & $5^4$ & $(3\times \SL_2(25)).2$ & $3\times \SL_2(25)$ & $3\times \SL_2(25)$\\ \hline
        \end{tabular}
            \label{SubMTablea}
\end{table}

\begin{table}[H]
        \caption{$\mathcal{D}$-conjugacy classes of radical-centric subgroups of $E_1$}
        \begin{tabular}{|c|c|c|c|}\hline
            $P$ & $|P|$ & $\Out_{\mathcal{L}}(P)$ & $\Out_{O^{5'}(\mathcal{D}^*)}(P)$ \\\hline
            $E_1$ & $5^8$ & $(3\times \mathrm{C}_{24}):2$ & $3\times \mathrm{C}_{24}$ \\\hline          
            $E_3$ & $5^4$ & $(3\times \SL_2(25)).2^\dagger$ & $3\times \SL_2(25)^\dagger$ \\ \hline
        \end{tabular}
    \label{SubMTableI}
\end{table}

\begin{table}[H]
    \caption{$\mathcal{L}$-conjugacy classes of radical-centric subgroups of $E_1$}
        \begin{tabular}{|c|c|c|c|}\hline
            $P$ & $|P|$ & $\Out_{\mathcal{L}_{\mathcal{P}}}(P)$ & $\Out_{O^{5'}(\mathcal{D}^*)_{\mathcal{P}}}(P)$  \\\hline
            $E_1$ & $5^8$ & $(3\times \mathrm{C}_{24}):2$ & $3\times \mathrm{C}_{24}$  \\\hline          
            $E_3$ & $5^4$ & $(3\times \SL_2(25)).2$ & $3\times \SL_2(25)$ \\ \hline
        \end{tabular}
    \label{SubMTableII}
\end{table}

\begin{table}[H]
    \caption{$\mathcal{L}$-conjugacy classes of radical-centric subgroups of $E_1$}
        \begin{tabular}{|c|c|c|c|}\hline
            $P$ & $|P|$ & $\Out_{O^{5'}(\mathcal{L}_{\mathcal{P}}).2}(P)$ & $\Out_{O^{5'}(\mathcal{L}_{\mathcal{P}})}(P)$ \\\hline
            $E_1$ & $5^8$ & $\mathrm{C}_{24}:2$ & $ \mathrm{C}_{24}$ \\\hline          
            $E_3$ & $5^4$ & $\SL_2(25).2$ & $\SL_2(25)$ \\ \hline
        \end{tabular}
    \label{SubMTableIII}
\end{table}

\appendix

\section{magma code for computations}\label{code}

We include the code of various claims throughout this paper. While none of this code is particularly elegant or optimized, it does substantiate the claims made.

We first include Cano's code \cite[pg 34]{cano} for amalgam enumeration in \cref{F3Unique} since it appears not to be freely available from any reliable source.

\begin{Verbatim}[frame=single, framesep=4mm, framerule=0.1mm]
function Amalgams(P1, B1, P2, B2);
tf, isom:=IsIsomorphic(B1,B2);
if tf eq false then
    return "false";
else
    inv:=Inverse(isom);
    AP1:=AutomorphismGroup(P1);
    f1, perAP1:=PermutationRepresentation(AP1);
    g1:=Inverse(f1);

AP2:=AutomorphismGroup(P2);
f2, perAP2:=PermutationRepresentation(AP2);
g2:=Inverse(f2);
AB1:=AutomorphismGroup(B1);
IdAB1:=Identity(AB1);
b1f, perAB1:=PermutationRepresentation(AB1);
b1g:=Inverse(b1f);
genB1:=Generators(B1);
N1:=Normalizer(P1,B1);
genN1:=Generators(N1);
AP1innerB1:=sub<perAP1| {(AP1!hom<P1 -> P1 | x:-> w^-1*x*w>)@ 
f1: w in genN1 }>;
kernelAstara:=sub<perAP1|{(AP1!hom<P1 -> P1 | x:-> w^-1*x*w>)@ 
f1: w in genN1 | AB1! hom<B1->B1| x:-> w^-1*x*w> eq IdAB1} >;
trans1:=Transversal(perAP1, AP1innerB1);
AP1outB1:=sub<perAP1| { w : w in trans1 | B1@ (w@ g1) eq B1 } >;
AP1fixB1:=sub<perAP1| AP1outB1, AP1innerB1>; genAP1fixB1:=
Generators(AP1fixB1);
N2:=Normalizer(P2,B2);
genN2:=Generators(N2);
AP2innerB2:=sub<perAP2| {(AP2!hom<P2 -> P2 | x:-> w^-1*x*w>)@ 
f2 : w in genN2} >;
kernelAstarb:=sub<perAP2|{(AP2!hom<P2 -> P2 | x:-> w^-1*x*w>)@ 
f2: w in genN2 |
AB1!hom<B1->B1| x:-> (w^-1*(x@ isom)*w)@ inv> eq IdAB1}>;
trans2:=Transversal(perAP2, AP2innerB2);
AP2outB2:=sub<perAP2| { w : w in trans2 | B2@ (w@ g2) eq B2 } >;
AP2fixB2:=sub<perAP2| AP2outB2, AP2innerB2>;
genAP2fixB2:=Generators(AP2fixB2);
ordenA1:=Index(AP1fixB1, kernelAstara);
ordenA2:=Index(AP2fixB2, kernelAstarb);
A1:=sub<perAB1|>;

for w in genAP1fixB1 do
    assert exists(q){x:x in perAB1| forall(t){b : b in genB1 |
    b@ (w@ g1) eq b@ (x@ b1g)} eq true};
    A1:=sub<perAB1| A1, q >;
    if Order(A1) eq Index(AP1fixB1, kernelAstara) then
        break;
    end if;
end for;
    A2:=sub<perAB1| >;
for m in genAP2fixB2 do
    assert exists(n){x:x in perAB1| forall(t){ p : p in genB1 | 
    ((p@ isom)@ (m@ g2))@ inv eq p@ (x@ b1g)} eq true }; 
    A2:=sub<perAB1| A2, n>;
    if Order(A2) eq Index(AP2fixB2, kernelAstarb) then
        break;
    end if;
end for;
    
h,y:=CosetAction(perAB1, A1);
O:=Orbits(h(A2));
o:=#O;
INV:=Inverse(h);
rep:={};
for j in [1..o] do
    for x in O[j] do sena:=false;
        if Order((x@ INV)@ b1g) eq 1 then rep:=rep join
        {(x@ INV)@ b1g}; sena:=true;
            break;
        end if;
    end for;
    if sena eq false then
        rep:=rep join {Random(O[j]@ INV)@ b1g};
    end if;
end for;

Rep:=SetToIndexedSet(rep);
return o, Rep, [ Rep[i]*isom : i in [1..o]], isom; end if;

end function;
\end{Verbatim}

\underline{\cref{J2Iden}:}

\begin{Verbatim}[frame=single, framesep=4mm, framerule=0.1mm]
> A:=[];
> G:=Sp(6,5);
> V:=GModule(G);
> M:=MaximalSubgroups(G);
> for H in M do
for> S:=Sylow(H`subgroup,5);
for> if Dimension(Fix(Restriction(V, S))) eq 1 then
for|if> Append(~A, H`subgroup);
for|if> end if;
for> end for;

> /*We use the ChiefFactors command alongside the descriptions in
>  Bray, Holt & Roney-Dougal to identify the cases we handle 
> computationally. */
> A:=[A[1], A[4], A[5]];

> P1:=A[1];
> T1:=Sylow(P1, 5);
> for X in Subgroups(T1:OrderEqual:=5) do 
for> Fix(Restriction(V, X`subgroup));
for> end for;
GModule of dimension 2 over GF(5)
GModule of dimension 2 over GF(5)
GModule of dimension 3 over GF(5)
GModule of dimension 2 over GF(5)
GModule of dimension 2 over GF(5)
GModule of dimension 2 over GF(5)

> P2:=A[2];
> L2:=NormalClosure(P2, Sylow(P2, 5));
> N:=NormalSubgroups(L2:OrderEqual:=5^8*2^4*3)[1]`subgroup;
> T:=Sylow(N, 2);
> C:=Centralizer(L2, T)
> FactoredOrder(C);
[ <2,4>,  <3,1>, <5,1>]
> R:=Sylow(C, 5);
> Fix(Restriction(V, R));
GModule of dimension 5 over GF(5)

> B:=[];
> G:=Sp(4,5);
> X:=Subgroups(G);
> for h in X do
for> H:=h`subgroup;
for> S:=Sylow(H, 5);
for> if #pCore(H,5) eq 1 and H eq NormalClosure(H, S) and 
for> #S eq 25 then
for|if> Append(~B, H);
for|if> end if;
for> end for;
> #B;
2;
> a,b:=IsIsomorphic(B[1], SL(2,25));
> a;
true
> a,b:=IsIsomorphic(B[2], DirectProduct(SL(2,5), SL(2,5)));
> a;
true
> Center(B[1]) eq Center(G);
true
> Center(G) subset Center(B[2]);
true
> Center(G) eq Center(Sylow(G,2));
true

> P3:=A[3];
> L3:=NormalClosure(P3, Sylow(P3, 5));
> Z:=Center(Sylow(L3, 2));
> CC:=Centralizer(L3, CC);
> D:=DerivedSubgroup(CC);
> a,b:=IsIsomorphic(D, Sp(4,5));
> a;
true
> B:=[];
> X:=Subgroups(D);
> for h in X do
for> H:=h`subgroup;
for> S:=Sylow(H, 5);
for> if #pCore(H,5) eq 1 and H eq NormalClosure(H, S) and 
for> #S eq 25 then
for|if> Append(~B, H);
for|if> end if;
for> end for;
> Fix(Restriction(V, Sylow(B[1], 5)));
GModule of Dimension 4 over GF(5)
> Fix(Restriction(V, Sylow(B[2], 5)));
GModule of Dimension 4 over GF(5)
\end{Verbatim}

\underline{\cref{EssenDeter}:}

Much of the calculations rely on the Fusion Systems package in MAGMA and so we append this to the version of MAGMA we use. This package is available at \cite{Webpage}. Instructions of how to use this package are given in \cite{Webpage}.

\begin{Verbatim}[frame=single, framesep=4mm, framerule=0.1mm]
> Attach("...\FusionSystems.m")
\end{Verbatim}

There are numerous ways to get our hands on a Sylow $3$-subgroup of $\mathrm{Co}_1$, but we give perhaps the easiest way. Since a Sylow $3$-subgroup of $\mathrm{Co}_1$ is isomorphic to a Sylow $3$-subgroup of $\Sp_6(3)$, we use this to access $S$ (this isomorphism can be checked in various ways in MAGMA). We also note here that the ``AllProtoEssentials" command gives all possible essential subgroups up to $\Aut(S)$-conjugacy.

\begin{Verbatim}[frame=single, framesep=4mm, framerule=0.1mm]
> S:=PCGroup(Sylow(Sp(6,3),3));
> AllProtoEssentials(S);
[
    GrpPC of order 2187 = 3^7
    PC-Relations:
        $.2^3 = $.7^2, 
        $.2^$.1 = $.2 * $.4^2, 
        $.3^$.2 = $.3 * $.5^2, 
        $.4^$.2 = $.4 * $.6^2, 
        $.5^$.2 = $.5 * $.7^2,

    GrpPC of order 6561 = 3^8
    PC-Relations:
        $.1^3 = $.8^2, 
        $.2^$.1 = $.2 * $.4^2, 
        $.3^$.1 = $.3 * $.5, 
        $.4^$.1 = $.4 * $.7, 
        $.4^$.2 = $.4 * $.6, 
        $.4^$.3 = $.4 * $.7, 
        $.5^$.1 = $.5 * $.8, 
        $.5^$.2 = $.5 * $.7^2, 
        $.5^$.3 = $.5 * $.8,

    GrpPC of order 6561 = 3^8
    PC-Relations:
        $.1^3 = $.8^2, 
        $.2^$.1 = $.2 * $.6, 
        $.3^$.1 = $.3 * $.7, 
        $.3^$.2 = $.3 * $.6^2 * $.8, 
        $.4^$.1 = $.4 * $.6 * $.7^2 * $.8, 
        $.4^$.2 = $.4 * $.8^2, 
        $.4^$.3 = $.4 * $.7^2, 
        $.5^$.2 = $.5 * $.7^2 * $.8^2, 
        $.6^$.1 = $.6 * $.8, 
        $.6^$.4 = $.6 * $.8, 
        $.7^$.2 = $.7 * $.8^2,

    GrpPC of order 6561 = 3^8
    PC-Relations:
        $.2^$.1 = $.2 * $.4^2, 
        $.3^$.1 = $.3 * $.6, 
        $.3^$.2 = $.3 * $.5^2, 
        $.4^$.2 = $.4 * $.8^2, 
        $.4^$.3 = $.4 * $.7^2, 
        $.5^$.1 = $.5 * $.7^2, 
        $.5^$.4 = $.5 * $.8, 
        $.6^$.2 = $.6 * $.7^2 * $.8^2, 
        $.7^$.2 = $.7 * $.8^2
]
\end{Verbatim}

We now create some of the possible essential subgroups given above for use in later claims. This code may also be modified slightly to substantiate some of the uniqueness claims regarding these subgroups as subgroups of $S$. 

\begin{Verbatim}[frame=single, framesep=4mm, framerule=0.1mm]
> S:=PCGroup(Sylow(Sp(6,3),3));
> JS:=AbelianSubgroups(S:OrderEqual:=3^6)[1]`subgroup;;
> T:=sub<S|DerivedSubgroup(S), JS>;
> X:=Subgroups(S:OrderEqual:=3^8);
> Z2S:=UpperCentralSeries(S)[3];
> #Z2S;
9
\end{Verbatim}

\underline{\cref{SingleEssenAct1}:}
\begin{Verbatim}[frame=single, framesep=4mm, framerule=0.1mm]
> for H in X do 
for> if FrattiniSubgroup(H`subgroup) eq FrattiniSubgroup(T) then
for|if> E1:=H`subgroup;
for|if> end if; 
for> end for;
> Center(E1) eq Center(S);
true
> FE1:=FrattiniSubgroup(E1);
> #FE1;
27
> Centralizer(E1, FE1) eq JS;
true
> CommutatorSubgroup(CommutatorSubgroup(JS, S), S) subset FE1;
false
\end{Verbatim}

\underline{\cref{SingleEssenAct2}:}
\begin{Verbatim}[frame=single, framesep=4mm, framerule=0.1mm]
> E2:=Centralizer(S, Z2S);
> FE2:=FrattiniSubgroup(E2);
> #FE2;
243
> FE2 eq CommutatorSubgroup(E2, JS);
true
> #Center(E2);
27
\end{Verbatim}

\underline{\cref{SingleEssenAct3}:}
\begin{Verbatim}[frame=single, framesep=4mm, framerule=0.1mm]
> for H in X do
for> if Agemo(H`subgroup,1) eq Center(S) and
for|if> not(H`subgroup eq E1) then
for|if> E3:=H`subgroup;
for|if> end if;
for> end for;
> Z2S eq UpperCentralSeries(E3)[3];
true
> FE3:=FrattiniSubgroup(E3);
> #FE3;
243
> Centralizer(E3, FE3) eq Z2S;
true
> Center(E2 meet E3) eq Center(E2);
true
> FrattiniSubgroup(E2 meet E3) eq Center(E2);
true
\end{Verbatim}

\underline{\cref{ExtraSpecial}:}
\begin{Verbatim}[frame=single, framesep=4mm, framerule=0.1mm]
> X1:=Subgroups(E1:OrderEqual:=3^5);
> E:=[];
> for U in X1 do
for> if IsNormal(S, U`subgroup) and
for|if> U`subgroup meet JS eq FrattiniSubgroup(E1) and
for|if> not(#Omega(U`subgroup, 1) eq 3^4) then
for|if> Append(~E, U`subgroup);
for|if> end if; 
for> end for;
> #E;
3
> A:=AutomorphismGroup(S);
> phi, Aut:=PermutationRepresentation(A);
> F:=[];
> for U in E do
for> AutUS:=sub<A| {hom<S->S|x:-> x^g> : g in Generators(U)}>;  
for> AutUSp:=sub<Aut| [phi(AutUS.i) : 
for> i in [1..#Generators(AutUS)]]>;
for> Append(~F,AutUSp);    
for> end for;
> IsConjugate(Aut, F[1], F[2]);
true
> IsConjugate(Aut, F[1], F[3]);
true
\end{Verbatim}

\underline{\cref{Co1}:}
\begin{Verbatim}[frame=single, framesep=4mm, framerule=0.1mm]
> G<x,y>:=PermutationGroup<24|\[
2,1,5,4,3,6,10,8,11,7,9,12,16,18,15,13,17,14,21,20,19,22,24,23]
,\[
3,4,6,7,8,9,11,12,13,14,15,16,17,18,19,20,1,21,2,22,23,5,24,10]>;
print "Group G is 2.M12 < Sym(24)";
> V:=IrreducibleModules(G, GF(3))[2];
/* 2M12 has two irreducible modules of dimension 6 over GF(3)
S fixes a subspace of dimension 1 or 2 in these modules.
Since |Z(S)|=3 we want the one with fixed space of dimension 1.*/
> Fix(Restriction(V, Sylow(G,3)));
GModule of dimension 1 over GF(3)
> M1:=MaximalSubgroups(G)[6]`subgroup;
> Fix(Restriction(V, pCore(M1,3)));
GModule of dimension 1 over GF(3)
> T:=Sylow(M1, 3);
>for H in IntermediateSubgroups(G, T) do
for> if IsIsomorphic(H, M1) and 
for> Dimension(Fix(Restriction(V, pCore(H,3)))) eq 1 then
for|if> H eq M1;
for|if> end if;
for> end for;
true
\* Hence we have found E1, distinct from E2, and proved that $K$ 
is the unique overgroup of $T$ satisfying the required 
properties.*/
> CohomologicalDimension(Restriction(V, M1), 1);
0
> X:=MatrixGroup(V);
> K:=MatrixGroup(Restriction(V, M1));
> Normalizer(GL(6,3), K) subset X;
true
\end{Verbatim}

\underline{\cref{F3Essen}:}

As in \cref{F3Unique}, we appeal to \cite{OnlineAtlas} for generators for any maximal subgroup of $\mathrm{F}_3$ which contains $S$ as a Sylow $3$-subgroup (specifically the subgroup $M_2$ of shape $3^{1+2+1+2+1+2}:\GL_2(3))$. We call the group obtained ``$G$" and proceed as below.

\begin{Verbatim}[frame=single, framesep=4mm, framerule=0.1mm]
>S:=PCGroup(Sylow(G,3));
>AllProtoEssentials(S);
[
    GrpPC of order 243 = 3^5
    PC-Relations:
        $.1^3 = Id($), 
        $.2^3 = Id($), 
        $.3^3 = Id($), 
        $.4^3 = Id($), 
        $.5^3 = Id($),

    GrpPC of order 19683 = 3^9
    PC-Relations:
        $.1^3 = $.7^2 * $.9^2, 
        $.3^3 = $.8^2, 
        $.5^3 = $.9^2, 
        $.2^$.1 = $.2 * $.5^2 * $.9^2, 
        $.3^$.1 = $.3 * $.5^2 * $.6^2 * $.9, 
        $.3^$.2 = $.3 * $.6 * $.7 * $.9^2, 
        $.4^$.1 = $.4 * $.6^2 * $.7 * $.8, 
        $.4^$.2 = $.4 * $.8^2 * $.9, 
        $.4^$.3 = $.4 * $.8 * $.9^2, 
        $.5^$.1 = $.5 * $.8^2, 
        $.5^$.2 = $.5 * $.7^2, 
        $.5^$.3 = $.5 * $.7 * $.8^2, 
        $.5^$.4 = $.5 * $.8^2 * $.9^2, 
        $.6^$.1 = $.6 * $.7^2, 
        $.6^$.2 = $.6 * $.8^2 * $.9, 
        $.6^$.3 = $.6 * $.8^2 * $.9^2, 
        $.6^$.5 = $.6 * $.9^2, 
        $.7^$.2 = $.7 * $.9^2, 
        $.7^$.3 = $.7 * $.9^2, 
        $.8^$.1 = $.8 * $.9,

    GrpPC of order 19683 = 3^9
    PC-Relations:
        $.1^3 = $.8, 
        $.3^3 = $.8^2, 
        $.4^3 = $.9^2, 
        $.2^$.1 = $.2 * $.6^2, 
        $.3^$.1 = $.3 * $.5^2, 
        $.3^$.2 = $.3 * $.5 * $.6 * $.8 * $.9^2, 
        $.4^$.1 = $.4 * $.6^2 * $.7, 
        $.4^$.2 = $.4 * $.6 * $.7 * $.9, 
        $.4^$.3 = $.4 * $.7 * $.8 * $.9^2, 
        $.5^$.1 = $.5 * $.9, 
        $.5^$.2 = $.5 * $.8^2, 
        $.5^$.3 = $.5 * $.8 * $.9^2, 
        $.5^$.4 = $.5 * $.8 * $.9, 
        $.6^$.2 = $.6 * $.8^2, 
        $.6^$.3 = $.6 * $.8^2 * $.9, 
        $.6^$.4 = $.6 * $.9^2, 
        $.7^$.1 = $.7 * $.8^2 * $.9^2, 
        $.7^$.2 = $.7 * $.8, 
        $.7^$.3 = $.7 * $.9^2
]
\end{Verbatim}

We create $S$ as before for this next claim. It verifies the calculation in \cref{E3inE1} and \cref{3Conjugate}, as well as claims made about $|\Aut(E_1)|_{3'}$ which are used throughout \cref{F3Sec}. We uses some of the commands from the Fusion Systems package.

\begin{Verbatim}[frame=single, framesep=4mm, framerule=0.1mm]
> S:=PCGroup(Sylow(G,3));
> E1:=Centralizer(S, UpperCentralSeries(S)[3]);
> X:=ElementaryAbelianSubgroups(E1: OrderEqual:=3^5);
> A:=[];
> for x in X do
for> if #Normalizer(E1, x`subgroup) eq 3^7 then
for|if> Append(~A, x`subgroup);
for|if> end if;
for> end for;
> #A;
3;
> a,b:=IsConjugate(S, A[1], A[2]);
> a;
true
> E1 eq NormalClosure(S, A[1]);
true
> Normalizer(E1, A[1]) eq NormalClosure(E1, A[1]);
true
> FrattiniSubgroup(E1) eq UpperCentralSeries(E1)[3];
true
> CommutatorSubgroup(FrattiniSubgroup(E1), E1) eq Center(E1);
true
> CommutatorSubgroup(FrattiniSubgroup(E1), S) eq Center(E1);
false
> AutE1:=AutomorphismGroup(E1);
> FactoredOrder(AutE1);
[ <2, 4>, <3, 15> ]
> InnE1:=AutYX(E1, E1);
> AutSE1:=AutYX(S, E1);
> phi, Aut:=PermutationRepresentation(AutE1);
> Inn:=sub<Aut| [phi(InnE1.i) : i in [1..#Generators(InnE1)]]>;
> AutS:=sub<Aut| [phi(AutSE1.i) : i in [1..#Generators(InnE1)]]>;
> X:=Subgroups(Aut: OrderEqual:=#Inn*#SL(2,3));
> B:=[];
> for H in X do
for> if Inn subset H`subgroup and 
for|if> IsIsomorphic(H`subgroup/Inn, SL(2,3)) then 
for|if> Append(~B, H);
for|if> end if;
for> end for;
> #B;
2
> for H in B do
for> if IsConjugate(Aut, AutS, Sylow(H`subgroup, 3)) then 
for|if> H, Index(Normalizer(Aut, H`subgroup), H`subgroup);
for|if> end if;
for> end for;
rec<recformat<order, length, subgroup, presentation> |
    order := 52488,
    length := 729,
    subgroup := Permutation group acting on a set of
    cardinality 1944
    Order = 104976 = 2^3 * 3^8>
6
\end{Verbatim}

\underline{\cref{F3SL2} and \cref{3Triv}:}

\begin{Verbatim}[frame=single, framesep=4mm, framerule=0.1mm]
> E2:=pCore(G,3);
> S:=Sylow(G,3);
> E1:=Centralizer(S, UpperCentralSeries(S)[3]);
> Center(E2) eq Center(S);
true
> ZE2:=UpperCentralSeries(E2)[3];
> FactoredOrder(ZE2);
[ <3, 3> ]
> UpperCentralSeries(S)[4] eq ZE2;
true
> CE2:=Centralizer(E2, ZE2);
> FactoredOrder(CE2);
[ <3, 7> ]
> ZWE2:=CommutatorSubgroup(E2, CE2);
> FactoredOrder(ZWE2);
[ <3, 4> ]
> ZE2 subset ZWE2;
true
> WE2:=Centralizer(E2, ZWE2);
> ZWE2 subset WE2;
true
> ZWE2 eq Centralizer(E2, WE2);
true
> FactoredOrder(WE2);
[ <3, 6> ]
> CommutatorSubgroup(CE2, WE2) eq ZE2;       
true
> FrattiniSubgroup(E1) subset WE2;
true
> ZWE2 subset FrattiniSubgroup(E1);
true
> for x in MaximalSubgroups(E2) do
for> if CommutatorSubgroup(x`subgroup, WE2) eq ZE2 then
for|if> x;
for|if> end if;
for> end for;
\end{Verbatim}

\underline{\cref{E3Unique}:}

\begin{Verbatim}[frame=single, framesep=4mm, framerule=0.1mm]
>S:=PCGroup(Sylow(G,3));
>E1:=Centralizer(S, UpperCentralSeries(S)[3]);
> A:=[];
> X:=Subgroups(GL(5,3):OrderEqual:=#SL(2,9));
> for P in X do
for> if IsIsomorphic(P`subgroup, SL(2,9)) then
for|if> Append(~A, P);
for|if> end if;
for> end for;
> #A;
1
> H:=A[1]`subgroup;
> T:=Sylow(H,3);
> K:=Normalizer(H, T);
> Normalizer(GL(5,3), K) subset Normalizer(GL(5,3), H);
true
\end{Verbatim}

\underline{\cref{ThmB2}:}

\begin{Verbatim}[frame=single, framesep=4mm, framerule=0.1mm]
> AutE1:=AutomorphismGroup(E1);
> AutSE1:=AutYX(S, E1);
> InnE1:=AutYX(E1, E1);
> phi, Aut:=PermutationRepresentation(AutE1);
> AutS:=sub<Aut| [phi(AutSE1.i): i in [1..#Generators(AutSE1)]]>;
> Inn:=sub<Aut| [phi(InnE1.i): i in [1..#Generators(InnE1)]]>;
> B:=Normalizer(Aut, AutS);
> R:=Sylow(B, 2);
> K:=sub<Aut| AutS, R>;
> Y:=MinimalOvergroups(Aut, K);
> Z:=[];
> for P in Y do
for> if IsIsomorphic(P/Inn, GL(2,3)) then
for|if> Append(~Z, P);
for|if> end if;
for> end for;
> #Z;
3
> #Sylow(Normalizer(Aut, K),3)/#Sylow(Normalizer(Aut, Z[1]), 3);
3
\end{Verbatim}

\underline{\cref{EssenDeterD3}:}
\begin{Verbatim}[frame=single, framesep=4mm, framerule=0.1mm]
> S:=PCGroup(Sylow(G,3));
> E1:=Centralizer(S, UpperCentralSeries(S)[3]);
> AllProtoEssentials(E1);
[
    GrpPC of order 243 = 3^5
    PC-Relations:
        $.1^3 = Id($), 
        $.2^3 = Id($), 
        $.3^3 = Id($), 
        $.4^3 = Id($), 
        $.5^3 = Id($)
]
\end{Verbatim}

\underline{\cref{MonsterEssen}:}
Once again, we appeal to the ATLAS of Finite Group Representations \cite{OnlineAtlas} to get generators, this time for a maximal subgroup of $\mathrm{M}$ (we take the $750$ point permutation representation of the maximal subgroup of shape $5^{2+2+4}:\Sym(3)\times \GL_2(5)$). We label the group generated by ``$G$" and proceed as follows:

\begin{Verbatim}[frame=single, framesep=4mm, framerule=0.1mm]
> S:=PCGroup(Sylow(G,5));
> AllProtoEssentials(S);
[
    GrpPC of order 625 = 5^4
    PC-Relations:
        $.1^5 = Id($), 
        $.2^5 = Id($), 
        $.3^5 = Id($), 
        $.4^5 = Id($),
            
    GrpPC of order 390625 = 5^8
    PC-Relations:
        $.2^$.1 = $.2 * $.5^2, 
        $.3^$.1 = $.3 * $.5^2 * $.6, 
        $.3^$.2 = $.3 * $.5^2 * $.7^3 * $.8^4, 
        $.4^$.1 = $.4 * $.5^2 * $.7^4, 
        $.4^$.2 = $.4 * $.6^4 * $.7^4 * $.8, 
        $.5^$.1 = $.5 * $.7^3, 
        $.5^$.2 = $.5 * $.7^2 * $.8^2, 
        $.5^$.3 = $.5 * $.8^3, 
        $.5^$.4 = $.5 * $.8^2, 
        $.6^$.1 = $.6 * $.7, 
        $.6^$.2 = $.6 * $.7^2 * $.8^2, 
        $.6^$.3 = $.6 * $.8, 
        $.6^$.4 = $.6 * $.8^2,
    
    GrpPC of order 390625 = 5^8
    PC-Relations:
        $.2^$.1 = $.2 * $.7^2 * $.8^2, 
        $.3^$.1 = $.3 * $.4^4 * $.5 * $.7^4 * $.8^2, 
        $.3^$.2 = $.4^3 * $.5 * $.6 * $.7^3, 
        $.4^$.1 = $.4^2 * $.5^4, 
        $.4^$.2 = $.3 * $.4^3 * $.5^4 * $.8, 
        $.4^$.3 = $.4 * $.6, 
        $.5^$.1 = $.4 * $.8, 
        $.5^$.2 = $.3 * $.4^2 * $.7^2 * $.8^4, 
        $.5^$.3 = $.5 * $.6 * $.7, 
        $.5^$.4 = $.5 * $.8, 
        $.6^$.1 = $.6 * $.7^4 * $.8^4, 
        $.6^$.2 = $.6 * $.7^4 * $.8^4, 
        $.7^$.2 = $.8^4, 
        $.8^$.2 = $.7 * $.8^2
]
\end{Verbatim}

For the next lemma, we need to calculate in $N_{M}(E_2)$ and calculate with the group $\mathbf{R}$. For this, we appeal again to the Online Atlas for generators of the maximal subgroup of shape $5^{3+3}:2\times \PSL_3(5)$ of $\mathrm{M}$ (we take the $7750$ point permutation representation). We label this group $G$ in what follows.

\underline{\cref{E3inE151}:}

\begin{Verbatim}[frame=single, framesep=4mm, framerule=0.1mm]
> M2:=MaximalSubgroups(G)[6]`subgroup; 
> E2:=pCore(M2,5);
> Center(E2);
Permutation group acting on a set of cardinality 7750
Order = 5
> S:=Sylow(M2,5);
> E1:=Centralizer(S, UpperCentralSeries(S)[3]);
> X:=ElementaryAbelianSubgroups(E1:OrderEqual:=5^4);
> Y:=[];
> for a in X do
for> if #Normalizer(S, a`subgroup) eq 5^6 and 
for> if> CommutatorSubgroup(Normalizer(S, a`subgroup),
for> if> a`subgroup) eq Center(E1) then
for|if> Append(~Y, a`subgroup);
for|if> end if;
for> end for;
> for a in Y do
for> if not(IsConjugate(Normalizer(M2, S), a, Y[1])) then
for|if> a;
for|if> end if; 
for> end for;
> E3:=Y[1];               
> E1 eq NormalClosure(S, E3);
true
> Index(Normalizer(M2, S), sub<M2| S, Normalizer(M2, E3)>);
4
\end{Verbatim}

\underline{\cref{MSL21}:}

\begin{Verbatim}[frame=single, framesep=4mm, framerule=0.1mm]
> R:=pCore(G, 5);                            
> FactoredOrder(FrattiniSubgroup(E2));
[ <5, 5> ]
> FrattiniSubgroup(R) eq 
> CommutatorSubgroup(E2, FrattiniSubgroup(E2));
true
> R eq Centralizer(S, FrattiniSubgroup(R));
true
\end{Verbatim}

\underline{\cref{E3inE152}:}

\begin{Verbatim}[frame=single, framesep=4mm, framerule=0.1mm]
> Center(E1) eq E3 meet FrattiniSubgroup(E1);
true
> FactoredOrder(FrattiniSubgroup(E1));
[ <5, 4> ]
> Normalizer(S, E3) eq sub<S| E3, FrattiniSubgroup(E1)>;
true
\end{Verbatim}

\underline{\cref{MSL2a}:}

\begin{Verbatim}[frame=single, framesep=4mm, framerule=0.1mm]
> IsElementaryAbelian(FrattiniSubgroup(E1));
true
> Center(E1) eq CommutatorSubgroup(S, FrattiniSubgroup(E1));
true
> FrattiniSubgroup(E1) eq Centralizer(S, FrattiniSubgroup(E1));
true
> Center(R) eq Center(FrattiniSubgroup(R));
true
\end{Verbatim}

We now switch back to of shape $5^{2+2+4}:\Sym(3)\times \GL_2(5)$ which we label $G$.

\underline{\cref{MSL2}:}

\begin{Verbatim}[frame=single, framesep=4mm, framerule=0.1mm]
> S:=PCGroup(Sylow(G,5));
> E1:=Centralizer(S, UpperCentralSeries(S)[3]);
> AutE1:=AutomorphismGroup(E1);
> FactoredOrder(AutE1);
[ <2, 6>, <3, 2>, <5, 13> ]
> InnE1:=AutYX(E1, E1);
> AutSE1:=AutYX(S, E1);
> phi, Aut:=PermutationRepresentation(AutE1);
> Inn:=sub<Aut| [phi(InnE1.i) : i in [1..#Generators(InnE1)]]>;
> AutS:=sub<Aut| [phi(AutSE1.i) : i in [1..#Generators(AutSE1)]]>;
> AutPhiE1:=NormalSubgroups(AutS:OrderEqual:=25)[1]`subgroup;
> N:=Normalizer(Aut, AutS); 
> C:=Centralizer(N, AutPhiE1);
> FactoredOrder(N);
[ <2, 5>, <3, 1>, <5, 9> ]
> FactoredOrder(C);
[ <2, 2>, <5, 9> ]
>  B:=sub<Aut| AutS, Subgroups(C:OrderEqual:=2^2)[1]`subgroup>; 
> FactoredOrder(B);
[ <2, 2>, <5, 7> ]
> X:=MinimalOvergroups(Aut, B);
> K:=[];
> for H in X do
for> if not(IsSolvable(H)) then
for|if> Append(~K, H);
for|if> end if;
for> end for;
> #K;
1
> FactoredOrder(Normalizer(Aut, K[1]));
[ <2, 6>, <3, 2>, <5, 7> ]
\end{Verbatim}

\underline{\cref{E3inE15b}:}

\begin{Verbatim}[frame=single, framesep=4mm, framerule=0.1mm]
> S:=Sylow(G, 5);
> E1:=Centralizer(S, UpperCentralSeries(S)[3]);
> X:=ElementaryAbelianSubgroups(E1:OrderEqual:=5^4);
> Y:=[];
> for a in X do
for> if #Normalizer(S, a`subgroup) eq 5^6 and
for> if> CommutatorSubgroup(Normalizer(S, a`subgroup),
for> if> a`subgroup) eq Center(E1) then
for|if> Append(~Y, a`subgroup);
for|if> end if;
for> end for;
> E3:=Y[1];    
> H:=NormalSubgroups(G)[#NormalSubgroups(G)-1]`subgroup;
> for a in Y do            
for> if not(IsConjugate(H, E3, a)) then
for|if> a;
for|if> end if;
for> end for;
> Index(Normalizer(G, E3), Normalizer(H, E3));
2
\end{Verbatim}

\underline{\cref{E1Unique5}:}
\begin{Verbatim}[frame=single, framesep=4mm, framerule=0.1mm]
> S:=PCGroup(Sylow(G,5));
> E1:=Centralizer(S, UpperCentralSeries(S)[3]);
> AutE1:=AutomorphismGroup(E1);
> AutSE1:=AutYX(S, E1);
> InnE1:=AutYX(E1, E1);
> phi, Aut:=PermutationRepresentation(AutE1);
> AutS:=sub<Aut| [phi(AutSE1.i) : i in [1..#Generators(AutSE1)]]>;
> Inn:=sub<Aut| [phi(InnE1.i) : i in [1..#Generators(InnE1)]]>;
> NAutS:=Normalizer(Aut, AutS);
> FactoredOrder(NAutS);
[ <2, 5>, <3, 1>, <5, 9> ]
> K:=Subgroups(NAutS:OrderEqual:=2^5*3)[1]`subgroup;
> K2:=Centralizer(K, Sylow(K,3));
> FactoredOrder(K2);
[ <2, 4>, <3, 1> ]
> X:=CyclicSubgroups(K2: OrderEqual:=12);
> A:=[];
> for P in X do
for> if #Centralizer(Aut, P`subgroup) mod 9 eq 0 then
for|if> Append(~A, P`subgroup);
for|if> end if;
for> end for;
> #A;
1
> L:=A[1];
> H:=Centralizer(Aut, L);
> IsIsomorphic(DirectProduct(Sylow(GL(2,5),3), GL(2,5)), H);
true
> T:=IrreducibleModules(H, GF(5));
> for V in T do
for> if Dimension(V) eq 2 and 
for> if> not(CohomologicalDimension(V, 1) eq 0) then
for|if> V;
for|if> end if;
for> end for;
\end{Verbatim}

\underline{\cref{E3Unique5}:}

\begin{Verbatim}[frame=single, framesep=4mm, framerule=0.1mm]
> X:=Subgroups(GL(4,5):OrderEqual:=#SL(2,25));
> A:=[];
> for H in X do                             
for> if IsIsomorphic(H`subgroup, SL(2,25)) then
for|if> Append(~A, H`subgroup);
for|if> end if; 
for> end for;
> #A;
1
> Y:=A[1];
> N:=Normalizer(Y, Sylow(Y,5));
> Normalizer(GL(4,5), N) subset Normalizer(GL(4,5), Y);
true
\end{Verbatim}

\underline{\cref{ThmC1}:}

It is difficult to construct $\Aut(\mathbf{Q})$ and by \cite{Winter}, we know that $\Out(\mathbf{Q})\cong \Sp_6(5):4$. Hence, we set $G$ to be the (unique up to conjugacy) maximal subgroup of $\GL_6(5)$.

\begin{Verbatim}[frame=single, framesep=4mm, framerule=0.1mm]
> AutQ:=MaximalSubgroups(G)[4]`subgroup;
> ChiefFactors(AutQ);
    G
    |  Cyclic(2)
    *
    |  J2
    *
    |  Cyclic(2)
    *
    |  Cyclic(2)
    1
> Y:=Subgroups(G: OrderEqual:=#AutQ);
> for H in Y do
for> if IsIsomorphic(H`subgroup, AutQ) and 
for> if> not(IsConjugate(G,H`subgroup, AutQ)) then
for|if> H;
for|if> end if; 
for> end for;
> S:=Sylow(AutQ, 5);
> Normalizer(G, Normalizer(AutQ, S)) subset Normalizer(G, S);
true
\end{Verbatim}

\underline{\cref{DStar5Pa}:}

\begin{Verbatim}[frame=single, framesep=4mm, framerule=0.1mm]
> a,G:=PermutationRepresentation(GL(2,5));
> #CyclicSubgroups(G:OrderEqual:=24);
1
> L:=CyclicSubgroups(G:OrderEqual:=24)[1]`subgroup;
> #IntermediateSubgroups(G, L);      
1
> K:=IntermediateSubgroups(G, L)[1];
> DStar:=DirectProduct(Sym(3), K);
> S:=Sylow(DStar,3);
> IsIsomorphic(Centralizer(DStar, S), 
> DirectProduct(CyclicGroup(3), CyclicGroup(24)));
true Mapping from: GrpPerm: $, Degree 33, Order 2^3 * 3^2 to 
GrpPerm: $, Degree 27, Order 2^3 * 3^2
Composition of Mapping from: GrpPerm: $, Degree 33, 
Order 2^3 * 3^2 to GrpPC and
Mapping from: GrpPC to GrpPC and
Mapping from: GrpPC to GrpPerm: $, Degree 27, Order 2^3 * 3^2
> X:=Subgroups(S:OrderEqual:=3);
> for x in X do
for> Index(DStar, Normalizer(DStar, x`subgroup));
for> end for; 
2
2
1
1
\end{Verbatim}

\underline{\cref{EssenDeterD5}:}

\begin{Verbatim}[frame=single, framesep=4mm, framerule=0.1mm]
> S:=PCGroup(Sylow(G,5));
> E1:=Centralizer(S, UpperCentralSeries(S)[3]);
> AllProtoEssentials(E1);
[
    GrpPC of order 625 = 5^4
    PC-Relations:
        $.1^5 = Id($), 
        $.2^5 = Id($), 
        $.3^5 = Id($), 
        $.4^5 = Id($)
]
\end{Verbatim}

\underline{\cref{UniqueCyclic} and \hyperlink{ThmE2}{Theorem E}:}

\begin{Verbatim}[frame=single, framesep=4mm, framerule=0.1mm]
> X:=CyclicSubgroups(DStar: OrderEqual:=24);
> for x in X do
for> Index(DStar, Normalizer(DStar, Sylow(x`subgroup,3)));
for> end for;
1
1
2
1
1
> L:=X[3]`subgroup;
> #IntermediateSubgroups(DStar, L);
7
\end{Verbatim}

\printbibliography
\end{document}